\documentclass{scrartcl}
\usepackage{amsmath, amsthm, amssymb}
\usepackage{mathrsfs}
\usepackage{scalerel}
\usepackage{url}
\usepackage{amssymb,amsfonts}
\usepackage{color}
\usepackage{shuffle}
\usepackage[usenames,dvipsnames]{xcolor}
\usepackage{graphicx}
\usepackage[font={small,it}]{caption}
\usepackage[linecolor=white,backgroundcolor=white,bordercolor=white,textsize=tiny]{todonotes}
\usepackage[utf8]{inputenc}
\usepackage{ytableau}
\usepackage{tikz-cd}
\usepackage{csquotes}

\usepackage{tikz}
\usetikzlibrary{calc,patterns}

\usepackage{mathtools} 

\usepackage{imakeidx}
\makeindex[intoc,title=Symbol index]

\usepackage{forest}

\forestset{
  root/.style = {minimum size = 0.7ex},
  decorated/.style = {
    for tree = {
      circle, fill, inner sep = 0ex, minimum size = 0.6ex,
      grow' = north, l = 0, l sep = 0.8ex, s sep = 0.5em,
      fit = tight, parent anchor = center, child anchor = center,
    }
  },
  default preamble = {decorated, root},
  begin draw/.code={\begin{tikzpicture}[baseline={([yshift=-0.5ex]current bounding box.center)}]},
}

\usepackage{silence}
\usepackage[top=3cm, bottom=3cm, left=2.5cm, right=3cm]{geometry}
\usepackage[colorlinks,backref=page]{hyperref}
\hypersetup{
pdfauthor={Joscha Diehl, Terry Lyons, Rosa Preiß, Jeremy Reizenstein},
pdftitle={Areas of areas generate the shuffle algebra},
pdfsubject={Rings and Algebras (math.RA); Probability (math.PR). MSC Classes: 17A30; 60H05; 93C15}
}
\usepackage{cleveref}

\let\underbar\underline

\newtheorem{theorem}{Theorem}[section]
\theoremstyle{plain}

\newtheorem{conjecture}[theorem]{Conjecture}
\newtheorem{corollary}[theorem]{Corollary}

\newtheorem{example}[theorem]{Example}

\newtheorem{lemma}[theorem]{Lemma}

\newtheorem{proposition}[theorem]{Proposition}
\newtheorem{remark}[theorem]{Remark}

\numberwithin{table}{section}
\numberwithin{equation}{section}

\theoremstyle{definition} 

\newtheorem{definition}[theorem]{Definition}

\newcommand{\R}{\mathbb{R}}

\newcommand{\E}{\mathbb{E}}

\let\todon\todo
\newcommand{\todonotes}[1]{\todon{\color{red}#1}}
\renewcommand{\todo}[1]{ \textbf{\color{Orange}(#1)} }

\newcommand{\spann}{\operatorname{span}}

\renewcommand{\Im}{\operatorname{Im}}

\newcommand{\ds}{d} 
\newcommand{\TC}{T((\R^\ds))} 
\newcommand{\TS}{T(\R^\ds)} 

\newcommand{\arealb}[1]{\overleftarrow{\area}(#1)}

\newcommand{\leftPrelieSym}{\leftPrelie_{\operatorname{Sym}}}
\newcommand{\lift}[1]{\underline{#1}}

\newcommand{\lincont}{\mathcal{L}}
\newcommand{\deshuffle}{\Delta_{\shuffle}}
\newcommand{\conc}{\mathbin{\cdot}}
\newcommand{\coeval}{\mathsf{coeval}}
\newcommand{\deconc}{\Delta_{\conc}}
\newcommand{\lin}{\mathrm{L}}
\newcommand{\mtimes}{\boxtimes}
\newcommand{\mshuffle}{\mathbin{\lift{\shuffle}}}
\newcommand{\mdeshuffle}{\Delta_{\mshuffle}}

\usepackage[shortlabels]{enumitem}
\setlist[enumerate]{label*=\arabic*.}
\DeclareMathOperator*{\binaryPlanarTrees}{\mathsf{BPT}}

\begin{document}

\def\word#1{{\color{blue}\mathtt{#1}}}
\let\alph\word

\newcommand{\area}{\mathsf{area}}
\newcommand{\Area}{\mathsf{Area}}
\newcommand{\ad}{\operatorname{ad}}
\newcommand{\Ad}{\operatorname{Ad}}

\newcommand{\grouplike}{G}
\newcommand{\primitive}{\mathfrak{g}}

\newcommand{\emptyWord}{e}

\newcommand{\evaluatedAt}[1]{\,\raisebox{-.5em}{$\vert_{#1}$}}
\newcommand{\id}{\operatorname{id}}

\newcommand{\eps}{\epsilon}

\newcommand{\leftPrelie}{\rhd}
\newcommand{\hs}{\mathbin{\succ}} 

\newcommand{\intnodes}{\operatorname{INT}}
\newcommand{\leaves}[1]{\lvert #1\rvert_{\scalebox{0.5}{$\mathsf{leaves}$}}}

\newcommand{\niceLikeExp}[1]{\qopname \relax o{#1}} 
\def\shuffleConcat{\mathbin{\scalerel*{\blacksquare}{o}}}
\def\shuffleConcatSymbol{{\scalerel*{\blacksquare}{o}}}
\newcommand{\im}{\operatorname{im}}

\newcommand\symmetricTensors{\Gamma}
\newcommand\symmetricAlgebra{\R}
\newcommand\gradedDual{{*\operatorname{gr}}}
\newcommand\TSgradedDual{{\TS^{\gradedDual}}}
\newcommand\anagrams{\mathsf{Anagrams}}
\newcommand{\gen}[2]{\left\langle #1;#2\right\rangle}
\newcommand{\quadsym}{\mathfrak{R}}
\newcommand{\quadsymlie}{\mathfrak{L}}
\newcommand{\quadprelie}{\mathfrak{P}}
\newcommand{\quaddend}{\mathfrak{D}}
\newcommand{\anagramdend}{\mathfrak{A}}

\newcommand\Rho{\mathfrak{r}}

\newcommand{\tortkara}{\mathscr{T}}
\newcommand{\areatortkara}{\mathscr{A}}
\newcommand{\tortop}{\mathsf{t}}
\newcommand{\tortjacobi}{\mathsf{J}}
\newcommand{\Inv}{\operatorname{Inv}}
\newcommand{\sign}{\mathsf{sign}}

\newcommand{\leftmultilie}{\boldsymbol{\lbrack}}
\newcommand{\rightmultilie}{\boldsymbol{\rbrack}_{\shuffleConcatSymbol}}
\newcommand\NEXTPAPER[1]{}
\renewcommand\todonotes[1]{}

\title{Areas of areas generate the shuffle algebra}
\author{Joscha Diehl\thanks{University of Greifswald}, Terry Lyons\thanks{Mathematical Institute, University of Oxford} \thanks{Alan Turing Institute, London}, Rosa Preiß\thanks{Institut für Mathematik, Technische Universität Berlin}, Jeremy Reizenstein\thanks{Facebook AI Research}}
\maketitle
\begin{abstract}%
  We consider the anti-symmetrization of the half-shuffle on words, which we call the `area' operator, since it corresponds to taking the signed area of elements of the iterated-integral signature.  The tensor algebra is a so-called Tortkara algebra under this operator.  We show that the iterated application of the area operator is sufficient to recover the iterated-integral signature of a path. Just as the ``information'' that the second level adds to the first one is known to be equivalent to the area between components of the path, this means that \emph{all} the information added by subsequent levels is equivalent to iterated areas.  On the way to this main result, we characterize (homogeneous) generating sets of the shuffle algebra. We finally discuss compatibility between the area operator and discrete integration and stochastic integration, and conclude with some results on the linear span of the areas of areas.
\end{abstract}

{\small\tableofcontents}
\section{Introduction}

\todonotes{

  - JR/RP: (improve?)  diagram of all the products

  - JR: symbol for foo on left and bar on right
  
  - RP: Precise Rocha citations for all Theorems and Definitions he already had in some form
}

We give a concise introduction here and spell out the notation more fully in the next section.
The shuffle algebra $\TS$ over $\ds$ letters is the vector space spanned by words in the letters $\word1,\dots,\word\ds$
with the commutative shuffle product. 
%
This is a free commutative algebra over the Lyndon words.
Put differently, it can be viewed a polynomial algebra in new commuting variables
$x_w$, where $w$ ranges over all Lyndon words. That is, as commutative algebras,
\begin{align*}
  \R[ x_w : \text{ $w$ is Lyndon} ]
  \cong
  \TS.
\end{align*}
The isomorphism is given by $  x_w \mapsto w \in \TS$.
There are many more (free) generators known: any basis for the Lie algebra, coordinates of the first kind, \dots (compare Corollary \ref{cor:someShuffleGenerators}).

The relevance for iterated integrals is as follows.
Let $X: [0,T] \to\R^d$ be a (piecewise smooth) curve and let $f_i \in \TS, i \in I$, be a generating set of the shuffle algebra.
Then: any term in the iterated-integrals signature $S(X)_{0,T}$ (\cite{bib:Che1954}, \cite[Chapter 2]{bib:Lyo2007}) is a polynomial in the real numbers
\begin{align*}
  \Big\langle f_i, S(X)_{0,T} \Big\rangle, i \in I.
\end{align*}
Indeed, by assumption, any word $w$ can be written as
\begin{align*}
  w = P_\shuffle( f_i : i \in I ),
\end{align*}
where $P_\shuffle$ is some shuffle polynomial in finitely many of the $f_i$.
By the shuffle identity we then get
\begin{align*}
  \Big\langle w, S(X)_{0,T} \Big\rangle
  &=
  \Big\langle P_\shuffle( f_i: i \in I ), S(X)_{0,T} \Big\rangle
  =
  P\left( \Big\langle f_i, S(X)_{0,T} \Big\rangle : i \in I \right),
\end{align*}
where $P$ is the corresponding polynomial expression in the real numbers $\langle f_i, S(X)_{0,T} \rangle$, $i\in I$.
The latter numbers then contain all the information
of the iterated-integrals signature, since every iterated integral
is a polynomial expression in them.

We are interested in whether there is a shuffle generating set in terms of ``areas of areas''.
Define the following bilinear operation on $\TS$\index{area@$\area$}
\begin{align*}
  \area(x,y) := x \hs y - y \hs x,
\end{align*}
where $\hs$ denotes the half-shuffle.
For $v,w \in \TS$, let 
\begin{align*}
  V_t &= \Big\langle v, S(X)_{0,t} \Big\rangle && \text{and} &
  W_t &= \Big\langle w, S(X)_{0,t} \Big\rangle,
\end{align*}
and define \index{Area@$\Area$}
\begin{align*}
  \Area(V,W)_t &:= \int_0^t \int_0^s dV_r dW_s - \int_0^t \int_0^s dW_r dV_s.
\shortintertext{We then have}
  \Area(V,W)_t &= \Big\langle \area(v,w), S(X)_{0,t} \Big\rangle.
\end{align*}
Our naming of $\area$ and $\Area$ stems from the fact that $\Area(V,W)$ is (two times)
the signed area (see \Cref{fig:area}) enclosed by the two-dimensional curve $(V,W)$, \cite{bib:LP2006}.
Note that the antisymmetrization $\Area(V,W)_t$ of the Riemann-Stieltjes integral (where the Riemann-Stieltjes integral forms a Zinbiel algebra on a suitable space of functions $V$ with $V_0=0$) has already been looked at as an algebraic operation by Rocha in 2003 in \cite[Equation (7)]{bib:Roc2003}, in \cite[Equation (6.11)]{bib:Roc2003b} and in 2005 in \cite[Equation (2.4)]{bib:RocNCSF}, it was even already noted by Rocha \cite[page 321]{bib:Roc2003}, \cite[page 3]{bib:RocNCSF} that the operation $\Area$ except being antisymmetric does not satisfy any additional identity of order three.

The following question is inspired by a remark made by T.L. during a talk in 2011:
\begin{center}
  \textit{Is repeated application of the $\Area$ operator enough to get the whole signature of a path $X$?}
\end{center}

For $d=2$ and the first two levels, this is quickly verified.
%
%
%
We start with the increments themselves, which we assume to be given (we think
of them as ``$0$-th order'' areas), which are   $\int dX^1$ and $\int dX^2$.
Then we can write, using integration-by-parts,
\begin{align*}
  \iint dX^1 dX^1 &= \frac12 \int dX^1 \cdot \int dX^1 \\
\iint dX^2 dX^2 &= \frac12 \int dX^2 \cdot \int dX^2 \\
\iint dX^1 dX^2 &= \frac12 \Big( \iint dX^1 dX^2 - \iint dX^2 dX^1 + \int dX^1 \cdot \int dX^2 \Big) \\
  \iint dX^2 dX^1 &= \frac12 \bigg( - \Big(\iint dX^1 dX^2 - \iint dX^2 dX^1\Big) + \int dX^1 \cdot \int dX^2 \bigg).
\end{align*}
and hence get all iterated integrals up to order $2$.

Products of integrals become, on the algebra side, $\shuffle$-products. This reads as
\begin{align*}
  \word{11} &= \frac12 \word1 \shuffle \word1 &
  \word{22} &= \frac12 \word2 \shuffle \word2 \\
  \word{12} &= \frac12 \left( \area(\word1,\word2) + \word1 \shuffle \word2 \right) &
  \word{21} &= \frac12 \left( - \area(\word1,\word2) + \word1 \shuffle \word 2 \right)
\end{align*}
In general, however, the expansion is non-unique, as the following example illustrates:
\begin{align*}
    \word{123}&=\tfrac13\area(\word1,\area(\word2,\word3))+\tfrac16\area(\area(\word1,\word3),\word2)
    +\tfrac13\word1\shuffle\area(\word2,\word3)
    \\&\quad-\tfrac16\word2\shuffle\area(\word1,\word3)+\tfrac12\word3\shuffle\area(\word1,\word2)
    +\tfrac16\word1\shuffle\word2\shuffle\word3\\
    &=\tfrac1{12}\area(\word1,\area(\word2,\word3))-\tfrac1{12}\area(\area(\word1,\word3),\word2)+\tfrac14\area(\area(\word1,\word2),\word3)
    \\&\quad+\tfrac1{12}\word1\shuffle\area(\word2,\word3)+\tfrac1{12}\word2\shuffle\area(\word1,\word3)+\tfrac14\word3\shuffle\area(\word1,\word2)
    +\tfrac16\word1\shuffle\word2\shuffle\word3
\end{align*}

To formulate the problem algebraically,
let $\areatortkara \subset \TS$ \index{A@$\areatortkara$} be the smallest linear space containing the letters $\word{1}, \dots, \word{d}$
that is closed under the (bilinear, non-associative) operation $\area$.
The question then becomes:
\begin{center}
  \textit{Is $\areatortkara$ a generating set for the shuffle algebra $T(\R^d)$?}
\end{center}
The affirmative answer to this question is given in this paper.

%
What we really have in mind here is a two-stage numerically-stable procedure for calculating the signature of a physical path.
In the first stage one calculates areas, areas of areas and so forth,
possibly using an \emph{analog physical apparatus}.%
\footnote{One physical device that has historically been used to measure area is a planimeter \cite{bib:FS2007}.
    In general, this is related to nonholonomic control, see for example \cite{bib:BD2015}.}
The second stage uses these measurements, say on a \emph{digital computer}, and computes polynomial expressions in these.

The rest of the paper is structured as follows.
In the next subsection we fix notation.
In Section \ref{sec:dynkinOperator}
we revisit results by Rocha \cite{bib:Roc2003} in purely algebraic terms.
The outcome of this is a formula for the Dynkin operator applied
to the signature. This make the areas operator appear naturally.
Together with Section \ref{sec:shuffleGenerators} this will
prove the generating property of areas-of-areas.

For completeness, we show in Section \ref{sec:coordinates} how to express coordinates 
of the first kind using only areas-of-areas.
Again, this is basically a purely algebraic revisiting of results by Rocha,
in which we also correct some of the expressions he gives.

In Section \ref{sec:shuffleGenerators}
we state a general condition for a set of polynomials
to be (free) generators of the shuffle algebra $\TS$.
%
We then show how a couple of well-known generators fall into this formulation
and, how using Section \ref{sec:dynkinOperator} (or \ref{sec:coordinates}),
the generating property of areas-of-areas is established in \Cref{cor:areasShuffleGenerate}.
The proof of that Corollary can also serve as a good roadmap for exploring the entire paper.

Apart from its geometric interpretation,
the area operation possesses some interesting properties.
Some of them we present in Section \ref{sec:applications},
where it is shown that it is nicely compatible with discrete integration
as well as stochastic integration.
In Section \ref{sec:linear} we collect some
results on the linear span generated by the area operator,
as it is of interest in its own right.


\medskip

\textbf{Acknowledgements.}

T.L. is supported by the EPSRC under the program grant EP/S026347/1 and by the Alan Turing Institute under the EPSRC grant EP/N510129/1.
R.P. is supported by European Research Council through CoG-683164 and was
affiliated to Max Planck Institute for Mathematics in the Sciences, Leipzig, in
autumn 2018.

T.L. and J.D. would like to thank Danyu Yang for many insightful discussions and
for realizing that it is natural to think of shuffle algebra as a space of
operators or sensors that take paths to scalar paths and the consequent insight
that it is completely natural to interpret the half-shuffles, and area etc. as
binary path operations, and so connect the structures of dendriform algebras
with the tensor/shuffle algebra structures in this way.

We would like to thank Cristopher Salvi for extended discussions, valuable presentations and
new insights on the area operator, in particular its Jacobian
bracketing (which helped shape the interpretation as signed volume) and the
interplay with the shuffle product, and on coordinates of the second kind.

Tensor algebra and tree computations for this project have been done in \texttt{python} and \texttt{sage}, where besides standard packages and further custom code by the authors the \texttt{python} package \nolinkurl{free_lie_algebra.py} \cite{bib:Rei2021}, which implements a lot of the definitions in \cite{bib:Reu1993}, has been of central use.

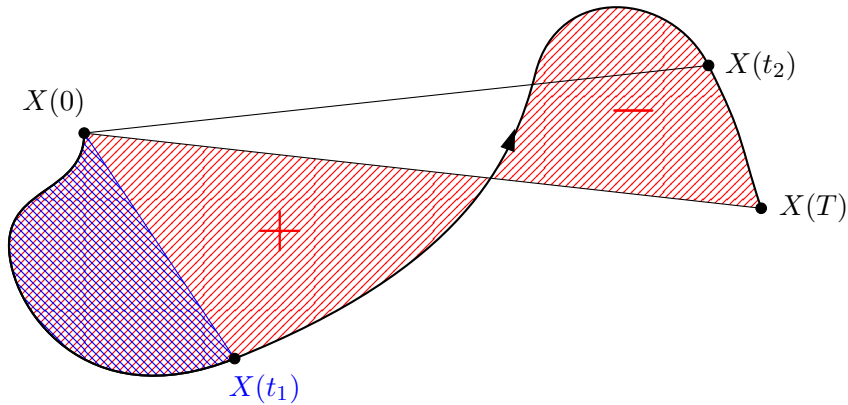
\begin{figure}[h]
  \centering
  \begin{tikzpicture}
    \coordinate (start) at (0,0);
    \coordinate (leftextreme) at (-1, -1.5);
    \coordinate (first) at (2,-3);
    \coordinate (mirrorsecond) at (6,0.8);
    \coordinate (second) at (8.3,0.9);
    \coordinate (end) at (9,-1);

    \draw [thick, pattern=north east lines, pattern color=red] (start) .. controls (0,-0.8) and ($(leftextreme)-(0,-0.8)$) .. (leftextreme) .. controls ($(leftextreme)+ (0,-0.8)$)
    and ($(first)-(2,0.8)$) .. (first) .. controls ($(first)+(3,1.2)$)
    and ($(mirrorsecond)-(0.3,1.2)$) .. (mirrorsecond) .. controls ($(mirrorsecond)+(0.3,1.2)$)
    and ($(second)-(0.5,-1)$) .. (second) .. controls ($(second)+(0.5, -1)$) and ($(end)-(0.3, -0.9)$) .. (end);

     \draw [thick, pattern=north west lines, pattern color=blue] (start) .. controls (0,-0.8) and ($(leftextreme)-(0,-0.8)$) .. (leftextreme) .. controls ($(leftextreme)+ (0,-0.8)$)
    and ($(first)-(2,0.8)$) .. (first);

    \draw (start) -- (end);
    \draw [blue] (start) -- (first);
    \draw (start) -- (second);

    \def\circ#1{\draw[fill]#1 circle (2pt);}
    \circ{(start)}
    \circ{(first)}
    \circ{(second)}
    \circ{(end)}
    \node at (-0.4,0.4) {$X(0)$};
    \node at ($(first)+(0.4,-0.4)$) {\color{blue}$X(t_1)$};
    \node at ($(second)+(0.7,0)$) {$X(t_2)$};
    \node at ($(end)+(0.7,0)$) {$X(T)$};

    \node at (2.6, -1.3) {\color{red}\Huge $+$};
    \node at (7.3, 0.3) {\color{red}\Huge $-$};

    \draw [shift={(5.67,-0.1)}, scale=0.3,rotate = 250] [fill] (0.4,-0.3) -- (-0.5,0) -- (0.4,0.3) -- cycle    ;
    
  \end{tikzpicture}
  \caption{The signed area of a curve $X$, shown at points $t=t_1$ (shaded blue) and at $t=T$ (shaded red).}
  \label{fig:area}
\end{figure}

\subsection{Notation}
\label{sec:notation}

\newcommand{\bracket}{\mathsf{lie}}

Denote by $\TC$ \index{TC@$\TC$} the space of formal \emph{infinite} linear combinations of words in the letters $\word{1}, .., \word{d}$.
Equip it with the \textbf{concatenation product} $\conc$ (often we write $b\conc b' = bb'$).

Denote by $\TS$ \index{TS@$\TS$} its dual, the space of \emph{finite} linear combinations of words.
Equip it with the \textbf{shuffle product} $\shuffle$. \index{sha@$\shuffle$} \index{<h@$\hs$}
It decomposes as
\begin{align*}
  a \shuffle a' = a \hs a' + a' \hs a,
\end{align*}
where $\hs$ is the \textbf{half-shuffle}.
The half-shuffle is defined on words $a = a_1 \dots a_m, b = b_1 \dots b_n$, where $b$ is not the empty word,
as
\begin{align*}
  a \hs b
  =
  (a \shuffle b_1 \dots b_{n-1}) \conc b_n.
\end{align*}

The dual pairing is written for $a \in \TS, b \in \TC$ as \index{<@$\Big\langle\cdot,\Big\rangle$}
\begin{align*}
  \Big\langle a, b \Big\rangle.
\end{align*}

Denote the grouplike elements of $\TC$ by $\grouplike$.\index{G@$\grouplike$}
Denote the primitive elements, or Lie elements, of $\TC$, i.e.~the free Lie algebra, by $\primitive$.\index{g@$\primitive$}

\newcommand\proj{\operatorname{proj}} \index{proj@$\proj_n$ and $\proj_{\ge n}$}
Denote by $\proj_n, \proj_{\ge n}, $ etc, the projection on $\TC$ to level $n$, to levels larger equal to $n$, \dots
We write $T_n((\R^d)) = \proj_n \TC, T_{\ge n}((\R^d)) = \proj_{\ge n} \TC, $ etc. \index{T@$T_n((\R^d))$ and $T_{\ge n}((\R^d))$}
Denote the empty word by $\emptyWord$.\index{ee@$\emptyWord$}

\newcommand{\TSTCRing}{\mathcal R\langle\langle\word1, \dots,\word d\rangle\rangle}
\newcommand{\TSTCRingFinite}{\mathcal R\langle\word1,\dots,\word d\rangle}
\newcommand\TSTC{\mathcal W}\index{TTT@$\TSTC$}

Denote by $\TSTCRingFinite$ 
the free tensor algebra over $d$ generators with coefficients in the ring $\mathcal R$, where $\mathcal R := (T(\R^d),\shuffle)$,
and by $\TSTCRing$ the corresponding space of tensor series.
We then canonically have $\TSTCRingFinite\subsetneq\TSTCRing$,
and identify the $\mathcal R$-algebra $\TSTCRing$ with
\begin{align}
  \label{eq:TSTC}
  \TSTC 
  :=\prod_{n=1}^\infty T(\R^d)\otimes T_n(\R^d),
\end{align}
where we use the shuffle product on the left and the concatenation product on the right.
We denote the product on both $\TSTCRing$ and $\TSTC$, which are isomorphic as $\R$-algebras,
by $\shuffleConcat$.\index{.@$\shuffleConcat$}
The $\mathcal{R}$-subalgebra $(\TSTCRingFinite,\shuffleConcat)$ is then $\R$-algebra-isomorphic to $(\TS\otimes\TS,\shuffleConcat)$.

We use the usual grading on $\TSTCRing$,
that is in the representation \eqref{eq:TSTC}, for $a,b$ words, $|a \otimes b| := |b|$.
Then, the projection $\proj_n$ makes also sense on $\TSTC$.


We furthermore introduce an $\mathcal{R}$-linear coproduct on $\TSTCRing$, which maps to the graded completion of the $\mathcal R$-module tensor product $\mtimes$:
\begin{equation*}
 \mdeshuffle:\,\TSTCRing\to\TSTCRing\mathbin{\hat\mtimes}\TSTCRing:=\prod_{m,n=1}^\infty\proj_m\TSTCRing\mtimes\proj_n\TSTCRing,
\end{equation*}
where the unshuffle coproduct on $\TSTCRing$ is defined via the usual unshuffle coproduct as
\begin{equation*}
 \mdeshuffle\Big(\sum_w a_w\lift{w}\Big):=\sum_w a_w\mdeshuffle \lift{w}:=\sum_w a_w\sum_{(w)}^\shuffle \lift{w_1}\mtimes\lift{w_2},
\end{equation*}
where the last Sweedler summation is well defined by the unshuffle coproduct on $\TS$ because there is a unique $\R$-linear map $\TS\otimes\TS\to\TSTCRing\hat{\mtimes}\TSTCRing$ characterized by sending each tensor pair of words $w\otimes v$ to $\lift{w}\mtimes\lift{v}$ (which is however non-surjective). We have the isomorphism
\begin{equation*}
 \prod_{m,n=1}^\infty \TS\otimes T_m(\R^d)\otimes T_n(\R^d)\cong\TSTCRing\mathbin{\hat\mtimes}\TSTCRing
\end{equation*}
as $\R$ vector spaces given by the map
\begin{equation*}
 \sum_{w,v} a_{w,v}\otimes w\otimes v\mapsto\sum_{w,v} a_{w,v}\,(\lift{w}\mtimes\lift{v})=\sum_{w,v}(a_{w,v}\lift{w})\mtimes\lift{v}=\sum_{w,v}\lift{w}\mtimes(a_{w,v}\lift{v}).
\end{equation*}
The unshuffle coproduct on $\TSTCRing$ is an $\mathcal{R}$-algebra homomorphism as a consequence of the homomorphism property of the usual unshuffle coproduct, as for words $w,v$ we have
\begin{align}\label{eq:mdeshuffle_homomorphism}
\mdeshuffle(p\lift{w}\shuffleConcat q\lift{v})
&=(p\shuffle q)\mdeshuffle\,\lift{w\conc v}
=(p\shuffle q) \sum_{(w),(v)}^\shuffle \lift{w_1\conc v_1}\mtimes\lift{w_2\conc v_2}\notag\\
&=\sum_{(w),(v)}^\shuffle(p\lift{w_1}\shuffleConcat q\lift{w_2})\mtimes(\lift{w_2}\shuffleConcat\lift{w_1})
=\Big(\sum_{(w)}^\shuffle p\lift{w_1}\mtimes\lift{w_2}\Big)\mathbin{\tilde{\shuffleConcat}}\Big(\sum_{(v)}^\shuffle q\lift{v_1}\mtimes\lift{v_2}\Big)\notag\\
&=(\mdeshuffle p\lift{w})\mathbin{\tilde{\shuffleConcat}}(\mdeshuffle q\lift{v}),
\end{align}
where
\begin{equation*}
 \Big(\sum_{w_1,v_1}a_{w_1,v_1}\lift{w_1}\mtimes\lift{v_1}\Big)\mathbin{\tilde{\shuffleConcat}}\Big(\sum_{w_2,v_2}b_{w_2,v_2}\lift{w_2}\mtimes\lift{v_2}\Big):=\sum_{w_1,v_1,w_2,v_2}(a_{w_1,v_1}\shuffle b_{w_2,v_2})(\lift{w_1}\shuffleConcat\lift{w_2})\mtimes(\lift{v_1}\shuffleConcat\lift{v_2})
\end{equation*}
is the usual induced product on the tensor product. 
When restricting to $\TSTCRingFinite$, we have $\mdeshuffle:\TSTCRingFinite\to\TSTCRingFinite\mtimes\TSTCRingFinite$
and the other compatibility relations of a Hopf algebra are checked along the same lines, so we indeed get an $\mathcal{R}$-Hopf algebra $(\TSTCRingFinite,\shuffleConcat,\mdeshuffle,\lift{\alpha})$, a Hopf algebra in the category of $\mathcal{R}$-modules, with antipode
\begin{equation*}
 \lift{\alpha}\Big(\sum_w a_w\lift{w}\Big)=\sum_w (-1)^{|w|}a_w\lift{\overleftarrow{w}},
\end{equation*}
where $\overleftarrow{w}$ is $w$ written backwards.

Now, since we have the homomorphism property of the unshuffle $\mdeshuffle$ on $\TSTCRing$ according to Equation \eqref{eq:mdeshuffle_homomorphism}, and furthermore
\begin{equation*}
 \mdeshuffle(\lift{\word{i}})=\lift{\emptyWord}\mtimes\lift{\word{i}}+\lift{\word{i}}\mtimes\lift{\emptyWord}
\end{equation*}
for any letter $\lift{\word{i}}$ in $\TSTCRing$, our $\mdeshuffle$ is exactly the coproduct from \cite{bib:Reu1993} for the choice of $K$ as the unital commutative ring $\mathcal{R}$ with characteristic zero. 
Thus, we may apply all the theory in Reutenauer's book valid for the general setting of a unital commutative ring of characteristic zero to $\TSTCRing$. In particular, we get that the group
\begin{align*}
 \lift{G}=\{g\in\TSTCRing|\mdeshuffle g=g\mtimes g,\, g\neq 0\}
\end{align*}
with product $\shuffleConcat$ \cite[Corollary 3.3]{bib:Reu1993} and the $\mathcal{R}$-Lie-algebra
\begin{align*}
 \lift{\mathfrak{g}}=\{x\in\TSTCRing|\mdeshuffle x=\lift{\emptyWord}\mtimes x+x\mtimes\lift{\emptyWord}\}
\end{align*}
with Lie bracket $[x,y]_{\shuffleConcatSymbol}:=x\shuffleConcat y-y\shuffleConcat x$ are in a one-to-one correspondence \cite[Theorems 3.1 and 3.2]{bib:Reu1993} via the exponential map
\cite[Equation (3.1.2)]{bib:Reu1993}
\begin{equation*}
 \exp_{\shuffleConcat}:\lift{\mathfrak{g}}\to\lift{G},
 \quad\exp_{\shuffleConcat}(x)
 =\lift{e}+\sum_{n=1}^{\infty}\frac{x^{\shuffleConcatSymbol n}}{n!}
\end{equation*}
with inverse the logarithm 
\cite[Equation (3.1.1)]{bib:Reu1993}
\begin{equation*}
 \log_{\shuffleConcat}:\lift{G}\to\lift{\mathfrak{g}},
 \quad\log_{\shuffleConcat}(g)
 =\sum_{n=1}^\infty(-1)^{n+1}\frac{(g-\lift{\emptyWord})^{\shuffleConcat n}}{n}.
\end{equation*}

Analogous to $G$ and $\mathfrak{g}$, we call the elements of $\lift{G}$ grouplike and the elements of $\lift{\mathfrak{g}}$ primitive.

Note however that $(\TSTCRingFinite,\shuffleConcat,\mdeshuffle)$ does not form a $\R$ Hopf algebra.



\newcommand{\eval}{\mathsf{eval}} 
Fixing $x \in \TC$, 
define for any $F=\sum_{w}a_w\otimes w\in\TSTC$, where $a_w\in\TS$ for all words $w$,
\begin{equation*}
 \eval_x(F):=\sum_{w}\langle x,a_w\rangle\ w \in\TC.
\end{equation*}
This operation now forms an associative algebra isomorphism from $(\TSTC,\shuffleConcat)$ to $(\lincont(\TC,\TC),\ast)$, where $\lincont(\TC,\TC)$ denotes the linear maps from $\TC$ to $\TC$ which are continuous in the product topology and $\ast$ denotes the convolution product of the Hopf algebra $(\TS,\conc,\deshuffle)$ extended to $\TC$. Indeed, for any $F,G\in\TSTC$, we have $\eval(F),\eval(G)\in\lincont(\TC,\TC)$ by definition with
\begin{equation*}
 \eval(F\shuffleConcat G)=\eval(F)\ast\eval(G),
\end{equation*}
since for $F=\sum_w a_w\otimes w$, $G=\sum_{w'} b_{w'}\otimes {w'}$, $a_w,b_w\in\TS$ for all words $w$, and $x\in\TC$, 
\begin{align*}
 \eval_x(F\shuffleConcat G)&=\sum_{w,w'}\langle x,a_w\shuffle b_{w'}\rangle\,w\conc w'=\sum_{w,w'}\langle\deshuffle x,a_w\otimes b_{w'}\rangle\,w\conc w'\\
 &=\sum_{(x)}^{\shuffle}\sum_{w}\langle x_1,a_w\rangle\,w\conc \sum_{w'}\langle x_2,b_{w'}\rangle w'=\sum_{(x)}^{\shuffle}\eval_{x_1}(F)\conc\eval_{x_2}(G)\\
 &=\mathsf{conc}(\eval(F)\otimes\eval(G))\deshuffle x=(\eval(F)\ast\eval(G))(x),
\end{align*}
where $\sum_{(x)}^\shuffle x_1\otimes x_2:=\deshuffle x$ is Sweedler's notation and $\mathsf{conc}:\,\TC\mathbin{\hat\otimes}\TC\to\TC$ is the continuous linear map corresponding to the bilinear map $\conc$.

Likewise, for arbitrary $y\in\TS$ and $F=\sum_{w}a_w\otimes w\in\TSTC$, $a_w\in \TS$, we define
\begin{equation*}
 \coeval^y(F):=\sum_{w}\langle w,y \rangle a_w\in T(\R^d).
\end{equation*}
Then, $\coeval$ forms an isomorphism from $(\TSTC,\shuffleConcat)$ to $(\lin(\TS,\TS),\star)$, where $\lin(\TS,\TS)$ denotes all linear maps from $\TS$ to itself and $\star$ is the convolution product of the Hopf algebra $(\TS,\shuffle,\deconc)$.

We refer to \cite{bib:Reu1993} for more details on all of this,
except for the half-shuffle, for which a nice entry point to the literature is for example \cite{bib:FP103}.

%
%
%
%
%

%
%
%
%
%
%

\subsection{Objectives}
\subsubsection{Revisiting the work of Rocha on coordinates of the first kind}
In Sections \ref{sec:dynkinOperator} and \ref{sec:coordinates}, we show how the area operation appears naturally in a purely algebraic formulation of the work of Rocha on coordinates of the first kind. What we may take over from Rocha here is a very interesting network of bilinear operations on $\TSTC$ refining the basic $\shuffleConcat$ product. It is based on a dendriform structure, as the following diagram (cf.\ \cite[diagram page 320]{bib:Roc2003}, \cite[Diagram 5]{bib:Roc2003b}) and definition show:

\begin{center}
\begin{tikzcd}
                                                                                                             &[2.5ex] \tau - \arrow[r]                              & \leftPrelie \arrow[rd, "(-)"'] \arrow[r, "(+)"]  & \leftPrelieSym                     & \text{{\color{gray}unknown}} \arrow[l, no head, dashed] \\[2ex]
\succeq \arrow[ru, no head] \arrow[rd, no head] \arrow[r, "\text{dendriform}" description, no head, dashed] & \preceq \arrow[u, no head] \arrow[d, no head] & \text{pre Lie} \arrow[u, no head, dashed] & {[\cdot,\cdot]_\shuffleConcatSymbol}                  & \text{Lie} \arrow[l, no head, dashed] \\[2ex]
                                                                                                             & + \arrow[r]                                   & \shuffleConcat \arrow[ru, "(-)"]            & \text{associative} \arrow[l, no head, dashed] &                                      
\end{tikzcd}
\end{center}

For $A=p\otimes q$, $B=p'\otimes q'$,
\begin{align*}
 A\succeq B&:=(p\hs p')\otimes(q\conc q'),\\
 A\preceq B&:=(p'\hs p)\otimes(q\conc q'),\\
 A\shuffleConcat B&:=A\succeq B+A\preceq B=(p\shuffle p')\otimes(q\conc q'),\\
 A\leftPrelie B&:=A\succeq B-B\preceq A=(p\hs p')\otimes[q,q'],\\
 A\leftPrelieSym B&:=A\leftPrelie B+B\leftPrelie A=A\succeq B+B\succeq A-A\preceq B-B\preceq A=\area(p,p')\otimes[q,q'],\\
 [A,B]_{\shuffleConcatSymbol}&:=A\shuffleConcat B-B\shuffleConcat A=A\leftPrelie B-B\leftPrelie A=A\succeq B+A\preceq B-B\succeq A-B\preceq A\\
 &=(p\shuffle p')\otimes[q,q'].
\end{align*}

In fact, one could describe this network for any dendriform structure, but Rocha's and our work offer a promising first usecase 
for talking about all of these operations together, 
while this system of operations \emph{without $\leftPrelieSym$} has been explored before e.g. in \cite{bib:KM2009}. 
The symmetrized pre-Lie operation $\leftPrelieSym$ stays the most mysterious also to us, 
we may only point to the discovery of Bergeron and Loday in \cite{bib:BL2011} that the symmetrization of pre-Lie does not in general satisfy any further identities except non-associative commutativity,
though since the pre-Lie product $\leftPrelie$ certainly isn't free we except some kind of relations for $\leftPrelieSym$ also,
but this is still a question of future work.

With the area operation forming the left part of the symmetrized pre-Lie operation $\leftPrelieSym$, we obtain our main argument to show that the set of all areas of areas forms a shuffle generating set, albeit not a minimal one.

\subsubsection{Areas of areas and further shuffle generating sets}

With our main focus being shuffle-generating sets in terms of areas of
areas, in Section \ref{sec:shuffleGenerators} we first give a general criterion
Lemma \ref{lem:lieAlgebra} for (homogeneous) subsets to form a shuffle generating
set (resp. a free shuffle generating set).
The condition being that the set contains (resp. forms) a dual basis to some basis of the free Lie algebra $\mathfrak{g}\subsetneq\TC$.
While our actual hands-on proof is based on the characterization of the annihilator of the free Lie algebra which we cite from \cite{bib:Reu1993},
we sketch a more abstract argument in Remark \ref{rem:milnorMoore} related to the Milnor-Moore theorem.

We continue by illustrating how our statement can be applied to some known shuffle generating sets, as well as to the image of $\rho$ (the dual of the Dynkin map, \Cref{sec:dynkinOperator}), which concludes one of our proofs that areas-of-areas generate the shuffle algebra.

\subsubsection{The area Tortkara algebra}
We study the smallest linear subspace $\areatortkara\subsetneq\TS$ closed under the area operation and containing the letters in Section \ref{sec:linear}. 
Thanks to the work by Dzhumadil'daev, Ismailov and Mashurov, we can use the categorial framework of \emph{Tortkara algebras}, 
where the objects are characterized as vector spaces with a bilinear antisymmetric operation which furthermore satisfy the \emph{Tortkara identity}, and the morphisms are homomorphisms of the bilinear operations as usual. 
Also thanks to \cite{bib:DIM2018}, we have a simple linear basis of $\areatortkara$ in terms of linear combinations of words, see Lemma \ref{lem:AequalsA}. 
We continue by a very important conjecture that the left bracketings of the area operation yield another basis of $\areatortkara$, which was shown for dimension two in both \cite{bib:DIM2018} and \cite{bib:Rei2018}. 
The rest of the section is dedicated to some interesting observations we made while, so far unsuccessfully, trying to prove that conjecture for any dimension.

\subsubsection{Applications and characterizations}
In Section \ref{sec:applications} we are connecting the purely algebraic considerations of this paper with the world of (deterministic and probabilistic) path spaces and iterated-integrals signatures on these path spaces,
as they have been the motivation for this work to begin with.
What we are generally looking at are characterizations of the area Tortkara algebra $\areatortkara$ in terms of special properties for given path spaces, like the space of piecewise linear paths (Subsection \ref{sec:computational}). 
The case of piecewise linear paths promises in fact to develop into the main application of the study of areas of areas.
They form the most common discretization of general continuous paths that one works with when actually computing iterated-integrals signature numerically, in machine learning for example.
It turns out that for piecewise linear paths, the computation of discrete areas is much simpler and better behaved than the computation of discrete integrals.

However, besides the discrete deterministic setting, the study of signatures has, since Lyons' theory of rough paths, been intimately related with stochastic analysis,
and we observe how areas of areas preserve the martingale property central in stochastic analysis, while general iterated Stratonovich integrals fail to do so.

\section{The Dynkin operator}
\label{sec:dynkinOperator}
\todonotes{JR: Friendly intro. In this we introduce, double tensor}

We recall the linear maps $r,D: \TC \to \TC$ from \cite[Section 1, p.20]{bib:Reu1993}. \index{r@$r$}\index{D@$D$}
The linear right-bracketing map or \textbf{Dynkin operator} $r$ is given on a word $w = \word{l}_1 \dots \word{l}_n$ as
\begin{align}
  \label{eq:r}
  r(\word{l}_1 \dots \word{l}_n) \coloneqq [\word{l}_1,[\word{l}_2,...[\word{l}_{n-1},\word{l}_n]]],
\end{align}
with $r(\emptyWord)=0$ and $r(\word{i})=\word{i}$ for any letter $\word{i}$.
The map $D$ (for \emph{derivation}) is given on a word $w$ as
\begin{align*}
  D(w) \coloneqq |w| w,
\end{align*}
where $|w|$ is the length of the word.
On $T_{\ge 1}((\R^d))$, $D$ is invertible with inverse $D^{-1}(w) = \frac{1}{|w|} w$.\index{Dm@$D^{-1}$}

\begin{remark}
  1. The seemingly simple Dynkin operator $r$ has found several applications.
  It for example characterizes Lie elements of $\TC$
  \cite[Theorem 3.1 (vi)]{bib:Reu1993}:
  $x \in \TC$ is a Lie series if and only if $\langle \emptyWord, x \rangle = 0$
  and $r(x) = D(x)$.
  See also \cite{bib:PR2002}, \cite{bib:Gar1990} and references therein.
  In the backward error analysis of numerical schemes it is used for example in \cite{bib:LMK2013}.

  2.
  \label{rem:logarithmicDerivative}
  Truncated at a fixed level, the grouplike elements / signatures of tree-reduced paths, form a Lie group.
  The Dynkin operator $r$ is a logarithmic derivative, i.e.~the derivative pulled-back to the tangent space at the identity, of an endomorphism of this group
  in the following sense (see \cite{bib:MP2013} for more on this).
  Let $\delta_\eps$ be the dilation operator, i.e.~the operation on tensors which multiplies each level $m$ by $\epsilon^m$,
  which corresponds to dilating or scaling a path by the factor $\epsilon$.
  For $g \in \grouplike$, let $g^\eps := \delta_\eps g$.
  Then
   \begin{align*}
    \big(\frac{d}{d\eps} g^\eps\big) \conc (g^\eps)^{-1}
    &=
    \big(\frac{d}{d\eps} g^\eps\big) \conc \alpha[g^\eps]
    =
    \big(\tfrac{1}{\eps} D[ g^\eps ]\big)\conc\alpha[g^\eps]
    =
    \tfrac{1}{\eps} \big(\mathsf{conc}\circ (D \otimes \alpha)\big) [g^\eps \otimes g^\eps] \\
    &=
    \tfrac{1}{\eps} \big(\mathsf{conc}\circ (D \otimes \alpha) \circ\deshuffle\big) [g^\eps] \qquad\text{as $g^\eps\in G$, \cite[Theorem 3.2]{bib:Reu1993}}\\
    &=
    \tfrac{1}{\eps} r[g^\eps],\qquad\text{see \cite[p32 or Lemma 1.5]{bib:Reu1993}}
   \end{align*}
   where $\alpha$ is the antipode on $\TC$ (which is the inverse in the Lie group, and corresponds to reversing a path),
   $\otimes$ is the external tensor product,
   $\mathsf{conc}$ is the linear map taking $a\otimes b$ to $a\conc b$,
  and $\deshuffle$ is the unshuffle coproduct, which \cite{bib:Reu1993} denotes with $\delta$.
\end{remark}

Let $\lift{r}$, $\lift{D}$, $\lift{D}^{-1}$ act on $\TSTC$
by letting $r,D,D^{-1}$ act on the right side of the tensor, i.e.
\index{rl@$\lift{r}$}
\index{Dl@$\lift{D}$}
\index{Dml@$\lift{D}^{-1}$}
\begin{align*}
  \lift{r}( a \otimes b ) &\coloneqq a \otimes {r}(b) \\
  \lift{D}( a \otimes b ) &\coloneqq a \otimes {D}(b) \\
  \lift{D}^{-1}( a \otimes b ) &\coloneqq a \otimes {D^{-1}}(b).
\end{align*}

Define%
\footnote{From now on, if we sum over a variable with no given index set,
we sum over all words in the alphabet $\word{1},..,\word{d}$, including the empty word $\emptyWord$.}
\index{S@$S$}
\index{R@$R$}
\index{rho@$\rho$}
\begin{align}
  S &\coloneqq \sum_w w \otimes w \notag \\
  R &\coloneqq \lift{r}(S) = \sum_w w \otimes r(w) = \sum_v \rho(v) \otimes v.  \label{eq:R}
\end{align}
Both are elements of $\TSTC$.
The last equality implicitly defines $\rho$. There also exists a recursive definition given by $\rho(\emptyWord)=0$, $\rho(\word{i})=\word{i}$ for any letter $\word{i}$ and
\begin{equation}\label{eq:rhorecursive}
 \rho(\word{i}w\word{j})=\word{i}\rho(w\word{j})-\word{j}\rho(\word{i}w)
\end{equation}
for any (empty or non-empty) word $w$ and letters $\word{i},\word{j}$,
see \cite[p.32]{bib:Reu1993}. Based on this recursion, we derive an expansion of $\rho$ via an action of elements of the symmetric group algebra in Proposition \ref{prop:rhogroupalgebra}.
We repeat that $r(\emptyWord) = \rho(\emptyWord) = 0$, so the sum in \eqref{eq:R} is actually only taken over words of strictly positive length.

We record the following for future use (\cite[Theorem 1.12]{bib:Reu1993}).
For any word $w$\footnote{If $|w| = 0$ then both sides are equal to zero.}
\begin{align}
  \label{eq:thm112}
  D w = \sum_{uv = w} \rho(u) \shuffle v = \sum_{uv = w, |u|\ge 1} \rho(u) \shuffle v.
\end{align}

Note that this yields yet another recursive definition of $\rho$:
\begin{equation*}
 \rho(\emptyWord)=0,
 \quad\rho(w)=|w|w-\sum_{\substack{uv=w\\[0.5ex]|u|,|v|\geq 1}}\rho(u)\shuffle v,
\end{equation*}
where $w$ is an arbitrary non-empty word.

\begin{proposition}
  \label{prop:rIsInvertible}
  The map $r: \grouplike \to \primitive$ is invertible. 
  To be specific, define for $x \in T_{\ge 1}((\R^d))$ the linear map 
  \begin{align*}
    A_x: \TC &\to \TC \\
    z &\mapsto D^{-1}( x z ).
  \end{align*}
  Then for $x \in \primitive$
  \begin{align}
    \label{eq:rInverse}
    r^{-1}[ x ]
    &= \sum_{\ell \ge 0} A_x^\ell \emptyWord \\
    &= \emptyWord + D^{-1}(x) + D^{-1}(x D^{-1}(x)) + D^{-1}(x D^{-1}(x D^{-1}(x))) + .. \notag
  \end{align}

  Equivalently, with $\lift{A}_R z := \lift{D}^{-1}[ R \shuffleConcat z ]$,\NEXTPAPER{Rosa: Is this a general formula for $\lift{r}^{-1}$?}
  \begin{align}
    \label{eq:equivalent}
    S
    &= \sum_{\ell\ge0} (\lift{A}_R)^\ell (\emptyWord\otimes \emptyWord) \\
    &= \emptyWord\otimes\emptyWord
        + \lift{D}^{-1}[ R ]
        + \lift{D}^{-1}[ R \shuffleConcat \lift{D}^{-1}[ R ] ]
        + \lift{D}^{-1}[ R \shuffleConcat \lift{D}^{-1}[ R \shuffleConcat \lift{D}^{-1}[ R ] ] ]
        + ..., \notag
  \end{align}
\end{proposition}

\begin{remark}
  Compare \cite[Theorem 4.1]{bib:EGP2007} for a statement in a more general setting.
\end{remark}

\begin{proof}[Proof of \Cref{prop:rIsInvertible}]
  The claimed equivalence is shown as follows.
  For $t \in \TSTC$, with zero coefficient for $\emptyWord\otimes\emptyWord$,
  \begin{align*}
    \eval_g(\lift{D}^{-1}( t ))
    =
    D^{-1}( \eval_g(t) ).
  \end{align*}
  Hence
  \begin{align*}
    g &= \emptyWord + D^{-1}( r(g) ) + D^{-1}( r(g) D^{-1}( r(g) ) ) + .. \quad \forall g \in G \\
      &\Leftrightarrow \\
    \eval_g(S)
    &=
    \eval_g(\emptyWord) + D^{-1}( \eval_g(R) ) + D^{-1}( \eval_g(R) D^{-1}( \eval_g(R) ) ) + .. \quad \forall g \in G \\
    &=
    \eval_g(\emptyWord) + \eval_g( \lift{D}^{-1}( R ) ) + \eval_g( \lift{D}^{-1}( R \shuffleConcat \lift{D}^{-1}( R ) ) ) + .. \\
          &\Leftrightarrow \\
    S
    &=
    \emptyWord + \lift{D}^{-1}( R ) + \lift{D}^{-1}( R \shuffleConcat \lift{D}^{-1}( R ) ) + ..,
  \end{align*}
  where we used the homomorphism property of $\eval_g$ and the fact that grouplike elements linearly span $\TC$ projectively
  (i.e. truncated, at level $n$, grouplike elements linearly span $T_{\le n}((\R^d))$).

    We now show \eqref{eq:rInverse}.
    Write $x := r[g] = D[g] \conc g^{-1}$ (compare Remark \ref{rem:logarithmicDerivative}.2).
    Then
    \begin{align*}
      g = \emptyWord + D^{-1}[ x \conc g ],
    \end{align*}
    i.e.
    \begin{align}\label{eq:gAxfixpoint}
      g = \emptyWord + A_x g.
    \end{align}
    Now since $x$ does not contain a component in the empty word, this actually amounts to a recursive formula,
    \begin{equation*}
     \proj_0 g=\emptyWord,\quad \proj_n g=\sum_{m=1}^n A_{\proj_m x}(\proj_{n-m}g),\quad n\geq 1
    \end{equation*}
    and thus Equation \eqref{eq:gAxfixpoint} has a unique solution. Hence
    \begin{align*}
      g = \sum_{\ell\ge 0} A^\ell_x \emptyWord,
    \end{align*}
    since the series converges due to being a finite sum for each homogeneous component and obviously provides a solution for Equation \eqref{eq:gAxfixpoint}.
    
    This shows that \eqref{eq:rInverse} gives a left-inverse.

    It is also a right inverse. Indeed,
    first note that
    for $x \in \primitive$
    and $n \ge 2$ we have $r[ A_x^n \emptyWord ] = 0$.
    For $n=2$, 
    using Lemma \ref{lem:PQ}, this follows from
    \begin{align*}
      r[ A_x^2 \emptyWord ]
      &=
      r[ D^{-1}( x D^{-1}(x) ) ]
      =
      D^{-1}( r[ x D^{-1}(x) ] )
      =
      D^{-1}( r[ r[D^{-1}x] D^{-1}(x) ] ) \\
      &=
      D^{-1}( [ r[D^{-1}x], r[D^{-1}(x)] ] )
      = 0.
    \end{align*}
    Assume it is true for $A^{n-1}_x$, then
    \begin{align*}
      r[ A_x^n \emptyWord ]
      &=
      r[ D^{-1}( x A_x^{n-1} \emptyWord ) ]
      =
      D^{-1}( r[ x A_x^{n-1} \emptyWord ] )
      =
      D^{-1}( r[ r[D^{-1}x] A_x^{n-1} \emptyWord ] ) \\
      &=
      D^{-1}( [ r[D^{-1}x], r[ A_x^{n-1} \emptyWord ] ] )
      = 0.
    \end{align*}
    Hence
    \begin{align*}
      r[ e + D^{-1}(x) + D^{-1}(xD^{-1}(x)) + .. ]
      = x,
    \end{align*}
    so that the Lemma indeed provides a right inverse.
\end{proof}

%
%


\begin{definition}
  \label{def:preLie}
  Define the following product on $\TSTC$,
  \index{<l@$\leftPrelie$}
  \begin{align*}
    (p \otimes q) \leftPrelie (p' \otimes q') := (p \hs p') \otimes [q,q'],
  \end{align*}
  where $[.,.]$ is the Lie bracket in $\TC$ and $\hs$ is the half-shuffle in $\TS$.
\end{definition}
\begin{remark}
  \label{rem:preLie}
  This product is \emph{pre-Lie}, 
  as the tensor product of a Zinbiel algebra and a Lie algebra is always a pre-Lie algebra (this is shown in Rocha's thesis as \cite[Proposition~4.13 and Corollary~4.14]{bib:Roc2003b}, 
  though there the terminology `chronological algebra' is used to mean what we call pre-Lie algebra), although we will not use this fact.
  It comes from the dendriform structure
  \begin{align*}
    (p \otimes q) \succeq (p' \otimes q') &:= (p \hs p') \otimes q q' \\
    (p \otimes q) \preceq (p' \otimes q') &:= (p' \hs p) \otimes q q',
  \end{align*}
  i.e. $x \leftPrelie y = x \succeq y - y \preceq x$.
  Indeed, the operations $\succeq$ and $\preceq$ together satisfy the three dendriform identities (e.g. \cite[Equations (8)-(10)]{bib:KM2009}), which is a straightforward consequence of the Zinbiel identity of the halfshuffle and the associativity of the concatenation,
  \begin{align*}
   (A\preceq B)\preceq C&=(p_3\hs (p_2\hs p_1))\otimes q_1  q_2  q_3=((p_3\hs p_2)\hs p_1)\otimes q_1  q_2  q_3+((p_2\hs p_3)\hs p_1)\otimes q_1  q_2  q_3\\
   &=A\preceq (B\preceq C)+A\preceq (C\preceq B),\\
   A\succeq (B\succeq C)&=(p_1\hs (p_2\hs p_3))\otimes q_1  q_2  q_3=((p_1\hs p_2)\hs p_3)\otimes q_1  q_2  q_3+((p_2\hs p_1)\hs p_3)\otimes q_1  q_2  q_3\\
   &=(A\succeq B)\succeq C+(B\succeq A)\succeq C,\\
   (A\succeq B)\preceq C&=(p_3\hs(p_1\hs p_2))\otimes q_1 q_2 q_3=((p_3\hs p_1+p_1\hs p_3)\hs p_2)\otimes q_1 q_2 q_3\\
   &=(p_1\hs(p_3\hs p_2))\otimes q_1 q_2 q_3=A\succeq(B\preceq C),
  \end{align*}
  for $A=p_1\otimes q_1$, $B=p_2\otimes q_2$, $C=p_2\otimes q_3$.

  For more background on pre-Lie products and this relation to
  dendriform algebras see for example \cite{bib:KM2009} and references therein.
\end{remark}

%


The object $R$ satisfies a quadratic fixed-point equation.
\begin{lemma}
  \label{lem:R}
  \begin{align}
    \label{eq:lemR}
    (\lift{D} - \id) R = R \leftPrelie R.
  \end{align}
\end{lemma}

\begin{proof}
  Let $|w| \ge 1$.
  Starting from \eqref{eq:thm112} and concatenating a letter $\word{a}$ from the right on both sides,
  we get
  \begin{align*}
    \sum_{uv = w, |u|\ge 1} (\rho(u) \shuffle v) \word{a} = (D w) \word{a} = (D-\id)( w\word{a} ).
  \end{align*}
  Hence
  \begin{align*}
    \sum_{uv = w, |u|\ge 1} \rho(u) \hs v\word{a} = (D-\id)( w\word{a} ),
  \end{align*}
  which means, for $|\bar w| \ge 2$,
  \begin{align}
    \label{eq:halfShuffle}
    \sum_{uv = \bar w, |u|\ge 1,|v|\ge 1} \rho(u) \hs v = (D-\id) \bar w.
  \end{align}

  Recall \index{ad@$\ad_v$ and $\Ad_v$}
  \begin{align*}
    \ad_v w &= [v, w] \\
    \Ad_v w &= [\word{k}_1, [\word{k}_2, .., [\word{k}_n, w]..]],
  \end{align*}
  where $v=\word{k}_1\cdots\word{k}_n$.  By \cite[Theorem 1.4]{bib:Reu1993}, for a Lie polynomial $P$ one has
  \begin{align}
    \label{eq:ad}
    \ad_P = \Ad_P.
  \end{align}
  For a word $w$ define the linear map $I_w$ as \index{Iw@$I_w$}
  \begin{align*}
    I_w x \coloneqq w \hs x,
  \end{align*}
  and extend linearly to the whole tensor algebra.
  %
  The map
  \begin{align*}
    I_\bullet \otimes \ad_\bullet: \TSTC \to \operatorname{Hom}_\R\big( \TSTC, \TSTC\big),
  \end{align*}
  is defined as
  \begin{align*}
    (I_x\otimes\ad_y) a \otimes b
    =
    (I_x a) \otimes (\ad_y b).
  \end{align*}
  Now
  \begin{align*}
    (I_\bullet \otimes \ad_\bullet) R
    &=
    (I_\bullet \otimes \ad_\bullet) \sum w \otimes r(w)
    =
    (I_\bullet \otimes \Ad_\bullet) \sum w \otimes r(w)
    =
    (I_\bullet \otimes \Ad_\bullet) \sum \rho(v) \otimes v \\
    &=
    \sum_{|v|\ge 1} I_{\rho(v)} \otimes \Ad_v,
  \end{align*}
  where we used \eqref{eq:ad} and then \eqref{eq:R}.
  Then
  \begin{align*}
    R \leftPrelie R
    &=
    ((I_\bullet\otimes\ad_\bullet ) R) R
    =
    \left(\sum_{|v|\ge 1} I_{\rho(v)} \otimes \Ad_v\right) \sum_{|w|\ge 1} w \otimes r(w)
    =
    \sum_{|v|,|w|\ge 1} (\rho(v) \hs w) \otimes r(vw) \\
    &=
    \sum_{|x|\ge2} \sum_{vw = x, |v|,|w|\ge 1} (\rho(v) \hs w) \otimes r(vw)
    =
    \sum_{|x|\ge2} (|x| - 1) x \otimes r(x) \\
    &=
    \sum_{|x|\ge2} x \otimes \big( (|x| - 1) r(x) \big)
    =
    \sum_{|x|\ge2} x \otimes r[ (D-\id) x ]
    =
    (\lift D - \id) R.\qedhere
  \end{align*}
\end{proof}

\begin{remark}
  \label{rem:ODEInterpretation}
  We sketch the connection to the ODE approach of \cite{bib:Roc2003}.  
  Let $S^\eps_t := \delta_\eps S(X)_t$ be the signature at time $t$, dilated by a factor $\eps > 0$.
  Define
  \begin{align*}
    Z^\eps_t := \frac{d}{d\eps} S^\eps_t \conc (S^\eps_t)^{-1},
  \end{align*}
  which, as we have seen in Remark \ref{rem:logarithmicDerivative}.2, is equal to $\eps^{-1} r[ S^\eps_t ]$.
  One can show (see \cite[(1.8)]{bib:AGS1989}, in the language of 'chronological algebras'),
  that $Z^\eps_t$ satisfies
  \begin{align}\label{eq:partial_Z}
    \partial_\eps Z^\eps_t = \int_0^t \left[ Z^\eps_r, \dot Z^\eps_r \right] dr,
  \end{align}
  where $\left[.,.\right]$ is the Lie bracket in $\primitive$.
  We may give an alternative proof of \eqref{eq:partial_Z} based on the quadratic fixed-point equation \eqref{eq:lemR}. 
  For the left hand side,
  \begin{align*}
    \partial_\eps Z^\eps_t
    &= \partial_\eps \left( \eps^{-1} r[ S^\eps_t ] \right)
    = - \eps^{-2} r[ S^\eps_t ] + \eps^{-1} r[ \partial_\eps S^\eps_t ]
    = - \eps^{-2} r[ S^\eps_t ] + \eps^{-2} r[ D S^\eps_t ]
    = \eps^{-2} r[ (D-\id) S^\eps_t ]\\
    &=\eps^{-2}\sum_{w}\langle S_t^{\eps},w\rangle\,r[(D-\id)w]
    =\eps^{-2}\eval_{S_t^\eps}[(\lift{D}-\id)R],
  \end{align*}
  %
  Aiming at the right hand side, we first note that in general for $p\otimes q, p'\otimes q' \in \TSTC$, we have
  \begin{align*}
    \int_0^t \left[ \Big\langle S_s, p \Big\rangle q, \Big\langle \dot S_s, p' \Big\rangle q' \right] ds
    &=
    \int_0^t \Big\langle S_s, p \Big\rangle d \Big\langle S_s, p' \Big\rangle \left[ q, q' \right]
    =
    \Big\langle S_t, p \hs p' \Big\rangle \left[ q, q' \right] \\
    &=
    \eval_{S_t}\left[ (p\otimes q) \leftPrelie (p'\otimes q') \right].
  \end{align*}
  This implies
  \begin{align*}
   \int_0^t \left[ Z^\eps_r, \dot Z^\eps_r \right] dr&=\eps^{-2}\left[r[S_t^\eps],r[\dot S_t^\eps]\right]
   =\eps^{-2}\sum_{w}\sum_{w'}\int_0^t \left[ \Big\langle S_s^\eps, w \Big\rangle r[w], \Big\langle \dot S_s^\eps, w' \Big\rangle r[w'] \right] ds\\
   &=\eps^{-2}\sum_{w}\sum_{w'}\eval_{S_t^\eps}[(w\otimes r[w])\leftPrelie (w'\otimes r[w'])]
   =\eps^{-2}\eval_{S_t^\eps}[R\leftPrelie R].
  \end{align*}

  Putting things together, we thus have that \eqref{eq:partial_Z} is equivalent to
  \begin{equation*}
   \eval_{S_t^\eps}[(\lift{D}-\id)R]=\eval_{S_t^\eps}[R\leftPrelie R].
  \end{equation*}
  which is of course an immediate consequence of \eqref{eq:lemR}.
\end{remark}

By symmetrizing the pre-Lie product in the quadratic fixed point equation \eqref{eq:lemR}, we make the $\area$-operator appear.
  Define \index{<S@$\leftPrelieSym$}
  \begin{align*}
    (p \otimes q) \leftPrelieSym (p' \otimes q')
    &:=
    (p \otimes q) \leftPrelie (p' \otimes q') + (p' \otimes q') \leftPrelie (p \otimes q) \\
    &\ =
    \area(p,p') \otimes [q,q'].
  \end{align*}
This product was introduced exactly like this already by Rocha in \cite[Lemma 6.5]{bib:Roc2003b} and is the tensor algebra analogue of the vector field product also introduced by Rocha in \cite[Equation (6.13) and Proposition 6.3]{bib:Roc2003b}.

\begin{corollary}
\label{cor:Rrecursion}
  We have
  \begin{align*}
    (\lift{D}-\id) R = \frac{1}{2} R \leftPrelieSym R.  
  \end{align*}

  Let $R_n := \proj_n R = \sum_{|w|=n} w \otimes r[w]$ be the $n$-th level of $R$. \index{Rn@$R_n$}
  Then for $n \ge 2$ this spells out as
  \begin{align*}
    (n-1) R_n
    &=
    \frac{1}{2} \sum_{\ell=1}^{n} R_\ell \leftPrelieSym R_{n-\ell}
    =
    \begin{cases}
      \sum_{\ell=1}^{(n-1)/2} R_\ell \leftPrelieSym R_{n-\ell} \ \ \qquad \qquad \qquad \qquad\ n \text{ odd } \\
      \sum_{\ell=1}^{n/2} R_\ell \leftPrelieSym R_{n-\ell} + \frac{1}{2} R_{n/2} \leftPrelieSym R_{n/2} \qquad n \text{ even }
    \end{cases}
  \end{align*}
 with $R_n\in\quadsym:=\gen{\word{i}\otimes\word{i},\,\word{i}=\word{1}\ldots\word{d}}{\leftPrelieSym}$.
\end{corollary}
\begin{proof}
  This follows immediately from Lemma \ref{lem:R}.
\end{proof}

\begin{remark}
  Note that \cite[Proposition 6.8]{bib:Roc2003b}
  has a slightly more complicated recursion.
  This stems from the facts that
  the $Z^n$ there relates to our $\frac{R_n}{n!}$ here.
\end{remark}




\begin{example}
  \label{ex:R}
  Let $R_n := \proj_n R$ be the $n$-th level of $R$.
  Then Lemma \ref{lem:R} and Corollary \ref{cor:Rrecursion} give
  \begin{align*}
    R_2   &= R_1 \leftPrelie R_1
          = \frac{1}{2} R_1 \leftPrelieSym R_1 \\
    2 R_3 &= R_1 \leftPrelie R_2 + R_2 \leftPrelie R_1
          = \frac{1}{2} R_1 \leftPrelieSym R_2 + \frac{1}{2} R_2 \leftPrelie R_1
          = R_1 \leftPrelieSym R_2 \\
    3 R_4 &= R_1 \leftPrelie R_3 + R_2 \leftPrelie R_2 + R_3 \leftPrelie R_1
          = \frac{1}{2} R_1 \leftPrelieSym R_3 + \frac{1}{2} R_2 \leftPrelieSym R_2 + \frac{1}{2} R_3 \leftPrelieSym R_1 \\
          &= R_1 \leftPrelieSym R_3 + \frac{1}{2} R_2 \leftPrelieSym R_2
  \end{align*}

  \newcommand{\w}[1]{\word{#1}}
  For $\ds=2$ this becomes
  \begin{align*}
    R_1 &= \w{1} \otimes \w{1} + \w{2} \otimes \w{2} \\
    R_2 &=
    \frac{1}{2} \left( \area(\w{1},\w{2}) \otimes [\w{1},\w{2}] + \area(\w{2},\w{1}) \otimes [\w{2},\w{1}] \right) \\
    R_3 &= \frac{1}{4}
       \Big(
         \area(\w{1},\area(\w{1},\w{2})) \otimes [\w{1},[\w{1},\w{2}]] + \area(\w{1},\area(\w{2},\w{1})) \otimes [\w{1},[\w{2},\w{1}]] \\
         &\qquad
         + \area(\w{2},\area(\w{1},\w{2})) \otimes [\w{2},[\w{1},\w{2}]] + \area(\w{2},\area(\w{2},\w{1})) \otimes [\w{2},[\w{2},\w{1}]]
          \Big) \\
   R_4 &= \frac{1}{12} \Big(
            \area(\w{1},\area(\w{1},\area(\w{1},\w{2}))) \otimes [\w{1},[\w{1},[\w{1},\w{2}]]]
            +
            \area(\w{1},\area(\w{1},\area(\w{2},\w{1}))) \otimes[\w{1},[\w{1},[\w{2},\w{1}]]]\\
          &\qquad
          +  \area( \w{1},\area( \w{2}, \area( \w{1},\w{2})))\otimes [\w{1},[\w{2},[\w{1},\w{2}]]]
          + \area(\w{1},\area(\w{2},\area(\w{2},\w{1})))\otimes[\w{1},[\w{2},[\w{2},\w{1}]]] \\
          &\qquad
          +  \area( \w{2},\area(\w{1},\area(\w{1},\w{2}))) \otimes [\w{2},[\w{1},[\w{1},\w{2}]]]
          + \area(\w{2},\area(\w{1},\area(\w{2},\w{1})))\otimes[\w{2},[\w{1},[\w{2},\w{1}]]] \\
          &\qquad
        +  \area( \w{2},\area(\w{2},\area(\w{1},\w{2})))\otimes [\w{2},[\w{2},[\w{1},\w{2}]]]
      + \area(\w{2},\area(\w{2},\area(\w{2},\w{1})))\otimes[\w{2},[\w{2},[\w{2},\w{1}]]] \Big) \\
  \end{align*}

\end{example}

\tikzset{bintrees/.style={%
    grow'=up
  ,every node/.style={scale=.8}%
  ,baseline={([yshift=-0.8ex]current bounding box.center)}
  ,level distance=0.3cm
  ,sibling distance=0.3cm
  ,inner sep=1pt
  ,dot/.style={circle,fill, inner sep=0pt, minimum size=0.3em}
  ,squ/.style={fill, inner sep=0pt, minimum size=0.3em}
  },
 wider/.append style={
  sibling distance=0.6cm
  ,level 2/.append style = {sibling distance = 0.3cm}
  },
 bitwider/.append style={
  sibling distance=0.4cm
  ,level 2/.append style = {sibling distance = 0.3cm}
  }}

\newcommand{\bindot}{\tikz[bintrees]{\node[dot]{};}}
\newcommand{\areatree}{\area_{\tikz[bintrees,baseline=(current bounding box.center)]{\node[dot]{};}}}
\newcommand{\brackettree}{\bracket_{\tikz[bintrees,baseline=(current bounding box.center)]{\node[dot]{};}}}
\newcommand{\discreteAreatree}{\discreteArea_{\tikz[bintrees,baseline=(current bounding box.center)]{\node[dot]{};}}}
\newcommand{\binsqu}{\tikz[bintrees]{\node[squ]{};}}
\newcommand{\treett}{\tikz[bintrees,bitwider]{\node[dot]{}child{node {$\tau_1$}} child{node {$\tau_2$}}}}
In general this looks as follows.
\begin{definition}
  \label{def:binaryPlanarTrees}
  Denote by $\binaryPlanarTrees_n$ the set of (complete, rooted) binary planar trees with $n$ leaves labelled with the letters $\word1,\dots,\word\ds$.\index{BinaryPlanarTrees@$\binaryPlanarTrees_n$}
  Given $\tau \in \binaryPlanarTrees_n$ we define $\areatree(\tau)$ (resp. $\brackettree(\tau)$) \index{areas@$\areatree$}\index{bracket@$\brackettree$}
    as the bracketing-out using $\area$ (resp. $[.,.]$).
  For example
  \begin{align*}
    \areatree(\tikz[bintrees]{\node[dot]{}child{node [dot]{}child{node {$\word 1$}}child{node {$\word 2$}}}child{node {$\word 3$}}}) &= \area( \word1, \area(\word2, \word3) ) \\
    \brackettree( \tikz[bintrees,wider]{\node[dot,wider]{}child{node [dot]{}child{node {$\word 1$}}child{node {$\word 2$}}}child{node [dot]{}child{node {$\word 3$}}child{node {$\word 4$}}}} ) &= [ [\word1,\word2],[\word3,\word4] ].
  \end{align*}

  Define a function $c:\binaryPlanarTrees_n\to\R$, which does not depend on the specific letter labels, recursively as follows \index{c@$c$}
  \begin{align*}
    c (\word{i}) &= 1 \qquad \text{for any $\word {i}$ in $\word1,\dots,\word\ds$}  \\
    c(\treett)&= 2c(\tau_1) c(\tau_2) (\leaves{\tau_1}+\leaves{\tau_2} - 1)
  \end{align*}
  where $\leaves{\tau}$ denotes the number of leaves of the tree $\tau$.
  For example
  \begin{align*}
    c( \word2 ) &= 1 \\
    c(\tikz[bintrees]{\node [dot]{}child{node {$\word 1$}}child{node {$\word 2$}}} ) &= 2\cdot 1 \cdot 1 \cdot (2-1) = 2 \\
    c( \tikz[bintrees]{\node[dot]{}child{node {$\word 3$}}child{node [dot]{}child{node {$\word 1$}}child{node {$\word 2$}}}} ) &= 2\cdot 1 \cdot 2 \cdot (3-1) = 8
  \end{align*}
\end{definition}

\begin{lemma}
  \label{lem:RwithTrees}
  \begin{align*}
    R_n = \sum_{\tau \in \binaryPlanarTrees_n} \frac{1}{c(\tau)}\ \areatree(\tau ) \otimes \brackettree(\tau).
  \end{align*}
\end{lemma}
\begin{remark}
  We note that \cite[Lemma 1]{bib:Roc2003} has
  a slightly more complicated expression for $R_n$,
  since in that work some of terms are factored out, owing to antisymmetry.
  We do not pursue this here
  since the end result, also in \cite{bib:Roc2003}, still contains redundant terms,
  which we do not know how to explicitly get rid of.
  In fact, due to antisymmetry alone, we already know that $c(\tau)$ is not a unique choice for this equation to hold,
  however it remains an interesting and more involved question if it is the only choice which is symmetric, i.e.~well-defined on \emph{non-planar} trees, and invariant under change of the leaf labels.
  
  A further very interesting question is to find a modified $c'$ which may not be symmetric and may depend on the leaf labels, such that the equations still holds, but such that the number of non-zero summands in the equation is minimized for each $n$.
\end{remark}

\begin{proof}[Proof of \ref{lem:RwithTrees}]
 For the purpose of this proof, let $R_n$ be defined as in Corollary \ref{cor:Rrecursion} and 
 \begin{equation*}
  R_n':=\sum_{\tau \in \binaryPlanarTrees_n} \frac{1}{c(\tau)}\ \areatree(\tau ) \otimes \brackettree(\tau).
 \end{equation*}
 We proceed by induction over $n$. We have 
 \begin{equation*}
   R_1=\sum_{|w|=1}w\otimes r(w)=\sum_{|w|=1}w\otimes w=\sum_{|w|=1}\frac{1}{c(w)}\areatree(w)\otimes\brackettree(w)=R_1'.
 \end{equation*}
 Assuming $R_n=R_n'$ holds for some $n\in\mathbb{N}$, we get
 \begin{align*}
  &R_{n+1}=\frac{1}{2n}\sum_{l=1}^{n+1} R_l \leftPrelieSym R_{n-l}\\
  &=\frac{1}{2n} \sum_{l=1}^{n+1}\sum_{\substack{\tau_1\in\binaryPlanarTrees_l,\\\tau_2\in\binaryPlanarTrees_{n+1-l}}}\frac{1}{c(\tau_1)c(\tau_2)}\area(\areatree(\tau_1),\areatree(\tau_2))\otimes[\brackettree(\tau_1),\brackettree(\tau_2)]\\
  &=\sum_{l=1}^{n+1}\sum_{\substack{\tau_1\in\binaryPlanarTrees_l,\\\tau_2\in\binaryPlanarTrees_{n+1-l}}}\frac{1}{c(\treett)}\areatree(\treett)\otimes\brackettree(\treett)\\
  &= \sum_{\tau \in \binaryPlanarTrees_{n+1}} \frac{1}{c(\tau)}\areatree(\tau) \otimes \brackettree(\tau )
  =R_{n+1}'.\qedhere
 \end{align*}
\end{proof}

Rather than working with the recursion for $R_n$ from Corollary \ref{cor:Rrecursion} directly on $\TSTCRing$, 
in the following theorem we will pursue the alternative approach of first applying the $\coeval^{S_h}$ to work on $\TS$, or, to be more specific, on $\areatortkara$ as we will see.

\begin{theorem}
  \label{thm:Rho}
 We have $R=\sum_h \Rho_h \otimes P_h$ where $\Rho_h := \rho(S_h)$ satisfies the recursion
 \begin{equation}\label{eq:Rho_recursion}
  \Rho_h=\frac{1}{|h|-1}\sum_{h_1<h_2}\langle S_h,[P_{h_1},P_{h_2}]\rangle\, \area(\Rho_{h_1},\Rho_{h_2}).
 \end{equation}
 More explicitly, we have
 \begin{equation*}
  \Rho_h=\sum_{\tau}\frac{1}{b(\tau)}q^h_{\tau}\,\areatree(\tau)=\sum_{\tau}\frac{1}{c(\tau)}p^h_{\tau}\,\areatree(\tau)
 \end{equation*}
with
\begin{align*}
 q^h_{\tau}&=\sum_{h_1<h_2}q^{h_1}_{\tau'}q^{h_2}_{\tau''}\langle S_h,[P_{h_1},P_{h_2}]\rangle,\\
 p^h_{\tau}&=\langle S_h,\brackettree(\tau)\rangle=\sum_{h_1,h_2}q^{h_1}_{\tau'}q^{h_2}_{\tau''}\langle S_h,[P_{h_1},P_{h_2}]\rangle
\end{align*}
for $\leaves{\tau},|h|\geq 2$ and
\begin{align*}
  q^{h}_{\word{i}}&=p^{h}_{\word{i}}=\delta_{h,\word{i}},\\
  q^{\word{i}}_{\tau}&=p^{\word{i}}_{\tau}=\delta_{\tau,\word{i}}
\end{align*}
with $b(\tau)=b(\tau')b(\tau'')(\leaves{\tau'}+\leaves{\tau''} -1)$, $b(\word{i})=1$.

\end{theorem}
\begin{proof}
 We have $r=\eval(R)$ with $\Im r=\mathfrak{g}$ and thus
 \begin{equation*}
  r=r\circ \eval(\sum_{h}S_h\otimes P_h)=\eval\Big(\sum_{w,h}\langle w,S_h\rangle \rho(w)\otimes P_h\Big)=\eval\Big(\sum_h\rho(S_h)\otimes P_h\Big),
 \end{equation*}
 which means $R=\sum_{h}\rho(S_h)\otimes P_h$ since $\eval$ is bijective. Putting $\Rho_h:=\rho(S_h)$, we get
 \begin{equation*}
  \sum_h(|h|-1)\Rho_h\otimes P_h=(\lift{D}-\id)R=\tfrac{1}{2}\,R\leftPrelieSym R=\tfrac{1}{2}\,\sum_{h_1,h_2}\area(\Rho_{h_1},\Rho_{h_2})\otimes[P_{h_1},P_{h_2}],
 \end{equation*}
 which by applying $\coeval^{S_h}$ on both sides yields
 \begin{equation*}
  (|h|-1)\Rho_h=\tfrac{1}{2}\sum_{h_1,h_2}\langle S_h,[P_{h_1},P_{h_2}]\rangle\area(\Rho_{h_1},\Rho_{h_2})=\sum_{h_1<h_2}\langle S_h,[P_{h_1},P_{h_2}]\rangle\area(\Rho_{h_1},\Rho_{h_2}).
 \end{equation*}
 Since due to $q_{\tau}^{\word{i}}=\delta_{\tau,\word{i}}$ and $b(\word{i})=1$ we have
 \begin{equation*}
  \Rho_{\word{i}}=\rho(\word{i})=\word{i}=\areatree(\word{i})=\sum_{\tau}\frac{1}{b(\tau)}q^{\word{i}}_\tau\areatree(\tau),
 \end{equation*}
 we obtain by induction
 \begin{align*}
  \Rho_h&=\frac{1}{|h|-1}\sum_{h_1<h_2}\langle S_h,[P_{h_1},P_{h_2}]\rangle\area(\Rho_{h_1},\Rho_{h_2})\\
  &=\frac{1}{|h|-1}\sum_{h_1<h_2}\langle S_h,[P_{h_1},P_{h_2}]\rangle\sum_{\tau_1,\tau_2}\frac{1}{b(\tau_1)b(\tau_2)}q^{h_1}_{\tau_1}q^{h_2}_{\tau_2}\area(\areatree(\tau_1),\areatree(\tau_2))\\
  &=\sum_{\tau_1,\tau_2}\frac{1}{(|h|-1)b(\tau_1)b(\tau_2)}\sum_{h_1<h_2}q_{\tau_1}^{h_1}q_{\tau_2}^{h_2}\langle S_h,[P_{h_1},P_{h_2}]\rangle\areatree(\treett)\\
  &=\sum_{\tau}\frac{1}{(\leaves{\tau'}+\leaves{\tau''} -1)b(\tau')b(\tau'')}\sum_{h_1<h_2}q_{\tau'}^{h_1}q_{\tau''}^{h_2}\langle S_h,[P_{h_1},P_{h_2}]\rangle\areatree(\tau)\\
  &=\sum_{\tau}\frac{1}{b(\tau)}q_\tau^h\areatree(\tau)
 \end{align*}
 Furthermore, due to Lemma \ref{lem:RwithTrees}, we have
 \begin{equation*}
  \Rho_h=\coeval^{S_h}(\Rho_{|h|})=\sum_{\tau\in\binaryPlanarTrees_{|h|}}\frac{1}{c(\tau)}\langle S_h,\brackettree(\tau)\rangle\areatree(\tau)=\sum_{\tau}\frac{1}{c(\tau)}p_\tau^h\areatree(\tau),
 \end{equation*}
 where we have $p_{\word{i}}^h=\langle S_h,\brackettree(\word{i})\rangle=\langle S_h,\word{i}\rangle=\delta_{h,\word{i}}$ as well as $p_{\tau}^{\word{i}}=\langle S_{\word{i}},\brackettree(\tau)\rangle=\langle\word{i},\brackettree(\tau)\rangle=\delta_{\tau,\word{i}}$ and by induction over $\leaves{\tau}\geq 2$
 \begin{align*}
  p_{\tau}^h&=\langle S_h,\brackettree(\tau)\rangle=\langle S_h,[\brackettree(\tau'),\brackettree(\tau'')]\rangle=\sum_{h_1,h_2}\langle S_h,[\langle S_{h_1},\brackettree(\tau')\rangle P_{h_1},\langle S_{h_2},\brackettree(\tau')\rangle P_{h_2}]\rangle\\
  &=\sum_{h_1,h_2}p_{\tau'}^{h_1}p_{\tau''}^{h_2}\langle S_h,[P_{h_1},P_{h_2}]\rangle.\qedhere
 \end{align*}
 \end{proof}

\begin{remark}
 Let $h(\tau)$ be the Hall word corresponding to the Hall tree $\tau$. Then,
 \begin{equation*}
  q^{h_0}_{\tau}=p^{h_0}_{\tau}=\delta_{h(\tau),h_0}.
 \end{equation*}
 This is immediate by definition of $q$ for $|h(\tau)|=\leaves{\tau}=1$, and then by induction over $\leaves{\tau}$ if $\tau$ is a Hall tree, then $\tau',\tau''$ are Hall trees and thus
 \begin{align*}
  q_{\tau}^{h}
  &=\sum_{h_1<h_2}q_{\tau'}^{h_1}q_{\tau''}^{h_2}\langle S_h,[P_{h_1},P_{h_2}]\rangle=\sum_{h_1<h_2}\delta_{h(\tau'),h_1}\delta_{h(\tau''),h_2}\langle S_h,[P_{h_1},P_{h_2}]\rangle\\
  &=\langle S_h,[P_{h(\tau')},P_{h(\tau'')}]\rangle=\delta_{h(\tau),h}
 \end{align*}
 due to $(P_h)_h$ and $(S_h)_h$ being dual bases. For $p_{\tau}^h=\langle S_h,\brackettree(\tau)\rangle$ the claim is immediate.
 
 Note furthermore that we could also have derived
 \begin{equation*}
  \Rho_h=\sum_{\tau}\frac{1}{b(\tau)}q_{\tau}^{h}\areatree(\tau)
 \end{equation*}
 from
 \begin{equation*}
  \Rho_h=\sum_{\tau}\frac{1}{c(\tau)}p_{\tau}^{h}\areatree(\tau)
 \end{equation*}
 by looking at how $q_{\tau}^h$ and $p_{\tau}^h$ relate to each other.
\end{remark}
\begin{example}
 \todonotes{JR: Move to some place after the $h$,$S_h$,$\zeta$ tables?}
 In the case of $T(\R^2)$ and the Lyndon words $H$, the values of $\Rho_h$ up to level five are
 \begin{align*}
  \Rho_{\word{1}}&=\word{1},\quad \Rho_{\word{2}}=\word{2},\\
  \Rho_{\word{12}}&=\word{12}-\word{21}=\area(\word{1},\word{2}),\\
  \Rho_{\word{112}}&=\word{112}-\word{121}=\tfrac{1}{2}\,\area(\word{1},\area(\word{1},\word{2})),\\
  \Rho_{\word{122}}&=-\word{212}+\word{221}=\tfrac{1}{2}\,\area(\area(\word{1},\word{2}),\word{2}),\\
  \Rho_{\word{1112}}&=\word{1112}-\word{1121}=\tfrac{1}{6}\,\area(\word{1},\area(\word{1},\area(\word{1},\word{2})))\\
  \Rho_{\word{1122}}&=-\word{1212}+\word{1221}-\word{2112}+\word{2121}=\tfrac{1}{6}\,\area(\word{1},\area(\area(\word{1},\word{2}),\word{2}))+\tfrac{1}{6}\,\area(\area(\word{1},\area(\word{1},\word{2})),\word{2}),\\
  \Rho_{\word{1222}}&=\word{2212}-\word{2221}=\tfrac{1}{6}\,\area(\area(\area(\word{1},\word{2}),\word{2}),\word{2}),\\
  \Rho_{\word{11112}}&=\word{11112}-\word{11121},\\
  \Rho_{\word{11122}}&=-\word{11212}+\word{11221}-\word{12112}+\word{12121}-\word{21112}+\word{21121},\\
  \Rho_{\word{11222}}&=\word{12212}-\word{12221}+\word{21212}-\word{21221}+\word{22112}-\word{22121},\\
  \Rho_{\word{12122}}&=\word{21212}-\word{21221}+\word{22112}-\word{22121},\\
  \Rho_{\word{11212}}&=\word{21112}-\word{21121},\\
  \Rho_{\word{12222}}&=-\word{22212}+\word{22221},
 \end{align*}
 where we gave the area bracketings according to the recursion Equation \eqref{eq:Rho_recursion} up to level four. The trend of the values being just $-1,0,1$ combinations of words does not continue to higher levels, e.g.
 \begin{equation*}
  \Rho_{\word{112212}}=-\word{211212}+\word{211221}-\word{212112}+\word{212121}-3\,\word{221112}+3\,\word{221121}.
 \end{equation*}
\todonotes{JR: Compute $\Rho_h$ for all $h$ in the tables}
\end{example}
\begin{example}
 In the case of $T(\R^3)$ and the Lyndon words $H$, the values of $\Rho_h$ up to level four which are not immediate from the previous example are
 \begin{align*}
  \Rho_{\word{123}}&=\word{123}-\word{132}-\word{312}+\word{321},\\
  \Rho_{\word{132}}&=-\word{213}+\word{231}-\word{312}+\word{321},\\
  \Rho_{\word{1123}}&=\word{1123}-\word{1132}-\word{1312}+\word{1321}-\word{3112}+\word{3121},\\
  \Rho_{\word{1132}}&=-\word{1213}+\word{1231}-\word{1312}+\word{1321}-\word{2113}+\word{2131}-\word{3112}+\word{3121},\\
  \Rho_{\word{1213}}&=-\word{2113}+\word{2131}+\word{3112}-\word{3121},\\
  \Rho_{\word{1223}}&= \word{1223}-\word{1232}+\word{3212}-\word{3221},\\
  \Rho_{\word{1232}}&=-\word{2123}+\word{2132}+\word{2312}-\word{2321}+2\,\word{3212}-2\,\word{3221},\\
  \Rho_{\word{1233}}&=-\word{1323}+\word{1332}-\word{3123}+\word{3132}+\word{3312}-\word{3321},\\
  \Rho_{\word{1322}}&= \word{2213}-\word{2231}+\word{2312}-\word{2321}+\word{3212}-\word{3221},\\
  \Rho_{\word{1323}}&= -\word{3123}+\word{3132}+\word{3213}-\word{3231}+2\,\word{3312}-2\,\word{3321},\\
  \Rho_{\word{1332}}&=\word{2313}-\word{2331}+\word{3213}-\word{3231}+\word{3312}-\word{3321}.
 \end{align*}

\end{example}
\begin{example}
 In the case of $T(\R^2)$ and the standard Hall words $H$, the values of $\Rho_h$ up to level five are
 \begin{align*}
  \Rho_{\word{1}}&=\word{1},\quad \Rho_{\word{2}}=\word{2},\\
  \Rho_{\word{12}}&=\word{12}-\word{21},\\
  \Rho_{\word{121}}&=-\word{112}+\word{121},\\
  \Rho_{\word{122}}&=-\word{212}+\word{221},\\
  \Rho_{\word{1211}}&=\word{1112}-\word{1121},\\
  \Rho_{\word{1221}}&=\word{1212}-\word{1221}+\word{2112}-\word{2121},\\
  \Rho_{\word{1222}}&=\word{2212}-\word{2221},\\
  \Rho_{\word{12111}}&=-\word{11112}+\word{11121},\\
  \Rho_{\word{12211}}&=-\word{11212}+\word{11221}-\word{12112}+\word{12121}-\word{21112}+\word{21121},\\
  \Rho_{\word{12221}}&=-\word{12212}+\word{12221}-\word{21212}+\word{21221}-\word{22112}+\word{22121},\\
  \Rho_{\word{12222}}&=-\word{22212}+\word{22221},\\
  \Rho_{\word{12112}}&=-\word{21112}+\word{21121},\\
  \Rho_{\word{12212}}&=-\word{21212}+\word{21221}-\word{22112}+\word{22121},
 \end{align*}
 where once again the trend of the values being just $-1,0,1$ combinations of words does not continue to higher levels, e.g.
 \begin{align*}
  \Rho_{\word{122112}}&=\word{121212}-\word{121221}+\word{122112}-\word{122121}+2\,\word{211212}-2\,\word{211221}\\
  &\quad+2\,\word{212112}-2\,\word{212121}+3\,\word{221112}-3\,\word{221121}.
 \end{align*}

\end{example}

\section{Coordinates of the first kind}
\label{sec:coordinates}

Let $(P_h)_{h\in H}$ be a basis for the free Lie algebra $\primitive$. \index{P@$P_h$}
For the index set $H$ we have a Hall set in mind (\cite[Section 4]{bib:Reu1993}),
but this is not necessary at this stage.
%
%
%
%
Any grouplike element $g \in \grouplike$ can be written as the exponential of a Lie series,
\begin{align}
  \label{eq:coordinatesOfFirstKind}
  g = \exp\left( \sum_{h\in H} c_h(g) P_h \right),
\end{align}
for some uniquely determined $c_h(g) \in \R$.
In fact, there exist unique $\zeta_h \in \TS$, $h\in H$, such that \index{zeta@$\zeta_h$}
$c_h(g) = \Big\langle \zeta_h, g \Big\rangle$.
The $\zeta_h$ are called the \textbf{coordinates of the first kind} (corresponding to $(P_h)_{h\in H}$), see for example \cite{bib:Kaw2009}.

We now formulate this in a way, where we do not have to test against $g \in G$.
Recall the product $\shuffleConcat$ on $\TSTC$: shuffle product on the left and concatenation product on the right.

For words $a,b$
\begin{align*}
  \eval_g( a \otimes b) \conc
  \eval_g( a' \otimes b' )
  &=
  \Big\langle a, g \Big\rangle\ \Big\langle a', g \Big\rangle\ bb'
  =
  \Big\langle a \shuffle a', g \Big\rangle\ bb' \\
  &=
  \eval_g\Big( (a \shuffle a') \otimes (b \conc b' ) \Big)
  =
  \eval_g\Big( (a\otimes b) \shuffleConcat (a'\otimes b') \Big).
\end{align*}
Both expressions are bilinear, so this is true for general elements in $\TSTC$.
Hence $\eval_g$ is an algebra homomorphism from $(\TSTC, \shuffleConcat)$ to $(\TC, \conc)$.
Then on one hand, using first \eqref{eq:coordinatesOfFirstKind} and then the homomorphism property
\begin{align*}
  g
  &= \exp\left( \sum_{h\in H} c_h P_h \right) = \exp\left( \sum_{h\in H} \langle \zeta_h, g \rangle P_h \right)
  = \exp\left( \eval_g \left( \sum_{h\in H} \zeta_h \otimes P_h \right) \right)\\
  &= \eval_g\left( \exp_{\shuffleConcat}\left( \sum_{h\in H} \zeta_h \otimes P_h \right) \right).
\end{align*}
Here, of course, for $x \in \TSTC$,
\begin{align*}
  \exp_{\shuffleConcat}(x)
  &:= \sum_{n\ge 0} \frac{ x^{ \shuffleConcat n} }{n!}
  := \sum_{n\ge 0} \frac{ \overbrace{x \shuffleConcat \dots \shuffleConcat x}^{n \text{ times}} }{n!}.
\end{align*}

On the other hand, trivially
\begin{align*}
  g &= \sum_w \langle w, g \rangle\ w = \eval_g\left( \sum_w w \otimes w \right).
\end{align*}

Since grouplike elements linearly span%
\footnote{This is true level by level, i.e.~projectively. See for example \cite[Lemma 8]{bib:DR2018}.}
all of $\TC$
we get that for all $x \in \TC$
\begin{align*}
  \eval_x\left( \exp_{\shuffleConcat}\left( \sum_{h\in H} \zeta_h \otimes P_h \right) \right)
  =
  \eval_x\left( \sum_w w \otimes w \right),
\end{align*}
which is equivalent to
\begin{align}
  \label{eq:equivalentTo}
  \sum_w w \otimes w = \exp_{\shuffleConcat}\left( \sum_{h\in H} \zeta_h \otimes P_h \right),
\end{align}
respectively
\begin{align*}
  \log_{\shuffleConcat} \sum_w w \otimes w = \sum_{h\in H} \zeta_h \otimes P_h.
\end{align*}
We have arrived at a definition of coordinates of the first kind which does not rely on testing against grouplike elements.

\begin{remark}
  \label{rem:mathcalR}
  Considering $S$ as an element of $\TSTCRing$,
  it is grouplike.
  Indeed, for $a,b \in \TS$,
  \begin{align*}
    \langle \underbar a\mtimes \underbar b, \mdeshuffle S \rangle
    &=
    \langle \underbar a\mtimes \underbar b, \mdeshuffle \sum_w w \underbar w \rangle 
    =
    \sum_w w\, \langle \underbar a\mtimes \underbar b, \mdeshuffle \underbar w \rangle
    =
    \sum_w w \langle \underbar a\mshuffle \underbar b, \underbar w \rangle \\
    &=
    a \shuffle b 
    =
    \langle \underbar a \mtimes \underbar b, \sum_{w,v} w \shuffle v\ \underbar w \mtimes \underbar v \rangle 
    =
    \langle \underbar a \mtimes \underbar b, \sum_w w \underbar w \mtimes \sum_v v \underbar v \rangle 
    \\&=
    \langle \underbar a \mtimes \underbar b, S \mtimes S \rangle.
  \end{align*}
  
  Here for a word $w \in \TS$ we write $\underbar w$ as its realization in $\TSTCRing$.
  Then
  \begin{align*}
    \Lambda := \log_{\shuffleConcat} S,
  \end{align*}
  is primitive.
  The search for coordinates of the first kind then amounts to finding
  a ``simple'' expression for this primitive element.
\end{remark}

~\\

%
One can construct the coordinates $\zeta_h$ as follows.
Pick any $S_h \in T(\R^d), h \in H$, such that 
\index{S@$S_h$}
\begin{align*}
  \langle S_h, P_{h'} \rangle = \delta_{h,h'}.
\end{align*}
Using \cite[Theorem 5.3]{bib:Reu1993}, if $P_h$ is a Hall basis,
one can actually pick the $S_h$ in such a way that they extend to the dual of the corresponding PBW basis of $T((\R^d))$,
and while this is not necessary here, in the rest of the paper we really mean that the $S_h$ are chosen in this specific way.
Then
\begin{align*}
  \Big\langle  S_h, \log\left( \exp\left( \sum_{h'\in H} c_{h'} P_{h'} \right) \right) \Big\rangle
  &=
  \Big\langle  S_h, \sum_{h'\in H} c_{h'} P_{h'} \Big\rangle
  =
  c_h.
\end{align*}

We want ``to put the logarithm on the other side''. This is indeed possible,
since the logarithm on grouplike elements extends to a \emph{linear} map $\pi_1$
on all of $\TC$ (see \cite[Section 3.2]{bib:Reu1993}
and also \cite{bib:MNT2013} for a general overview on idempotents), given as
\begin{align}
  \label{eq:eulerianIdempotent}
  \pi_1(u) := 
  \sum_{n\ge 1} \frac{(-1)^{n+1}}{n} \sum_{v_1,..,v_n \text{ non-empty}} \langle v_1 \shuffle .. \shuffle v_n, u \rangle\ v_1 \conc .. \conc v_n.
\end{align}
Denote its dual map by $\pi_1^\top$.%
\footnote{\cite[Section 6.2]{bib:Reu1993} uses the notation $\pi_1^*$.} \index{pi1@$\pi_1$ and $\pi_1^\top$}
It is given as
\begin{align*}
  \pi^\top_1(v) = \sum_{n\ge 1} \frac{(-1)^{n+1}}{n} \sum_{u_1,..,u_n \text{ non-empty}} \langle v, u_1 \conc .. \conc u_n \rangle\ u_1 \shuffle .. \shuffle u_n.
\end{align*}

Then for all $h \in H$
\begin{align*}
  \Big\langle  \pi^\top_1 S_h, \exp\left( \sum_{h'\in H} c_{h'} P_{h'} \right) \Big\rangle
  =
  c_h,
\end{align*}
that is, the coordinate of first kind are given by (compare \cite[Theorem 1]{bib:GK2008})
\begin{align}
  \label{eq:pi1starCoordinates}
  \zeta_h = \pi^\top_1 S_h \qquad h \in H.
\end{align}

We note that $\zeta_h$ must of course be independent of the choice of the $S_h$ and this is indeed the case,
since $\ker \pi^\top_1 = (\im \pi_1)^\bot = \primitive_n^\bot$.

\begin{example}
  \label{ex:coordinatesOfTheFirstKind}
  Let $(P_h)_{h\in H}$ be the Lyndon basis (which is a Hall basis, \cite[Section 5]{bib:Reu1993}).
  In the case $d=2$, we give in \autoref{tab:example} the first few elements for $P_h, S_h$ and $\pi_1^\top S_h$,
  where we take $S_h$ as in \cite[Theorem 5.3]{bib:Reu1993}.

  \begin{table}
    \centering
  \newcommand{\specialcell}[2][c]{\begin{tabular}[#1]{@{}l@{}}#2\end{tabular}}
\begin{tabular}{llll}
\shortstack{Lyndon\\word $h$} & $P_h$ & $S_h$ & $\zeta_h=\pi_1^\top S_h$ \\
\hline
$1$     & $\word{1}$                 & $\word{1}$                     & $\word{1}$ \\
$2$     & $\word{2}$                 & $\word{2}$                     & $\word{2}$ \\
$12$    & $[\word{1},\word{2}]$             & $\word{12}$                    & $+\tfrac{1}{2} \word{12} -\tfrac{1}{2} \word{21}$ \\
$112$   & $[\word{1},[\word{1},\word{2}]]$         & $\word{112}$                   & $+\tfrac{1}{6} \word{112} -\tfrac{1}{3} \word{121} +\tfrac{1}{6} \word{211}$ \\
$122$   & $[[\word{1},\word{2}],\word{2}]$         & $\word{122}$                   & $+\tfrac{1}{6} \word{122} -\tfrac{1}{3} \word{212} +\tfrac{1}{6} \word{221}$ \\
$1112$  & $[\word{1},[\word{1},[\word{1},\word{2}]]]$     & $\word{1112}$                  & $-\tfrac{1}{6} \word{1121} +\tfrac{1}{6} \word{1211}$ \\
$1122$  & $[\word{1},[[\word{1},\word{2}],\word{2}]]$     & $\word{1122}$                  & $+\tfrac{1}{6} \word{1122} -\tfrac{1}{6} \word{1212} +\tfrac{1}{6} \word{2121} -\tfrac{1}{6} \word{2211}$ \\
$1222$  & $[[[\word{1},\word{2}],\word{2}],\word{2}]$     & $\word{1222}$                  & $-\tfrac{1}{6} \word{2122} +\tfrac{1}{6} \word{2212}$ \\

$11112$ & $[\word{1},[\word{1},[\word{1},[\word{1},\word{2}]]]]$ & $\word{11112}$                 & $\frac{1}{30}[\tiny\begin{array}{@{}l@{}} -\word{11112}-\word{11121}+4\word{11211}-\word{12111}-\word{21111} \end{array}]$\\
$11122$ & $[\word{1},[\word{1},[[\word{1},\word{2}],\word{2}]]]$ & $\word{11122}$                 & $\frac{1}{30}[\tiny\begin{array}{@{}l@{}} 2\word{11122}-3\word{11212}-3\word{11221}+2\word{12112}+2\word{12121}\\\quad-3\word{12211}+2\word{21112}+2\word{21121}-3\word{21211}+2\word{22111} \end{array}]$\\
$11222$ & $[\word{1},[[[\word{1},\word{2}],\word{2}],\word{2}]]$ & $\word{11222}$                 & $\frac{1}{30}[\tiny\begin{array}{@{}l@{}} 2\word{11222}-3\word{12122}+2\word{12212}+2\word{12221}-3\word{21122}\\\quad+2\word{21212}+2\word{21221}-3\word{22112}-3\word{22121}+2\word{22211} \end{array}]$\\
$12122$ & $[[\word{1},\word{2}],[[\word{1},\word{2}],\word{2}]]$ & $\word{12122} + 3\ \word{11222}$ & $\frac{1}{30}[\tiny\begin{array}{@{}l@{}} 3\word{11222}-2\word{12122}-2\word{12212}+3\word{12221}-2\word{21122}\\\quad+3\word{21212}-2\word{21221}-2\word{22112}-2\word{22121}+3\word{22211} \end{array}]$\\
$11212$ & $[[\word{1},[\word{1},\word{2}]],[\word{1},\word{2}]]$ & $\word{11212} +2\ \word{11122}$ & $\frac{1}{30}[\tiny\begin{array}{@{}l@{}} \word{11122}+\word{11212}+\word{11221}-4\word{12112}+\word{12121}\\\quad+\word{12211}+\word{21112}-4\word{21121}+\word{21211}+\word{22111}\end{array}]$\\
$12222$ & $[[[[\word{1},\word{2}],\word{2}],\word{2}],\word{2}]$ & $\word{12222}$                 & $\frac{1}{30}[\tiny\begin{array}{@{}l@{}} -\word{12222}-\word{21222}+4\word{22122}-\word{22212}-\word{22221} \end{array}]$\\
\end{tabular}
\caption{Example values for the Lyndon basis on two elements. The first column shows the Lyndon words, which are the Hall words for this basis. For each Lyndon word $h$, we show element $P_h$ of the Hall basis which is also the PBW basis element labelled by $h$. Next we show the corresponding element $S_h$ of the dual PBW basis, which also serves as $S_h$ described above. Finally we show the corresponding coordinate of the second kind.}
\label{tab:example}
\end{table}

\begin{table}
    \centering
  \newcommand{\specialcell}[2][c]{\begin{tabular}[#1]{@{}l@{}}#2\end{tabular}}
\begin{tabular}{llll}
\shortstack{Lyndon\\word $h$} & $P_h$ & $S_h$ & $\zeta_h=\pi_1^\top S_h$ \\
\hline
$1$     & $\word{1}$                 & $\word{1}$                     & $\word{1}$ \\
$12$    & $[\word{1},\word{2}]$             & $\word{12}$                    & $\tfrac{1}{2} \word{12} -\tfrac{1}{2} \word{21}$ \\
$112$   & $[\word{1},[\word{1},\word{2}]]$         & $\word{112}$               & $\tfrac{1}{6} [\word{112} -2 \word{121} + \word{211}]$ \\
$122$   & $[[\word{1},\word{2}],\word{2}]$         & $\word{122}$                   & $\tfrac{1}{6} [\word{122} -2 \word{212} + \word{221}]$ \\
$123$   & $[\word{1},[\word{2},\word{3}]]$         & $\word{123}$                   & $\tfrac{1}{6}[2\word{123}-\word{132}-\word{213}-\word{231}-\word{312}+2\word{321}]$ \\
$132$   & $[[\word{1},\word{3}],\word{2}]$         & $\word{123}+\word{132}$        & $\tfrac{1}{6}[\word{123}+\word{132}-2\word{213}+\word{231}-2\word{312}+\word{321}]$ \\
$1123$   & $[\word{1},[\word{1},[\word{2},\word{3}]]]$         & $\word{1123}$                   & $\frac{1}{6}[\tiny\begin{array}{@{}l@{}} \word{1123}-\word{1213}-\word{1231}\\\quad+\word{1321}+\word{3121}-\word{3211} \end{array}]$\\
$1132$   & $[\word{1},[[\word{1},\word{3}],\word{2}]]$         & $\word{1123}+\word{1132}$ & $\frac{1}{6}[\tiny\begin{array}{@{}l@{}} \word{1123}+\word{1132}-\word{1213}-\word{1312}\\\quad+\word{2131}-\word{2311}+\word{3121}-\word{3211} \end{array}]$\\
$1213$   & $[[\word{1},\word{2}],[\word{1},\word{3}]]$         & $\word{1123}+\word{1132}+\word{1213}$ & $\frac{1}{6}[\tiny\begin{array}{@{}l@{}} \word{1213}-\word{1312}-\word{2113}+\word{2131}+\word{3112}-\word{3121} \end{array}]$

\end{tabular}
\caption{Example values for the Lyndon basis on three elements. The first column shows the Lyndon words, which are the Hall words for this basis. For each Lyndon word $h$, we show element $P_h$ of the Hall basis which is also the PBW basis element labelled by $h$. Next we show the corresponding element $S_h$ of the dual PBW basis, which also serves as $S_h$ described above. Finally we show the corresponding coordinate of the second kind.}
\label{tab:example3}
\end{table}

  \begin{table}
    \centering
    \newcommand{\specialcell}[2][c]{\begin{tabular}[#1]{@{}l@{}}#2\end{tabular}}
\begin{tabular}{llll}
\shortstack{Hall\\word $h$} & $P_h$ & $S_h$ & $\zeta_h=\pi_1^\top S_h$ \\
\hline
$1$     & $\word{1}$                 & $\word{1}$                     & $\word{1}$ \\
$2$     & $\word{2}$                 & $\word{2}$                     & $\word{2}$ \\
$12$    & $[\word{1},\word{2}]$             & $\word{12}$                    & $\tfrac{1}{2} \word{12} -\tfrac{1}{2} \word{21}$ \\
$121$   & $[[\word{1},\word{2}],\word{1}]$         & $\word{112}+\word{121}$   & $\tfrac{1}{6} [2 \word{121} -\word{112} - \word{211}]$ \\
$122$   & $[[\word{1},\word{2}],\word{2}]$         & $\word{122}$                   & $\tfrac{1}{6} [\word{122} + \word{221}-2 \word{212} ]$ \\
$1211$  & $[[[\word{1},\word{2}],\word{1}],\word{1}]$     & $\word{1112}+\word{1121}+\word{1211}$ & $\tfrac{1}{6}[ \word{1211}- \word{1121}]$ \\
$1221$  & $[[[\word{1},\word{2}],\word{2}],\word{1}]$     & $\word{1122}+\word{1212}+\word{1221}$             & $\tfrac{1}{6}[\word{1212} -\word{1122}-\word{2121} + \word{2211}]$ \\
$1222$  & $[[[\word{1},\word{2}],\word{2}],\word{2}]$     & $\word{1222}$                  & $\tfrac{1}{6} [\word{2212} - \word{2122}]$ \\
$12111$  & $[[[[\word{1},\word{2}],\word{1}],\word{1}],\word{1}]$     & $\begin{array}[t]{@{}l}\word{11112}+\word{11121}\\\quad+\word{11211}+\word{12111}\end{array}$ & $\tfrac{1}{30}[\tiny\begin{array}{@{}l@{}}\word{11112}+\word{11121}-4\word{11211}+\word{12111}+\word{21111}\end{array}]$ \\
$12211$  & $[[[[\word{1},\word{2}],\word{2}],\word{1}],\word{1}]$     & $\begin{array}[t]{@{}l}\word{11122}+\word{11212}+\word{11221}\\\quad+\word{12112}+\word{12121}+\word{12211}\end{array}$ & $\tfrac{1}{30}[ \tiny\begin{array}{@{}l@{}}2\word{11122}-3\word{11212}-3\word{11221}+2\word{121112}+2\word{12121}\\\quad-3\word{12211}+2\word{21112}+2\word{21121}-3\word{21211}+2\word{22111}\end{array}]$ \\
$12221$  & $[[[[\word{1},\word{2}],\word{2}],\word{2}],\word{1}]$     & $\begin{array}[t]{@{}l}\word{11222}+\word{12122}\\\quad+\word{12212}+\word{12221}\end{array}$ & $\tfrac{1}{30}[ \tiny\begin{array}{@{}l@{}}-2\word{11222}+3\word{11112}-2\word{12212}-2\word{12221}+3\word{21122}\\\quad-2\word{21212}-2\word{21221}+3\word{22112}+3\word{22121}-2\word{22211}\end{array}]$ \\
$12222$  & $[[[[\word{1},\word{2}],\word{2}],\word{2}],\word{2}]$     & $\word{12222}$ & $\tfrac{1}{30}[\tiny\begin{array}{@{}l@{}} -\word{12222}-\word{21222}+4\word{22122}-\word{22212}-\word{22221}\end{array}]$ \\
$12112$  & $[[[\word{1},\word{2}],\word{1}],[\word{1},\word{2}]]$     & $\begin{array}[t]{@{}l}4\word{11122}+3\word{11212}+2\word{11221}\\\quad+2\word{12112}+\word{12121}\end{array}$ & $\tfrac{1}{30}[\tiny\begin{array}{@{}l@{}}-\word{11122}-\word{11212}-\word{11221}+4\word{12112}-\word{12121}\\\quad-\word{12211}-\word{21112}+4\word{21121}-\word{21211}-\word{22111}\end{array}]$ \\
$12112$  & $[[[\word{1},\word{2}],\word{2}],[\word{1},\word{2}]]$     & $3\word{11222}+2\word{12122}+\word{12212}$ & $\tfrac{1}{30}[\tiny\begin{array}{@{}l@{}}-3\word{11222}+2\word{12122}+2\word{12212}-3\word{12221}+2\word{21122}\\\quad-3\word{21212}+2\word{21221}+2\word{22112}+2\word{22121}-3\word{22211}\end{array}]$ \\

\end{tabular}
\caption{Example values for the standard Hall basis on two elements. The first column shows the Hall words. For each Hall word $h$, we show element $P_h$ of the Hall basis which is also the PBW basis element labelled by $h$. Next we show the corresponding element $S_h$ of the dual PBW basis, which also serves as $S_h$ described above. Finally we show the corresponding coordinate of the second kind.}
\label{tab:exampleHall}
\end{table}

\end{example}

The expressions given by \eqref{eq:pi1starCoordinates}
can become quite unwieldy.
This motivated Rocha to look for more tractable expressions in \cite{bib:Roc2003}.
We will now reproduce his results using purely algebraic arguments.

%
%
%

\subsection{Coordinates of first kind in terms of areas-of-areas}

As in Remark \ref{rem:mathcalR} we consider the grouplike element $S \in \TSTCRing$.
The goal is to find a ``simple expression'' for
\begin{align*}
  \Lambda := \log_{\shuffleConcat} S.
\end{align*}

Following Rocha, 
we obtain

\begin{equation*}
 R=\lift{r}(S)=(\lift{r}\circ\exp_{\shuffleConcat})[\Lambda].
\end{equation*}

%
\newcommand{\RHSLie}{RHS-Lie}
%
%
%

The last step consists now in inverting $\lift{r} \circ \exp_{\shuffleConcat}$ here.
We shall need the following version of Baker's identity \cite[(1.6.5)]{bib:Reu1993}.
\begin{lemma}
  \label{lem:PQ}
  Let $x,q\in\TC$ (resp. $L, Q \in \TSTCRing$) with $q$ (resp. $Q$) having no coefficient in the empty word $\emptyWord$ (resp. $\lift{\emptyWord}$)
  and $x$ (resp. $L$) primitive. Then
  \begin{equation*}
    r(x\conc q)=[x,r(q)],
    \qquad
    \lift r( L\shuffleConcat Q )
    =
    [ L, \lift r(Q) ]_{\shuffleConcatSymbol}.
  \end{equation*}
\end{lemma}
\begin{proof}
  For $x$, $L$ Lie, by \cite[Theorem 1.4]{bib:Reu1993}, $\ad_x=\Ad_x$ on $\TC$ and $\ad_L = \Ad_L$ on $\TSTCRing$.
  Hence for $q\in\TC$ and $Q\in\TSTCRing$ any polynomial having no coefficient in the empty word,
  \begin{align*}
    r[x\conc q]
    &=\ad_x r[q]
    =
    \Ad_x r[q]
    =[x,r[q]],\\
    \lift{r}[ L \shuffleConcat Q ]
    &=
    \ad_L \lift r[Q]
    =
    \Ad_L \lift r[Q]
    =
    [L, \lift r[Q] ]_{\shuffleConcatSymbol}. \qedhere
  \end{align*}
\end{proof}

We denote by $[.,.]_\shuffleConcatSymbol$ the Lie bracket on $\TSTC$ coming from the product $\shuffleConcat$. \index{[]@$[.,.]_\shuffleConcatSymbol$}
Note that
\begin{align*}
  [p\otimes p', q\otimes q']_\shuffleConcatSymbol = (p \shuffle q) \otimes [p',q'].
\end{align*}
\begin{remark}
  This is the Lie structure for the pre-Lie structure $\leftPrelie$ (\Cref{rem:preLie}),
  i.e. 
  \begin{equation*}
  [x,y]_\shuffleConcatSymbol = x\shuffleConcat y-y\shuffleConcat x=x \leftPrelie y - y \leftPrelie x=x\succeq y-y\succeq x-y\preceq x+x\preceq y.
  \end{equation*}
\end{remark}
For $x \in \TSTC$ denote by $\ad_{\shuffleConcat;x}$ the corresponding adjunction operator,\index{ads@$\ad_{\shuffleConcat;x}$}
i.e. $\ad_{\shuffleConcat;x} y := \left[ x, y \right]_{\shuffleConcat}$.

Let $\Lambda \in \TSTCRing$ be primitive.
Then, using Lemma \ref{lem:PQ},
\begin{align*}
  \lift{r}\left( \Lambda^{ \shuffleConcat n} \right)
  &=
  [ \Lambda, \lift{r}\left(\Lambda^{n-1}\right) ]_\shuffleConcatSymbol.
\end{align*}
Iterating this, we get
\begin{align*}
  \lift{r}( \Lambda^{\shuffleConcat n} )
  &= 
  (\ad_{\shuffleConcat;\Lambda})^{n-1} \lift{D}\Lambda.
\end{align*}
Hence
\begin{align}
  R
  =
  \lift{r}[ \exp_{\shuffleConcat}( \Lambda ) ]
  =
  \lift{r}\left[ \sum_{n \ge 0} \frac{\Lambda^{\shuffleConcat n}}{n!} \right]
  =
  \sum_{n \ge 1} \frac{(\ad_{\shuffleConcat;\Lambda})^{n-1} }{n!} \lift{D} \Lambda. \label{eq:hence}
\end{align}

This can now be used to recursively construct $\Lambda$ from $R$. Put 
\begin{equation*}
\leftmultilie x_1,\ldots, x_n\rightmultilie:=[x_1,[\ldots,[x_{n-1},x_n]\ldots]_{\shuffleConcatSymbol},\quad\leftmultilie x_1,x_2\rightmultilie:=[x_1,x_2]_{\shuffleConcatSymbol}, \quad \leftmultilie x\rightmultilie:=x.
\end{equation*}

\begin{proposition}
 We have $\Lambda_1=R_1$ and
 \begin{equation*}
  \Lambda_n=\frac{1}{n}R_n-\frac{1}{n}\sum_{i=2}^n\frac{1}{i!}\sum_{\substack{n_1,\ldots,n_i\\n_1+\cdots+n_i=n}}n_i\,\leftmultilie \Lambda_{n_1},\ldots,\Lambda_{n_i}\rightmultilie.
 \end{equation*}
\end{proposition}
\begin{proof}
 Rewriting Equation \eqref{eq:hence} for the homogeneous part $R_n$ yields
 \begin{align*}
  R_n&=\lift{D}\Lambda_n+\frac{1}{2}\sum_{m=1}^{n-1}[\Lambda_m,\lift{D}\Lambda_{n-m}]_{\shuffleConcatSymbol}+\sum_{i=3}^n \frac{1}{i!}\sum_{\substack{n_1,\ldots,n_1\\n_1+\cdots+n_1=n}} \leftmultilie \Lambda_{n_1},\ldots,\Lambda_{n_{i-1}},\lift{D}\Lambda_{n_i}\rightmultilie\\
  &=n\Lambda_n+\sum_{m=1}^{n-1}(n-m)[\Lambda_m,\Lambda_{n-m}]_{\shuffleConcatSymbol}+\sum_{i=3}^n \frac{1}{i!}\sum_{\substack{n_1,\ldots,n_1\\n_1+\cdots+n_1=n}} n_i\leftmultilie \Lambda_{n_1},\ldots,\Lambda_{n_{i-1}},\Lambda_{n_i}\rightmultilie\\
  &=n\Lambda_n+\sum_{i=2}^n \frac{1}{i!}\sum_{\substack{n_1,\ldots,n_1\\n_1+\cdots+n_1=n}} n_i\leftmultilie \Lambda_{n_1},\ldots,\Lambda_{n_i}\rightmultilie,
 \end{align*}
 which shows the claim.
\end{proof}

\begin{example}
  \label{ex:letUsSpell}
  Let us spell out the first few summands of \eqref{eq:hence},
  \begin{align*}
    R =
    \lift{D} \Lambda
    + \frac{1}{2!} \left[ \Lambda, \lift{D} \Lambda \right]_\shuffleConcatSymbol
    + \frac{1}{3!} \left[ \Lambda, \left[ \Lambda, \lift{D} \Lambda \right]_\shuffleConcatSymbol \right]_\shuffleConcatSymbol
    + \frac{1}{4!} \left[ \Lambda, \left[ \Lambda, \left[ \Lambda, \lift{D} \Lambda \right]_\shuffleConcatSymbol \right]_\shuffleConcatSymbol \right]_\shuffleConcatSymbol
    + ..
  \end{align*}

  Level by level (remember that $\Lambda_0 = 0$), we see
  \begin{align*}
    R_1&=\Lambda_1 \\
    R_2&=2 \Lambda_2 + \tfrac{1}{2!} [\Lambda_1,\Lambda_1]_\shuffleConcatSymbol\\
    R_3&=3 \Lambda_3
    +
    \tfrac{1}{2!} \big( [\Lambda_1, 2 \Lambda_2]_\shuffleConcatSymbol + [\Lambda_2, \Lambda_1]_\shuffleConcatSymbol \big)
    +
    \tfrac{1}{3!} [\Lambda_1, [\Lambda_1, \Lambda_1 ]_\shuffleConcatSymbol]_\shuffleConcatSymbol\\
    R_4&=4 \Lambda_4
    +
    \tfrac{1}{2!}
    \big(
      [\Lambda_1, 3 \Lambda_3]_\shuffleConcatSymbol
      +
      [\Lambda_2, 2 \Lambda_2]_\shuffleConcatSymbol
      +
      [\Lambda_3, \Lambda_1]_\shuffleConcatSymbol \big)\\
    &\quad\,\,\,\,+
    \tfrac{1}{3!}
    \big(
      [\Lambda_1, [\Lambda_1,2 \Lambda_2]_\shuffleConcatSymbol]_\shuffleConcatSymbol
      +
      [\Lambda_1, [\Lambda_2,\Lambda_1]_\shuffleConcatSymbol]_\shuffleConcatSymbol
      +
      [\Lambda_2, [\Lambda_1,\Lambda_1]_\shuffleConcatSymbol]_\shuffleConcatSymbol
    \big)\\
    &\quad\,\,\,\,+
    \tfrac{1}{4!}
    [\Lambda_1, [\Lambda_1, [\Lambda_1, \Lambda_1 ]_\shuffleConcatSymbol]_\shuffleConcatSymbol]_\shuffleConcatSymbol
  \end{align*}




  Plugging in the expressions from Example \ref{ex:R} in for $R$,
  we get for $d=2$
  \begin{align*}
    \Lambda_1 &= R_1 = \word{1} \otimes \word{1} + \word{2} \otimes \word{2} \\
    \Lambda_2 &= \tfrac{1}{2} R_2
              = \tfrac{1}{4} \big( \area(\word{1},\word{2})\otimes [\word{1},\word{2}] + \area(\word{2},\word{1})\otimes[\word{2},\word{1}] \big)
              = \tfrac{1}{2} \area(\word{1},\word{2})\otimes[\word{1},\word{2}]  \\
    \Lambda_3 &= \tfrac{1}{3}
                \left( R_3
                  - \tfrac{1}{2} [\Lambda_1, 2 \Lambda_2 ]_\shuffleConcatSymbol
                  - \tfrac{1}{2} [\Lambda_2, \Lambda_1 ]_\shuffleConcatSymbol
                  - \tfrac{1}{3} [\Lambda_1, [\Lambda_1, \Lambda_1]_\shuffleConcatSymbol]_\shuffleConcatSymbol \right)
              = \tfrac{1}{3}
                \left( R_3 - \tfrac{1}{2} [\Lambda_1, \Lambda_2 ]_\shuffleConcatSymbol \right) \\
              &=
              \tfrac{1}{6}
              \area(\word{1},\area(\word{1},\word{2}))
              \otimes
              [\word{1},[\word{1},\word{2}]]
              +
              \tfrac{1}{6}
              \area(\word{2},\area(\word{1},\word{2}))
              \otimes
              [\word{2},[\word{1},\word{2}]] \\
              &\qquad
              -
              \tfrac{1}{12}
              \left( \word{1} \shuffle \area(\word{1},\word{2}) \right)
              \otimes
              [\word{1},[\word{1},\word{2}]]
              -
              \tfrac{1}{12}
              \left( \word{2} \shuffle \area(\word{1},\word{2}) \right)
              \otimes
              [\word{2},[\word{1},\word{2}]] \\
    \Lambda_4 &= 
    \tfrac{1}{4}
    \bigl\{ R_4
    -
    \tfrac{1}{2!}
    \big(
      [\Lambda_1, 3 \Lambda_3]_\shuffleConcatSymbol
      +
      [\Lambda_2, 2 \Lambda_2]_\shuffleConcatSymbol
      +
      [\Lambda_3, \Lambda_1]_\shuffleConcatSymbol \big) \\
              &\qquad
    -
    \tfrac{1}{3!}
    \big(
      [\Lambda_1, [\Lambda_1,2 \Lambda_2]_\shuffleConcatSymbol]_\shuffleConcatSymbol
      +
      [\Lambda_1, [\Lambda_2,\Lambda_1]_\shuffleConcatSymbol]_\shuffleConcatSymbol
      +
      [\Lambda_2, [\Lambda_1,\Lambda_1]_\shuffleConcatSymbol]_\shuffleConcatSymbol
    \big) \\
    &\qquad
    -
    \tfrac{1}{4!}
    [\Lambda_1, [\Lambda_1, [\Lambda_1, \Lambda_1 ]_\shuffleConcatSymbol]_\shuffleConcatSymbol]_\shuffleConcatSymbol \bigr\} \\
    &= 
    \tfrac{1}{4}
    \bigl( R_4
    -
    [\Lambda_1, \Lambda_3]_\shuffleConcatSymbol
    -
    \tfrac{1}{3!} [\Lambda_1, [\Lambda_1,\Lambda_2]_\shuffleConcatSymbol]_\shuffleConcatSymbol
    \bigr) \\
    &=
    \scalebox{0.7}{$\tfrac{1}{24}
    \Big(
        \area( \word1, \area( \word1, \area(\word1,\word2))) \otimes [\word1,[\word1,[\word1,\word2]]]
        +
      \area( \word1, \area( \word2, \area(\word1,\word2))) \otimes [\word1,[\word2,[\word1,\word2]]] $} \\
    &\qquad\qquad
    \scalebox{0.7}{
        $+
        \area( \word2, \area( \word1, \area(\word1,\word2))) \otimes [\word2,[\word1,[\word1,\word2]]]
        +
      \area( \word2, \area( \word2, \area(\word1,\word2))) \otimes [\word2,[\word2,[\word1,\word2]]]$}
      \\
    &\qquad\qquad
    \scalebox{0.7}{
      $-
        (\word1 \shuffle \area(\word1, \area(\word1,\word2))) \otimes [\word1,[\word1,[\word1,\word2]]]
        -
      (\word2 \shuffle \area(\word1, \area(\word1,\word2))) \otimes [\word2,[\word1,[\word1,\word2]]]$} \\
    &\qquad\qquad 
    \scalebox{0.7}{
      $-
        (\word1 \shuffle \area(\word2, \area(\word1,\word2))) \otimes [\word1,[\word2,[\word1,\word2]]]
        -
      (\word2 \shuffle \area(\word2, \area(\word1,\word2))) \otimes [\word2,[\word2,[\word1,\word2]]]
    \Big)$}
  \end{align*}
\end{example}
\begin{remark}
  Comparing with \cite[p.322]{bib:Roc2003} we note that we correct some of the coefficients appearing in $\Lambda_3$ and $\Lambda_4$ there.
\end{remark}

\begin{definition}

  Let $\widetilde\binaryPlanarTrees_n$ be binary planar trees, \index{binaryPlanarTreesw@$\widetilde\binaryPlanarTrees_n$}
  with two types of inner nodes, $\bindot$ and $\binsqu$,
  and such that the subset of all $\binsqu$ nodes
  is either empty or forms a subtree with the same root as the tree itself. In other words the square nodes are all connected to the root.

  Define $e: \widetilde\binaryPlanarTrees_n \to \R$ as follows.
  If the root of $\tau$ is $\bindot$, then
  \begin{align*}
    e(\tau) := \frac{1}{n c(\tau)},
  \end{align*}
  where $c$ was defined in Lemma \ref{lem:RwithTrees}.
  Otherwise, we can write $\tau$ uniquely as
  \begin{align*}
    \tau &=\tikz[bintrees]{\node[squ]{}child{node{$\tau^{(1)}$}}
           child{node[squ]{}
             child{node{$\tau^{(2)}$}}
               child{node{$\cdot$}child{node[xshift=0.4ex]{$\cdot$}
                   child{node[squ,xshift=0.4ex]{}[sibling distance = 1.5em]
                   child{node{$\tau^{(\ell-1)}$}}
                   child{node{$\tau^{(\ell)}$}}
               }}}}}
               = (\tau^{(1)} \rightarrow_\shuffleConcatSymbol (\tau^{(2)} \rightarrow_\shuffleConcatSymbol (\dots \rightarrow_\shuffleConcatSymbol (\tau^{(\ell-1)} \rightarrow_\shuffleConcatSymbol \tau^{(\ell)} )))),
  \end{align*}
  for some $\ell = \ell(\tau) \ge 2$, $\tau^{(1)}, \dots, \tau^{(\ell-1)} \in \widetilde\binaryPlanarTrees$
  and $\tau^{(\ell)} \in \binaryPlanarTrees$.
  Here $\sigma \rightarrow_\shuffleConcatSymbol \rho$ is the grafting, to a new root of type $\shuffleConcatSymbol$,
  with $\sigma$ on the left and $\rho$ on the right.
  Then 
  \begin{align*}
    e(\tau) := - \sum_{j=2}^{\ell(\tau)} \frac{\leaves{ \tau^{(\ge j)} } e( \tau^{(\ge j)} )}{j! \leaves{\tau}} \left( \prod_{i=1}^{j-1} e(\tau^{(i)}) \right) ,
  \end{align*}
  where
  \begin{align*}
    \tau^{(\ge j)} &:=
    \tikz[bintrees]{\node[squ]{}child{node{$\tau^{(j)}$}}
           child{node[squ]{}[sibling distance = 1.2em]
             child{node{$\tau^{(j+1)}$}}
               child{node{$\cdot$}child{node[xshift=0.4ex]{$\cdot$}
                   child{node[squ,xshift=0.4ex]{}[sibling distance = 1.5em]
                     child{node{$\tau^{(\ell-1)}$}}
                     child{node{$\tau^{(\ell)}$}}
               }}}}}
               =(\tau^{(j)} \rightarrow_\shuffleConcatSymbol (\tau^{(j+1)} \rightarrow_\shuffleConcatSymbol (\dots \rightarrow_\shuffleConcatSymbol (\tau^{(\ell-1)} \rightarrow_\shuffleConcatSymbol \tau^{(\ell)})))), \qquad j=1,\dots,\ell-1 \\
               \tau^{(\ge \ell)} &:= \tau^{(\ell)}.
  \end{align*}

  Finally, for a tree $\tau \in \widetilde\binaryPlanarTrees_n$ and a word $w$ of length $n$,
  define $\widetilde\areatree(\tau)$ as \index{areasw@$\widetilde\areatree$}
  bracketing out using $\area$
  if a node of type $\bullet$ is encountered
  and multiplying using $\shuffle$ when a node $\shuffleConcatSymbol$ is encountered.
\end{definition}

\begin{example}
  The trees in $\widetilde\binaryPlanarTrees_2$ are
  \begin{align*}
    \tikz[bintrees]{\node [dot]{}child{node {$\word i$}}child{node {$\word j$}}},
    \tikz[bintrees]{\node [squ]{}child{node {$\word i$}}child{node {$\word j$}}}
  \end{align*}
  for letters $\word i$ and $\word j$, and the trees in $\widetilde\binaryPlanarTrees_3$ are
  \begin{align*}
    \tikz[bintrees]{\node[dot]{}child{node {$\word i$}}child{node [dot]{}child{node {$\word j$}}child{node {$\word k$}}}},
    \tikz[bintrees]{\node[dot]{}child{node [dot]{}child{node {$\word i$}}child{node {$\word j$}}}child{node {$\word k$}}},
    \tikz[bintrees]{\node[squ]{}child{node {$\word i$}}child{node [dot]{}child{node {$\word j$}}child{node {$\word k$}}}},
    \tikz[bintrees]{\node[squ]{}child{node [dot]{}child{node {$\word i$}}child{node {$\word j$}}}child{node {$\word k$}}},
    \tikz[bintrees]{\node[squ]{}child{node {$\word i$}}child{node [squ]{}child{node {$\word j$}}child{node {$\word k$}}}},
    \tikz[bintrees]{\node[squ]{}child{node [squ]{}child{node {$\word i$}}child{node {$\word j$}}}child{node {$\word k$}}},
  \end{align*}
  for letters $\word i$, $\word j$ and $\word k$.

\newcommand{\squaretwothree}{\tikz[bintrees]{\node [squ]{}child{node {$\word 2$}}child{node {$\word 3$}}}}
  We have
  \begin{align*}
    e( \word{2} )
    &=
    c( \word{2} ) = 1 \\
    e (\squaretwothree)
      &=
      - \frac{1}{2\cdot 2} e( \word{2} ) \leaves{ \word3 } e( \word{3} )
      =
      -\frac{1}{4} \\
    e( \tikz[bintrees]{\node[squ]{}child{node {$\word 1$}}child{node [squ]{}child{node {$\word 2$}}child{node {$\word 3$}}}} )
    &=
    - \frac{1}{3} \left( \frac{1}{2!} e( \word1 ) \leaves{ \squaretwothree } e( \squaretwothree ) + \frac{1}{3!} e( \word2 ) e( \word2 ) \leaves{\word3} e( \word3) \right)
    =
    - \frac{1}{3}
    \left(
      - \frac{1}{4} + \frac{1}{6}
    \right) = \frac{1}{36}.
  \end{align*}

  And
  \begin{align*}
    \widetilde\areatree(\tikz[bintrees]{\node[squ]{}child{node {$\word 1$}}child{node [dot]{}child{node {$\word 2$}}child{node {$\word 3$}}}} )
    =
    \word1 \shuffle \area(\word2,\word3).
  \end{align*}
\end{example}

\begin{theorem}
  \label{thm:firstExpressionForLambda}
  Then
  \begin{align*}
    \Lambda_n = \sum_{\tau \in \widetilde\binaryPlanarTrees_n} e(\tau)\ \widetilde\areatree(\tau) \otimes \brackettree(\tau )
    \,\in\quadsymlie:=\gen{\quadsym}{[\cdot,\cdot]_\shuffleConcatSymbol}.
  \end{align*}
\end{theorem}
\begin{proof}
 Define
 \begin{align*}
   \tau\in\widetilde\binaryPlanarTrees_{n;\geq i} 
   &:=
   \{ \tau \in \widetilde\binaryPlanarTrees_{n} \mid \ell(\tau) \ge i \},\\
   \boldsymbol{\lbrack} x_1,\ldots, x_n\boldsymbol{\rbrack}&:=[x_1,[\ldots,[x_{n-1},x_n]\ldots],\quad\boldsymbol{\lbrack} x_1,x_2\boldsymbol{\rbrack}:=[x_1,x_2], \quad \boldsymbol{\lbrack} x\boldsymbol{\rbrack}:=x.
   \end{align*}
 Due to $e(\word{i})=c(\word{i})=1$ for all letters $\word{i}$, we have 
 \begin{equation*}
  \Lambda_1=R_1=\sum_{\word{i}=\word{1}}^{\word{d}}\word{i}\otimes\word{i}=\sum_{\tau\in\binaryPlanarTrees_1}e(\tau)\widetilde\areatree(\word{i})\otimes\brackettree(\word{i}),
 \end{equation*}
 and then via induction over $n$
 \begin{align*}
  &\Lambda_n=\frac{1}{n}R_n-\frac{1}{n}\sum_{i=2}^n\frac{1}{i!}\sum_{\substack{n_1,\ldots,n_i\\n_1+\cdots+n_i=n}}n_i\,\leftmultilie \Lambda_{n_1},\ldots,\Lambda_{n_i}\rightmultilie\\
  &=\sum_{\tau\in\binaryPlanarTrees_n}\frac{1}{nc(\tau)}\areatree(\tau)\otimes\brackettree(\tau)\\
  &\quad-\frac{1}{n}\sum_{i=2}^n\frac{1}{i!}\sum_{\substack{n_1,\ldots,n_i\\n_1+\cdots+n_i=n}}\sum_{\substack{\tau_1,\ldots \tau_i\\\tau_j\in\widetilde\binaryPlanarTrees_{n_j}}}\leaves{\tau_i}\prod_{j=1}^{i}e(\tau_j)
  \widetilde\areatree(\tau_1)\shuffle\dots\shuffle\widetilde\areatree(\tau_i)\otimes\boldsymbol{\lbrack}\brackettree(\tau_1),\ldots,\brackettree(\tau_i)\boldsymbol{\rbrack}\\
  &=\sum_{\tau\in\binaryPlanarTrees_n}e(\tau)\,\areatree(\tau)\otimes\brackettree(\tau)\\
  &\quad-\frac{1}{n}\sum_{i=2}^n\frac{1}{i!}\sum_{\tau\in\widetilde\binaryPlanarTrees_{n;\geq i}}\leaves{\tau^{(\geq i)}} e(\tau^{(\geq i)})\prod_{j=1}^{i-1} e(\tau^{(j)})\,\widetilde\areatree(\tau)\otimes\brackettree(\tau)\\
  &=\sum_{\tau\in\binaryPlanarTrees_n}e(\tau)\,\areatree(\tau)\otimes\brackettree(\tau)-\sum_{\tau\in\widetilde\binaryPlanarTrees_n^{\geq 2}}\sum_{i=2}^{\ell(\tau)}\frac{\leaves{\tau^{(\geq i)}} e(\tau^{(\geq i)})}{i! \leaves{\tau}}\prod_{j=1}^{i-1}e(\tau^{(k)})\,\widetilde\areatree(\tau)\otimes\brackettree(\tau)\\
  &=\sum_{\tau\in\widetilde\binaryPlanarTrees_n}e(\tau)\,\widetilde\areatree(\tau)\otimes\brackettree(\tau).
  \qedhere
 \end{align*}
\end{proof}

\begin{remark}
\NEXTPAPER{Rosa: Is there a $\zeta_h$ such that $\zeta_h\otimes P_h\notin\quaddend$?}

  Recall, from \Cref{cor:Rrecursion},
  \begin{align*}
    \quadsym = \gen{\word{i}\otimes\word{i},\,\word{i}=\word{1}\ldots\word{d}}{\leftPrelieSym}.
  \end{align*}
  Define
  \begin{align*}
    \quadprelie &:=\gen{\word{i}\otimes\word{i},\,\word{i}=\word{1}\ldots\word{d}}{\leftPrelie} \\
    \quaddend &:=\gen{\word{i}\otimes\word{i},\,\word{i}=\word{1}\ldots\word{d}}{\succeq,\preceq}.
  \end{align*}

 Then, we have $S_n,R_n,\Lambda_n\in\quaddend$, and the chain of inclusions
 \begin{equation*}
  \quadsym\subsetneq\quadsymlie
  \subseteq\quadprelie
  \subsetneq\quaddend.
  \end{equation*}

  Indeed, the mere inclusions are clear since $\leftPrelieSym$ and $[\cdot,\cdot]_{\shuffleConcatSymbol}$ are symmetrization and antisymmetrization of $\leftPrelie$, and $\leftPrelie$ itself is defined as a combination of $\succeq$ and $\preceq$. Regarding the strictness of two of the inclusions, on the one hand for any $d\geq 2$, the only anagram axis of $\word{12}\otimes\word{12}$ contained in $\quadsym$ is spanned by
  \begin{equation*}
   (\word{1}\otimes\word{1})\leftPrelieSym(\word{2}\otimes\word{2})=(\word{2}\otimes\word{2})\leftPrelieSym(\word{1}\otimes\word{1})=\area(\word{1},\word{2})\otimes[\word{1},\word{2}]=(\word{12}-\word{21})\otimes(\word{12}-\word{21}),
  \end{equation*}
  and thus the $\quadsymlie$ element
  \begin{equation*}
   [\word{1}\otimes\word{1},\word{2}\otimes\word{2}]_{\shuffleConcatSymbol}=(\word{1}\shuffle\word{2})\otimes[\word{1},\word{2}]=(\word{12}+\word{21})\otimes(\word{12}-\word{21})
  \end{equation*}
  is not contained in $\quadsym$. On the other hand, the anagram space of $\word{12}\otimes\word{12}$ in $\quadprelie$ is spanned by the two vectors
  \begin{align*}
   (\word{1}\otimes\word{1})\leftPrelie(\word{2}\otimes\word{2})&=(\word{1}\hs\word{2})\otimes[\word{1},\word{2}]=\word{12}\otimes(\word{12}-\word{21}),\\
   (\word{2}\otimes\word{2})\leftPrelie(\word{1}\otimes\word{1})&=(\word{2}\hs\word{1})\otimes[\word{2},\word{1}]=-\word{21}\otimes(\word{12}-\word{21}),
  \end{align*}
  and is thus easily seen to not contain the $\quaddend$ element
  \begin{equation*}
   (\word{1}\otimes\word{1})\succeq(\word{2}\otimes\word{2})=(\word{1}\hs\word{2})\otimes(\word{1}\conc\word{2})=\word{12}\otimes\word{12}.
  \end{equation*}
  However, it remains an open problem whether $\quadsymlie$ and $\quadprelie$ conincide.
  
  Finally, we note the inclusion $\quaddend\subseteq\anagramdend$, where $(\anagramdend,\succeq,\preceq)$ is the dendriform algebra with linear basis given by all $w\otimes v$ such that $w$ is a word and $v$ is an anagram of $w$, and leave as a further question for future work whether $\quaddend$ and $\anagramdend$ actually conincide.
\end{remark}


Since the expansion in this theorem is not in terms of a \emph{basis} of the Lie algebra,
these are not yet coordinates of the first kind.
But, by a straightforward projection procedure we get
\begin{corollary}
  \label{cor:coordinatesOfFirstKind}\todonotes{RP: example}
  \begin{align*}
    S = \exp_\shuffleConcatSymbol \left( \Lambda \right),
  \end{align*}
  with
  \begin{align*}
    \Lambda = \sum_h \zeta_h \otimes P_h,
  \end{align*}
  where $h$ runs over Hall words,
  $P_h$ are the corresponding Lie Hall basis elements,
  and the
  $\zeta_h$ are expressed as linear combinations of shuffles of areas-of-areas,
  \begin{equation*}
   \zeta_h=\sum_{\substack{\tau \in \widetilde\binaryPlanarTrees_{|h|},\\ \text{foliage of }\tau\in\anagrams(h)}}e(\tau)
   \Big\langle S_h,\brackettree(\tau)\Big\rangle\,\widetilde\areatree(\tau).
  \end{equation*}

\end{corollary}
\begin{remark}
  Again, this result is not satisfying because the $\zeta_h$ are expensive to calculate due to the large number of summands, which are not even linearly independent.
  We mention it only for completeness.
\end{remark}
\begin{proof}[Proof of Corollary \ref{cor:coordinatesOfFirstKind}]
  Let $P_h$ be Lie basis and $S_h$ its dual basis.
  Then
  \begin{align*}
    \Lambda
    &= \sum_h \sum_{\tau \in \widetilde\binaryPlanarTrees_n} e(\tau)\ \widetilde\areatree(\tau) \langle S_h, \brackettree(\tau) \rangle \otimes P_h \\
    &=: \sum_h \zeta_h \otimes P_h,
  \end{align*}
  where $S_h$ can be expressed as an element of $T(\R^d)$ which is a linear combination of anagrams of $h$, thus $\langle S_h,\brackettree(\tau)\rangle=0$ if the foliage of $\tau$ is not an anagram of $h$.
\end{proof}

\section{Shuffle generators}
\label{sec:shuffleGenerators}

For a countable index set $I$ consider the free commutative algebra $\R[ x_i : i \in I]$
over the indeterminates $x_i, i \in I$
(\cite[Definition 1.2.12]{bib:Row1988}). If $V$ is a vector space with a countable basis,
we also write $\R[ V ]$ for $\R[ x_i : i \in I ]$ where $I$ is some basis of $V$.
A commutative algebra $\mathcal A$ is \textbf{generated} by some elements $z_i \in \mathcal A$, $i \in I$,
if the commutative algebra morphism
\begin{align*}
  \R[ x_i : i \in I ] \to \mathcal A,
\end{align*}
extended from $x_i \mapsto z_i$, is surjective. If it is also injective, the algebra is \textbf{freely generated}
by the elements $z_i$.
The goal of this section is to find a simple condition
on a countable family $z_i \in \TS, i \in I$, to be (freely) generating.

Before stating the general results, let us begin with the example of the image of $\rho$.

\begin{proposition}
 Any basis for the image of $\Im\rho$ is generating. More explicitly, for any non-empty word $w$, we have
 \begin{equation}\label{eq:wordsasrhoshuffles}
  w=\sum_{\substack{w_1,\ldots,w_n\\ w_1\cdots w_n=w}}\frac{1}{k_{|w_1|,\ldots,|w_n|}}\,\rho(w_1)\shuffle\cdots\shuffle\rho(w_n),
 \end{equation}
 where $k_{m_1,\ldots,m_n}=(m_1+\cdots+m_n)k_{m_2,\ldots,m_n}$, with $k_{m}=m$.
\end{proposition}
\begin{proof}
 For any letter $\word{i}$, we have $\word{i}=\rho(\word{i})$ in accordance with Equation \eqref{eq:wordsasrhoshuffles}. Assume the equation holds for all non-empty words $v$ with $|v|\leq\ell$ for some $\ell\geq 1$, and let $w$ be a word with $|w|=\ell+1$. Then, by Equation \eqref{eq:thm112} we have
 \begin{align*}
  |w|w&=Dw=\sum_{uv=w}\rho(u)\shuffle v=\sum_{uv=w}\rho(u)\shuffle \sum_{\substack{v_1,\ldots,v_n\\v_1\cdots v_n=v}}\frac{1}{k_{|v_1|,\ldots,|v_n|}}\,\rho(v_1)\shuffle\ldots\shuffle\rho(v_n)\\
  &=\sum_{\substack{w_1,\ldots,w_n\\w_1\cdots w_n=w}}\frac{1}{k_{|w_2|,\ldots,|w_n|}}\,\rho(w_1)\shuffle\ldots\shuffle\rho(w_n),
 \end{align*}
 again in accordance with Equation \eqref{eq:wordsasrhoshuffles}. 
 
 Note that in order for the induction to work, we made use again of the fact that $\rho(\emptyWord)=0$, so we only sum over non-empty words.
\end{proof}


\begin{lemma}
  \label{lem:lieAlgebra}
  For each $n \ge 1$, let $X_n \subset T_n(\R^d)$ be a subset of the shuffle algebra at level $n$.
  Let $X \coloneqq \bigcup_{n\ge 1} X_n$.
  Then:
  \begin{center}
    For all $n\ge 1$, for all nonzero $L \in \mathfrak g_n$ there is an $x \in X_n$
    such that $\langle x, L \rangle \not= 0$\\
    \vspace{0.5em}
    \emph{if and only if} \\
    \vspace{0.5em}
    $X$ generates the shuffle algebra $T(\R^d)$.
  \end{center}

  If moreover $|X_n| = \dim \mathfrak g_n$, $n\ge 1$, then $X$ is \emph{freely} generating.
\end{lemma}
\begin{remark}
  \label{rem:milnorMoore}
  This lemma can also be seen as a consequence of (the proof of) the Milnor-Moore theorem,
  see for example \cite[p.48]{bib:Car2007}.
  Let us sketch this.
  Let $\TSgradedDual$ be the graded dual of $\TS$,
  the subspace of $\TC$ consisting of only finite linear combinations of words.
  Endowed with the unshuffle coproduct, the dual of the shuffle product, this is a cocommutative, conilpotent coalgebra.
  Then, by (the proof of) \cite[Theorem 3.8.1]{bib:Car2007},
  there exists an isomorphism of cocommutative coalgebras
  \begin{align*}
    e_{\TSgradedDual}: \symmetricTensors[ \primitive ] \to \TSgradedDual.
  \end{align*}
  Here $\symmetricTensors[ \primitive ] \subset T( \primitive )$ are the symmetric tensors over $\primitive$,
  generated, as a vector space, by the elements $\underbrace{v \otimes \dots \otimes v}_{n \text{ times}}$, $v \in \mathfrak g$, $n\ge 0$, and endowed with the deconcatenation coproduct.
  The map $e_\TSgradedDual$ acts on these elements as
  \begin{align*}
    e_\TSgradedDual\left( \underbrace{v \otimes \dots \otimes v}_{n \text{ times}} \right) = \frac{v^n}{n!},
  \end{align*}
  where the $n$-th power on the right hand side is taken with respect to the concatenation product (under which $\TSgradedDual$ is closed).
  The grading on $\TSgradedDual$ induces a grading on $\symmetricTensors[\primitive]$ via the isomorphism $e_\TSgradedDual$.
  The graded dual (with respect to this induced grading) of $\symmetricTensors[\primitive]$ is then given by $\symmetricAlgebra[ \mathfrak g^\gradedDual ]$, i.e.~the symmetric algebra over $\primitive^\gradedDual$,
  where $\mathfrak g^\gradedDual$ is the graded dual of $\mathfrak g$.
  Since $e_\TSgradedDual$ is an isomorphism of cocommutative coalgebras, the dual map
  \begin{align*}
    e_\TSgradedDual^\gradedDual: \TS \to \symmetricAlgebra[ \mathfrak g^\gradedDual ],
  \end{align*}
  is an isomorphism of commutative algebras.
  $X$ (freely) generating $\TS$ is then equivalent to $e_\TSgradedDual^\gradedDual\left( X \right)$ (freely) generating $\R[ \mathfrak g^\gradedDual ]$,
  which is equivalent to our condition, using Lemma \ref{lem:triangularGenerating}.

\end{remark}
\newcommand\shuffleFromBelow{\mathsf{shuff}}
\begin{proof}
  We show for every level $N$:

  \begin{center}
    $\forall n \le N$ $\forall\ 0 \not= L \in \mathfrak g_n$ there is $x \in X_n$ with $\langle x, L \rangle \not= 0$ \\
    \vspace{0.5em}
    \emph{if and only if} \\
    \vspace{0.5em}
    ${\bigcup}_{1 \le n\le N} X_n$ shuffle generates $T_{\le N}(\R^d)$.
  \end{center}

  It is clearly true for $N=1$.
  Let it be true for some $N$.
  We show it for $N+1$.

  Let $\shuffleFromBelow_{N+1} \subset T_{N+1}(\R^d)$   \index{shuff@$\shuffleFromBelow$}
  denote the linear space of shuffles of everything ``from below'', i.e.
  \begin{align*}
    \shuffleFromBelow_{N+1} := \bigcup_{n=1}^{N} \left\{ T_n(\R^d) \shuffle T_{N-n}(\R^d) \right\}.
  \end{align*}
  By \cite[Theorem 3.1 (iv)]{bib:Reu1993}
  \begin{align*}
    \Big\langle \shuffleFromBelow_{N+1}, L \Big\rangle = 0,
  \end{align*}
  for all $L \in \primitive_{N+1}$.
  In other words, $\shuffleFromBelow_{N+1}$ is contained in the annihilator of $\mathfrak{g}_{N+1}$.
  By \cite[Theorem 6.1]{bib:Reu1993},
  the shuffle algebra is freely generated by the Lyndon words in $\word1, \dots, \word{d}$,
  which have dimension $\dim \primitive_n$ on level $n$.
  Hence
  \begin{align*}
    \dim \shuffleFromBelow_{N+1} = \dim T_{N+1}(\R^d) - \dim \primitive_{N+1}.
  \end{align*}
  By dimension counting we hence have that 
  $\shuffleFromBelow_{N+1}$ must actually be \emph{equal} to the annihilator of $\mathfrak{g}_{N+1}$.
  Then, a fortiori, $\mathfrak{g}_{N+1}$ is the annihilator of $\shuffleFromBelow_{N+1}$.

  By Lemma \ref{lem:obvious},
  \begin{center}
    $T_{N+1}(\R^d) = \shuffleFromBelow_{N+1} + \spann_\R X_{N+1}$\\
    \vspace{0.5em}
    \emph{if and only if}\\
    \vspace{0.5em}
    $\forall\ 0 \not= L \in \mathfrak g_{N+1}$ there is $x \in \spann_\R X_{N+1}$ with $\langle x, L \rangle \not= 0$.
  \end{center}
  But this is the case if and only if $\forall\ 0 \not= L \in \mathfrak g_{N+1}$ there is $x \in X_{N+1}$ with $\langle x, L \rangle \not= 0$.
  This finishes the proof regarding the generating property.

  Regarding freeness:
  denote $\iota: \R[ x_v: v \in X ] \to T(\R^d)$ the
  extension, as a commutative algebra morphism, of the map $x_v \mapsto v$.
  Denote $\iota_{Lyndon}: \R[ y_w : w \in L ] \to T(\R^d)$ the
  extension, as a commutative algebra morphism, of the map $y_w \mapsto w$,
  where $L$ are the Lyndon words. By \cite[Theorem 6.1]{bib:Reu1993}, $\iota_{Lyndon}$ is an isomorphism.
  By what we have shown so far, $\iota$ is surjective.
  Since $X$ consists of homogeneous elements, we can grade $\R[ x_v: v \in X ]$
  induced from the grading of $T(\R^d)$
  and analogously for $\R[ y_w : w \in L ]$.
  By assumption, the graded dimensions match.
  Hence, there is an isomorphism of graded, commutative algebras
  \begin{align*}
    \Phi: \R[ x_v : v \in X ] \to \R[ y_w : w \in L ].
  \end{align*}
  Since $\iota_{Lyndon}$ is an isomorphism of graded, commutative algebras
  and $\iota$ is epimorphism of graded, commutative algebras (where each homogeneous subspace is finite dimensional!)
  we must have that $\iota$ is in fact an isomorphism.
\end{proof}
We used the following simple lemma.
\begin{lemma}
  \label{lem:obvious}
  Let $V$ be a finite dimensional vector space with dual $W:=V^*$.
  We denote the pairing by $\langle w, v\rangle$, for $w \in W, v \in V$.
  Let $W_1, W_2$ be subspaces of $W$ and
  let 
  \begin{align*}
    W_1^\bot :=  \{ v \in V : \langle w_1, v \rangle = 0\ \forall w_1 \in W_1 \},
  \end{align*}
  be the annihilator of $W_1$.
  Then: 
  \begin{center}
  $\forall\ 0 \not= v_1 \in W_1^\bot$ there is $w_2 \in W_2$ with $\langle w_2, v_1 \rangle \not= 0$\\
    \vspace{0.5em}
  \emph{if and only if}\\
    \vspace{0.5em}
  $W_1 + W_2 = W$.
  \end{center}
\end{lemma}
\begin{proof}
  Recall the well-known identity (\cite[Exercise 17.8.c)]{bib:Hal2017})
  \begin{align*}
    (W_1 + W_2)^\bot = W_1^\bot \cap W_2^\bot. 
  \end{align*}
  Then
  \begin{align*}
    W_1 + W_2 = W \Leftrightarrow W_1^\bot \cap W_2^\bot = \{ 0 \},
  \end{align*}
  which is the claim.
\end{proof}

\begin{corollary}
  \label{cor:someShuffleGenerators}
  Let $X$ be a set of homogeneous elements of $T(\R^d)$. 
  Then, the following are equivalent:
  \begin{enumerate}[(i)]
  \item $X$ freely shuffle generates $T(\R^d)$,
  \item $X$ is a homogeneous realization of a  dual basis to a homogeneus basis of $\mathfrak{g}$,
  \item $X$ is a homogeneous basis for the image of a projection $\pi^\top$, where $\pi$ is a graded projection $\pi: \TC \to \mathfrak g \subset \TC$ \footnote{Identifying $\mathfrak g$ as a subset of $\TC$.}.
  \end{enumerate}

  Examples include:
  \begin{enumerate}
    \item $\pi := \pi_1$, the Eulerian idempotent \eqref{eq:eulerianIdempotent}
      \vspace{-0.2em}
      \begin{itemize}
        \item[$\leadsto$] Coordinates of the first kind.
      \end{itemize}
    \item A rescaling of the Dynkin map $r$ \eqref{eq:r} (to make it a projection)
      \vspace{-0.2em}
      \begin{itemize}
        \item[$\leadsto$] A basis for the image of $\rho$, for example $\Rho_h$ from \Cref{thm:Rho},
          which by Corollary \ref{cor:Rrecursion} can be expressed as areas-of-areas.
      \end{itemize}
    \item $\pi$ the orthogonal projection (with respect to the inner product in the ambient space $\TC$) onto $\mathfrak g$
      (the Garsia idempotent, \cite{bib:Duc1991})
      \vspace{-0.2em}
      \begin{itemize}
        \item[$\leadsto$] Any (homogeneous) basis for the Lie algebra $\mathfrak g \subset \TC$, identified as elements of $\TS$.
      \end{itemize}
      \todonotes{JR: Include coordinates of the second kind; what's the span of $S_h$, Does it shuffle generate?  }
  \end{enumerate}
\end{corollary}
\begin{remark}
  1.
  Point 3. is shown in \cite[Section 6.5.1]{bib:Reu1993}.
  We include it here, as it falls nicely into the setting of Lemma \ref{lem:lieAlgebra}.

  2.
  Coordinates of the first kind must - by definition - contain all the information
  of the signature, so it is reasonable that they shuffle generate $\TS$.
  For the other sets this is not immediately evident.
  The basis for the Lie algebra is one such example
  and it does not even live in the correct space (formally, it is an element of the concatenation algebra $\TC$ not of the shuffle algebra $\TS$).

\end{remark}
\begin{proof}
  
  (iii)$\Rightarrow$(i): Assume first that we have given a graded projection $\pi$ with image $\mathfrak{g}$, and a homogeneous basis $(x_i)_i$ of $\Im\pi^\top$.
  Because of the grading it makes sense to speak of the component
  $\pi_n: T_n((\R^d)) \to T_n((\R^d))$.
  Then $\pi^\top_n: T_n(\R^d) \to T_n(\R^d)$ and
  \begin{align*}
    \im\left( \pi_n^\top \right)^\bot = \ker\left( \pi_n \right).
  \end{align*}
  Since $\pi_n$ itself is also a projection, we have that
  \begin{align*}
    T_n(\R^d) = \ker( \pi_n ) \oplus \im( \pi_n ).
  \end{align*}
  Hence, for every $L \in \im( \pi )$ there is $x \in \im\left( \pi^\top \right)$
  with $\langle x, L \rangle \not= 0$. Then Lemma \ref{lem:lieAlgebra} applies. \todonotes{JD/JR/RP: explain more}
  
  (i)$\Rightarrow$(ii): Let now $X$ be a homogeneous free shuffle generating set. 
  Then $\{\langle x,\cdot\rangle|x\in X_n\}$ spans the whole dual space of $\mathfrak{g}_n$. 
  Indeed, assume this is not the case, then a comparison with some $\R^n$ shows that there is a nonempty annihilator of $X_n$ inside $\mathfrak{g}_n$, but this contradicts the criterion from Lemma \ref{lem:lieAlgebra}. 
  Hence, $X$ does span the dual space of $\mathfrak{g}$, thus contains a dual basis to some basis of $\mathfrak{g}$, and since $X$ is freely generating, $X$ is actually that dual basis 
  (otherwise, a subset of $X$ would already generate, which contradicts the assumption that $X$ freely generates).
  
  (ii)$\Rightarrow$(iii): Let $P_h \in \TC, h \in H$, be some homogeneous basis for the Lie algebra.
  Let $D_h \in \TS, h \in H$ be a realization of a dual basis.
  That is
  \begin{align*}
    \Big\langle D_h, P_{h'} \Big\rangle = \delta_{h, h'}.
  \end{align*}
  Then choose $\pi$ such that $\ker \pi = \left( \operatorname{span}_\R \{ D_h : h \in H \} \right)^\top$.
  \todonotes{JD/RP/JR: explain more}
\end{proof}
\begin{proposition}
 If $X$ is a homogeneous set and $\pi: \TC \to \mathfrak g \subset \TC$ is a graded projection, then $\pi^\top X$ is a shuffle generating set if and only if $\pi^\top X$ spans $\Im\pi^\top$ if and only if $X$ is a shuffle generating set. If $X$ is freely generating, then so is $\pi^\top X$.
\end{proposition}
\begin{proof}
 Since for any $x\in X$ and $p\in\mathfrak{g}$ we have
 \begin{equation*}
  \langle\pi^\top x,p\rangle=\langle x,\pi p\rangle=\langle x,p\rangle,
 \end{equation*}
 the condition for being a shuffle generating set in Lemma \ref{lem:lieAlgebra} is fullfilled for $X$ if and only if it is fullfilled for $\pi^\top X$. 
 Since any basis of $\Im\pi^\top$ is a free and thus also a minimal shuffle generating set by Corollary \ref{cor:someShuffleGenerators}, 
 $\pi^\top X\subseteq\Im\pi^\top$ shuffle generates if and only if it linearly spans $\Im\pi^\top$. 
 If $X$ freely shuffle generates, than the shuffle generating set $\pi^\top X$ must also have minimal dimension for each homogeneity, and thus freely generate due to the freeness of the shuffle algebra.
\end{proof}

Point 3.2 in Corollary \ref{cor:someShuffleGenerators} proves, using Corollary \ref{cor:Rrecursion}, what we set out to prove:
areas-of-areas do shuffle generate $\TS$.
\begin{corollary}
  \label{cor:areasShuffleGenerate}
  The set $\areatortkara$ of the Introduction is a generating set for $\TS$.
  A free generating set is given e.g.\ by
  any basis for the image of $\rho$.
\end{corollary}
\begin{proof}

We give three proofs. 

\textbf{Via $\Lambda$ and coordinates of the first kind}


Using $Dw=\sum_{uv=w,\,|u|\geq 1}\rho(u)\shuffle v$,  $\ad_P=\Ad_P$ for any Lie polynomial $P$ and $\sum_{w} w\otimes r(w)=\sum_{w} \rho(w) \otimes w$, we showed $(\lift{D}-\id)R=R\leftPrelie R$ (Lemma \ref{lem:R}).

We then have $R\in\quadsym$ due to $R_1=\sum_{\word{i}=\word{1}}^\word{d} \word{i}\otimes\word{i}$ and $R_n=\frac{1}{2(n-1)}\sum_{l=1}^{n-1}R_l\leftPrelieSym R_{n-l}$ (Corollary \ref{cor:Rrecursion}).

Using $R=\lift{r}\exp_{\shuffleConcatSymbol}\Lambda$ and Baker's identity for $\lift{r}$ (Lemma \ref{lem:PQ}), we show that $\Lambda_n$ is a linear combination of $R_n$ and $\shuffleConcatSymbol$-Lie-bracketings of lower order $\Lambda_i$. Thus $\Lambda\in\quadsymlie$, and together with the fact $\zeta_h=\coeval^{S_h}(\Lambda)$ following from the definition of $\Lambda$ we conclude that each $\zeta_h$ is a linear combination of shuffles of areas of areas.

\textbf{Via $R$ and $\rho$}

From $R\in\quadsym$ we conclude that the image of $\rho$ lies in $\areatortkara$, since $\rho(v)=\coeval^{v}(R)$ for any word $v$. Since the image of $\rho$ shuffle generates the shuffle algebra, Point 3.2 in Corollary \ref{cor:someShuffleGenerators}, so does $\areatortkara$.

\textbf{Via \cite{bib:DIM2018} and $\rho$}

Via a combinatorial expression for $\rho(\word{1}\word{2}\ldots\word{d})$ for any $d$ (Proposition \ref{prop:rhogroupalgebra}) we conclude (Corollary \ref{cor:rhoimageareas}) that the image of $\rho$ is a subspace of
\begin{equation*}
 \spann_\R \{ \word{i} : \word{i} \text{ a letter }\} \oplus \spann_\R \{ w (\word{ij} - \word{ji}) : w \text{ a word}, \word{i},\word{j} \text{ letters} \},
\end{equation*}
which is nothing but $\areatortkara$, according to \cite{bib:DIM2018}. Again, since $\rho$ shuffle generates the shuffle algebra, so do areas of areas.

\end{proof}

\begin{remark}\label{rem:areasShuffleGenerate}
  Corollary~\ref{cor:areasShuffleGenerate} is an a priori stronger statement than the following easy-to-prove statement, with which it is occasionally confused.
  \begin{center}
    \emph{(A) Any word is a linear combination of shuffles of letters and $\area$s of arbitrary words.}
  \end{center}
  An illustration of (A) is as follows.
  \begin{align*}
    \word{123}&=(\word1\hs\word2)\hs\word3=\tfrac12\{\word1\shuffle\word2+\area(\word1,\word2)\}\hs\word3
    \\&=\tfrac14\big[\{\word1\shuffle\word2+\area(\word1,\word2)\}\shuffle\word3+\area(\{\word1\shuffle\word2+\area(\word1,\word2)\},\word3)\big]
    \\&=\tfrac14\big[\word1\shuffle\word2\shuffle\word3+\area(\word1,\word2)\shuffle\word3+\area(\word1\shuffle\word2,\word3)
    +\area(\area(\word1,\word2),\word3)\big]
  \end{align*}
  Corollary~\ref{cor:areasShuffleGenerate} implies that this can be done with all the shuffles \emph{outside} all the $\area$s, namely
  \begin{center}
    \emph{(B) Any word is a linear combination of shuffles of letters and iterated $\area$s of letters.}
  \end{center}
For example
  \begin{align*}
    \word{123}&=\tfrac13\area(\word1,\area(\word2,\word3))+\tfrac16\area(\area(\word1,\word3),\word2)
    +\tfrac13\word1\shuffle\area(\word2,\word3)
    \\&\quad-\tfrac16\word2\shuffle\area(\word1,\word3)+\tfrac12\word3\shuffle\area(\word1,\word2)
    +\tfrac16\word1\shuffle\word2\shuffle\word3.
  \end{align*}
\end{remark}

\section{Applications}
\label{sec:applications}

The antisymmetrizing feature of the area operation
leads to pleasant properties
for piecewise linear paths and semimartingales.

\subsection{Piecewise linear paths: computational aspects}
\label{sec:computational}

\newcommand\discreteArea{\mathsf{DiscreteArea}}
For two time series $a_0, \dots, a_n, b_0, \dots, b_n \in \R$
define the new time series
\begin{align*}
  \discreteArea(a,b)_\ell
  &:= 
  \operatorname{Corr}_1(a,b)_\ell
  - 
  \operatorname{Corr}_1(b,a)_\ell \\
  &:=
  \sum_{i=0}^{\ell-1} a_{i+1} b_i
  -
  \sum_{i=0}^{\ell-1} b_{i+1} a_i,
  \quad \ell=0, \dots, n,
\end{align*}
set to be $0$ for $\ell=0$.
It is known (\cite[Section 3.2]{bib:DR2018}),
that for a piecewise linear curve $X$ through the points $0, x_1, \dots, x_n \in\R^2$, one has
\begin{align}
  \label{eq:oneHas}
  \Big\langle \area(\word1,\word2), S(X)_{0,n} \Big\rangle
  =
  \discreteArea\left(x^1, x^2\right)_n.
\end{align}

We will show that this iterates nicely.

\begin{lemma}\label{lem:piecewiseLinear}
  If $X,Y$ are piecewise linear then $\Area(X,Y)$ is piecewise linear.
\end{lemma}
\begin{proof}
  \begin{align*}
    \frac{d}{dt} \Area(X,Y)_t
    &= \int_0^t dX_r\,\dot Y_t - \int_0^t dY_r\,\dot X_t \\
    \frac{d^2}{dt^2} \Area(X,Y)_t
    &= \int_0^t dX_r\,\ddot Y_t + \dot X_t \dot Y_t - \int_0^t dY_r\,\ddot X_t - \dot Y_t \dot X_t \\
    &= \int_0^t dX_r\,\ddot Y_t - \int_0^t dY_r\,\ddot X_t 
    = 0,
  \end{align*}
  since $X,Y$ are piecewise linear.
  Hence, $\Area(X,Y)$ is indeed piecewise linear.
\end{proof}

In particular, for $\phi \in \areatortkara$ (defined in the Introduction) and $X$ piecewise linear
\begin{align*}
  t \mapsto \Big\langle \phi, S(X)_{0,t} \Big\rangle,
\end{align*}
is piecewise linear.
Note that by Lemma \ref{lem:AequalsA}, $\phi$ can be written as
linear combination of elements of the form $w (ij - ji)$.
One can also see directly that such elements yield something piecewise linear:
\NEXTPAPER{Rosa: inequality for $C^2$ paths}
\begin{align*}
  &\frac{d^2}{dt^2}
  \Big\langle w (ij -ji), S(X)_{0,t} \Big\rangle\\
  &\quad=
  \frac{d^2}{dt^2}
  \left\{
  \int_0^t \int_0^s \Big\langle w, S(X)_{0,r} \Big\rangle dX^{(i)}_r dX^{(j)}_s
  -
  \int_0^t \int_0^s \Big\langle w, S(X)_{0,r} \Big\rangle dX^{(j)}_r dX^{(i)}_s
  \right\} \\
  &\quad=
  \int_0^t \Big\langle w, S(X)_{0,r} \Big\rangle dX^{(i)}_r\,\ddot X^{(j)}_t
  +
  \Big\langle w, S(X)_{0,t} \Big\rangle \dot X^{(i)}_t \dot X^{(j)}_t\\
  &\qquad
  -
  \int_0^t \Big\langle w, S(X)_{0,r} \Big\rangle dX^{(j)}_r\,\ddot X^{(i)}_t
  -
  \Big\langle w, S(X)_{0,t} \Big\rangle \dot X^{(i)}_t \dot X^{(j)}_t\\
  &\quad= 0.
\end{align*}

\begin{lemma}\label{lem:notrivialsignaturecomps}
 For all nonzero $z\in T(\mathbb{R}^d)$, there is a piecewise linear path $X$ such that $\langle z,S(X)\rangle\neq 0$.
\end{lemma}
\begin{proof}
 Let $z\in T(\mathbb{R}^d)\setminus\{0\}$ be arbitrary and let $n$ be its degree (the length of the longest word in the word-expansion of $z$). Since $G_{\leq n}:=\proj_{\leq n} G$ spans $T_{\leq n}(\mathbb{R}^d)$ (see e.g.\ \cite[Lemma~8]{bib:DR2018}), there are $g_1,\dots,g_k\in G$ and $r_1,\dots,r_k\in\R$ such that 
 \begin{equation*}
  \langle z,r_1g_1+\dots+r_kg_k\rangle=\langle z,z\rangle\neq0,
 \end{equation*}
 and hence there is $g_i\in G$ such that $\langle z,g_i\rangle\neq 0$. Now, due to Chow's theorem according to \cite[Theorem~7.28]{bib:FV2010}, there is a piecewise linear $X$ such that $\proj_{\leq n}g_i=\proj_{\leq n}S(X)$, which implies $\langle z,S(X)\rangle=\langle z,g_i\rangle\neq0$.
\end{proof}

\begin{theorem}
 $\langle S(X)_{0,t},\phi\rangle$ is piecewise linear for all piecewise linear paths $X$ if and only if $\phi\in \mathbb{R}\oplus \areatortkara$.
\end{theorem}
\begin{proof}
 We already showed in Lemma \ref{lem:piecewiseLinear} that for piecewise linear $X$, $\langle S(X)_{0,t},\phi\rangle$ is again piecewise linear for all $\phi\in \areatortkara$. Since the whole tensor space $T(\mathbb{R}^d)=\mathbb{R}\oplus \areatortkara\oplus B$, where $B$ is $\spann_{\mathbb{R}}\{w\word{i}\word{j},\text{ $w$ a word, $\word{i}\leq\word{j}$ letters}\}$, and since the sum of a function which is not piecewise linear with a piecewise linear function is again not piecewise linear, it only remains to show that for any $b\in B\setminus\{0\}$, there is a piecewise linear $X$ such that $t\mapsto\langle b,S(X)_{0,t}\rangle$ is not piecewise linear.
 
 To this end, let $b=\sum_{i\leq j} d_{ij}\word{i}\word{j}\in B\setminus\{0\}$ be arbitrary. If there is a letter $\word{l}$ such that $d_{ll}\neq 0$ (case 1), choose a piecewise linear path $X:\,[0,2]\to\mathbb{R}^d$ such that $\langle d_{ll},S(X)_{0,1}\rangle\neq 0$ and such that $X{\restriction_{[1,2]}}$ is linear with $x_l=1$ and $x_i=0$ for $i\neq l$, where $x_i:=\dot X^i_{3/2}$. Otherwise, since $b$ is nonzero, there are letters $\word{k}<\word{l}$ such that $d_{kl}\neq 0$ (case 2), and in this case, choose a piecewise linear path $X:\,[0,2]\to\mathbb{R}^d$ such that $\langle d_{kl},S(X)_{0,1}\rangle\neq 0$ and such that $X{\restriction_{[1,2]}}$ is linear with $x_k=1, x_l=1$ and $x_i=0$ for $i\notin\{k,l\}$, where $x_i:=\dot X^i_{3/2}$. In both cases, such a piecewise linear $X$ exists due to \Cref{lem:notrivialsignaturecomps}.
 
 Since $X_{[1,2]}$ is linear, we have for arbitrary $z\in T(\mathbb{R}^d$ that $t\mapsto\langle z,S(X)_{0,t}\rangle$ is polynomial on $[1,2]$, and thus arbitraryly often continuously differentiable on $(1,2)$. Thus, since $\ddot{X}=0$ and $\dot{X}$ constant on $(1,2)$, we have
 \begin{equation*}
 \lim_{t\searrow 1}\frac{d^2}{dt^2}
  \Big\langle\sum_{i\leq j} d_{ij} \word{i}\word{j}, S(X)_{0,t} \Big\rangle=\sum_{i\leq j}\Big\langle d_{ij},S(X)_{0,1}\Big\rangle x_i x_j=\begin{cases}\langle d_{ll},S(X)_{0,1}\rangle\neq0, &\text{case 1}\\\langle d_{kl},S(X)_{0,1}\rangle\neq 0,&\text{case 2}\end{cases}
 \end{equation*}
In both cases, we conclude that $t\mapsto\langle b,S(X)_{0,t}\rangle$ is not piecewise linear on any interval $[1,s]$, $1<s\leq2$, which finishes the proof.
\end{proof}

The fact that ``being linear'' is preserved under the $\Area$-operation immediately
leads to the following theorem.
\begin{theorem}
  \label{thm:piecewiseLinear}
  Let $X$ in $\R^\ds$ be a piecewise linear
  curve through the points $0, x_1, \dots, x_n \in\R^\ds$.
  Then: for every tree $\tau$,
  \begin{align*}
    \Big\langle \areatree(\tau), S(X)_{0,n} \Big\rangle
    =
    \discreteAreatree\left( \tau, x \right)_n.
  \end{align*}
  Here, $\areatree$ is defined in Lemma \ref{lem:RwithTrees}
  and $\discreteAreatree$ is defined similarly,
  as iterated bracketing using the $\discreteArea$-operator.
\end{theorem}
\begin{example}
  For $\tau = \tikz[bintrees]{\node[dot]{}child{node [dot]{}child{node {$\word 1$}}child{node {$\word 2$}}}child{node {$\word 3$}}}$ the statement reads as
  \begin{align*}
    \Big\langle \areatree\left( \tikz[bintrees]{\node[dot]{}child{node [dot]{}child{node {$\word 1$}}child{node {$\word 2$}}}child{node {$\word 3$}}} \right), S(X)_{0,n} \Big\rangle
    &:=
    \Big\langle \area\left( \area\left( \word1, \word2\right), \word3 \right), S(X)_{0,n} \Big\rangle \\
    &=\\
    \discreteAreatree\left( \tikz[bintrees]{\node[dot]{}child{node [dot]{}child{node {$\word 1$}}child{node {$\word 2$}}}child{node {$\word 3$}}}, x \right)_n
    &:=
    \discreteArea\left( \discreteArea\left( x^{(\word1)}, x^{(\word2)} \right), x^{(\word3)} \right)_n,
  \end{align*}
  which one can verify by a direct, but tedious, calculation.
\end{example}
\begin{remark}
  This is not obvious at all.
  Indeed, if we just look at the discrete integration operator (still assuming $x_0 = 0$)
  \newcommand\integrate{\mathsf{DiscreteIntegral}}
  \begin{align*}
    \Big\langle \word{12}, S(X)_{0,n} \Big\rangle
    &=
    \sum_{i=0}^{n-1} \frac{1}{2} \left( x^1_i + x^1_{i+1} \right) \left( x^2_{i+1}-x^2_i \right)
    =:
    \integrate\left(x^1,x^2\right)_n,
  \end{align*}
  this does \emph{not} iterate. Indeed,
  \begin{align*}
    \Big\langle \word{123}, S(X)_{0,n} \Big\rangle
    &= 
    \sum_{i=0}^{n-1}
    \left(
      \sum_{j=0}^{i-1} \frac{1}{2} \left( x^1_{j} + x^1_{j+1} \right) x^2_{j,j+1}
      +
      \left( \frac{1}{2} x^1_{i} + \frac{1}{3!} x^1_{i,i+1} \right)
      x^2_{i,i+1}
    \right)
    x^3_{i,i+1} \\
    &\not= \integrate\left( \integrate\left(x^1,x^2\right), x^3 \right)_n.
  \end{align*}

\end{remark}
\begin{proof}
  If $Y,Z$ are piecewise linear between the points $0,y_1,..$ and $0,z_1,\dots$, then
  $\Area(Y,Z)$ is piecewise linear between the points
  \begin{align*}
    0, \discreteArea(y,z)_1, .., \discreteArea(y,z)_n.
  \end{align*}
  We can hence iterate \eqref{eq:oneHas}.
\end{proof}

\subsection{Martingales}

\newcommand\dito{d_{\operatorname{It\bar{o}}}}
\newcommand\dstrat{d_{\operatorname{Strat}}}

Another pleasant property of the area-operation presents itself
when working with a continuous semimartingale $M$.
One has (see \cite[Chapter III]{bib:IW1988})
\begin{align}
  \label{eq:quadraticVariation}
  \int_0^T M^i_{0,r} \dstrat M^j_r
  =
  \int_0^T M^i_{0,r} \dito M^j_r
  +
  \frac{1}{2} [ M^i, M^j ]_T.
\end{align}
where $\dstrat$ denotes Stratonovich integration, $\dito$ denotes It\=o integration
and $[ ., . ]$ denotes the quadratic covariation.

\begin{proposition}
 Let $(\Omega,\mathcal{F},\mathcal{F}_t,\mathbb{P})$ be a filtered probability space.
 Let $M$ be a continuous,
 $\mathcal{F}_t$-martingale such that all iterated It\={o} integrals are martingales.
 Then $t \mapsto \langle\phi,S_{\operatorname{Strat}}(M)_{0,t}\rangle$ is an
 $\mathcal{F}_t$-martingale for all $\phi\in \areatortkara$.
\end{proposition}
\begin{proof}
 Let $X,Y$ be as in the statement. Assume for simplicity $X_0=Y_0=0$ almost surely.
 Then, using Equation \eqref{eq:quadraticVariation},
 \begin{equation*}
  \Area_{\operatorname{Strat}}(X,Y)_t:=\int_0^t X_r \dstrat Y_r-\int_0^T Y_r\dstrat X_r=\int_0^t X_r\dito Y_r -\int_0^t Y_r\dito X_r,
 \end{equation*}
 is again an $(\mathcal F_t)_t$-martingale.
 Hence for every $\phi\in \areatortkara$, $\langle\phi,S_{\operatorname{Strat}}(M)_{0,t}\rangle$ is a
 $\mathcal{F}_t$-martingale.
 %
\end{proof}

\begin{proposition}
 Let $M$ be the piecewise linear interpolation of a time discrete martingale whose moments all exist.
 Then $t \mapsto \langle\phi,S(M)_{0,t}\rangle$ is again the piecewise linear interpolation of a time discrete martingale.
\end{proposition}
\begin{proof}
 Let $(a_n)_n, (b_n)_n$ be two time discrete
 $L^2$ martingales for the filtration $\mathcal{F}_k$ whose moments all exist. Then,
 \begin{align*}
  &\mathbb{E}[\discreteArea(a,b)_{k+1}|\mathcal{F}_k]-\discreteArea(a,b)_k\\
  &\quad=\mathbb{E}[\discreteArea(a,b)_{k+1}-\discreteArea(a,b)_k|\mathcal{F}_k]=\mathbb{E}[a_{k+1}b_k-b_{k+1}a_k|\mathcal{F}_k]\\
  &\quad=b_k\mathbb{E}[a_{k+1}|\mathcal{F}_k]-a_k\mathbb{E}[b_{k+1}|\mathcal{F}_k]=b_k a_k-a_kb_k=0,
  \end{align*}
  thus $(\discreteArea(a,b)_n)_n$ is again an $(\mathcal{F}_n)_n$ martingale. If $(c_n)_n, (d_n)_n$ are $\tau_n$ local {\color{orange} $L^2$} martingales, then
  \begin{equation*}
   \discreteArea(a,b)^{\tau_k}=\discreteArea(a^{\tau_k},b^{\tau_k})^{\tau_k}=\discreteArea(a^{\tau_k},b^{\tau_k})
  \end{equation*}
  is a martingale for any $k$ and thus $\discreteArea(a,b)$ is a $\tau_k$ local martingale.
\end{proof}


The previous results imply that for $M$ a martingale with all iterated integrals being in $L^1(\Omega)$, 
or for $M$ a linear interpolation of a time discrete martingale, all expectancies of areas of areas vanish for $M$. 
This naturally leads to the very interesting question of what is the class of all semimartingale paths 
such that all area expectancies vanish?

In particular, for all these paths,
we have, by \Cref{cor:rhoimageareas}, that the expected Stratonovich signature lies in the kernel of $r$,
\begin{align*}
   r(\E[ S_{\operatorname{Strat}}(M)_{0,T}])&=\E\left[ r\left( S_{\operatorname{Strat}}(M)_{0,T} \right) \right]=\sum_w\E[\langle w,S_{\operatorname{Strat}}(M)_{0,T}\rangle]\,r(w)\\
   &=\sum_w\E[\langle\rho(w),S_{\operatorname{Strat}}(M)_{0,T}\rangle]\,w=0.
\end{align*}
In fact this last property obviously holds for any Semimartingale with $\langle \rho(w),S_{\operatorname{Strat}}\rangle=0$ for any $w$, 
which is a priori a larger class than just those paths where the expectancy of all areas of areas vanishes since $\Im\rho$ does not linearly span $\areatortkara$.

\section{Linear span of area expressions}
\label{sec:linear}
\newcommand{\linearA}{\tilde \areatortkara}
\newcommand{\letters}{L}
\def\hssymbol{\mathbin{\succ}}
\newcommand{\hstwo}[2]{#1\hssymbol#2} 
\newcommand{\vol}{\mathsf{vol}}


In \cite{bib:Dzh2007}
the antisymmetric, non-associative operation $\area$ was studied in detail.
It was shown that
\begin{itemize}
  \item $\area$ does not satisfy any new identity of degree $3$;
    in particular it does \emph{not} satisfy the Jacobi identity.
  \item On degree $4$ there is exactly one new identity, the \emph{Tortkara identity}.
    Over a field of characteristic zero, we have the following equivalent formulations:
    \begin{align*}
     &\area(\area(a,b),\area(c,b))
     =\area(\vol(a,b,c),b),\\
      &\area( \area(a,b), \area(c,d) )
      +
      \area( \area(a,d), \area(c,b) )
      =
      \area( \vol(a,b,c), d)
      +
      \area( \vol(a,d,c), b),\\
    %
      &2\cdot \area(\area(a,b),\area(c,d))\\
      &=\area( \vol(a,b,c), d)
      +
      \area( \vol(a,d,c), b)
      +\area( \vol(b,a,d), c)
      +
      \area( \vol(b,c,d), a)     
    %
    %
    \end{align*}
     where $\vol(x,y,z) := \area(\area(x,y),z) + \area(\area(y,z),x) + \area(\area(z,x),y)$.
\end{itemize}

We chose the notation $\vol$ because $\langle\vol(u,v,w),S(X)_{0,T}\rangle$ is six times the signed volume
(\cite{bib:DR2018}) of the curve $(U,V,W)$, where
\begin{equation*}
  U_t = \Big\langle u, S(X)_{0,t} \Big\rangle, \quad
  V_t = \Big\langle v, S(X)_{0,t} \Big\rangle, \quad
  W_t = \Big\langle w, S(X)_{0,t} \Big\rangle.
\end{equation*}

The Tortkara identity is readily verified on all of $T^{\geq 1}(\R^d)$ by computing
\begin{align*}
 \area(\area(\word{1},\word{2}),\area(\word{3},\word{2}))&=-2\,\word{1223}+2\,\word{1232}+2\,\word{2213}-2\,\word{2231}-2\,\word{3212}+2\,\word{3221}\\
 &=\area(\vol(\word{1},\word{2},\word{3}),\word{2}),
\end{align*}
where
\begin{equation*}
 \vol(\word{1},\word{2},\word{3})=\word{123}-\word{132}-\word{213}+\word{231}+\word{312}-\word{321}.
\end{equation*}
Indeed, this computation suffices to show the Tortkara identity on $T^{\geq 1}(\R^d)$ due to the universal property of the free Zinbiel algebra $(T^{\geq 1}(\R^3),\hs)$, 
i.e.\ for any $a,b,c\in T^{\geq 1}(\R^d)$, 
there is a unique Zinbiel homomorphism $(T^{\geq 1}(\R^3),\hs)\to(T^{\geq 1}(\R^d),\hs)$ with $\word{1}\mapsto a$,  $\word{2}\mapsto b$,  $\word{3}\mapsto c$, 
and then the Tortkara identity follows for $a,b,c$ from the above computation by the homomorphism property
and the fact that $\area$ is nothing but the antisymmetrization of $\hs$.

\begin{proof}[Proof of equivalence of the tortkara identities]
 Let $\tortkara$ be a vector space over an arbitrary field with a bilinear antisymmetric operation $\tortop$ and 
 \begin{equation*}
  \tortjacobi(x,y,z) := \tortop(\tortop(x,y),z) + \tortop(\tortop(y,z),x) + \tortop(\tortop(z,x),y).
 \end{equation*}
 
 \begin{enumerate}
  \item Assume first that for all $x,y,z\in\tortkara$ we have
  \begin{equation*}
   \tortop(\tortop(x,y),\tortop(z,y))=\tortop(\tortjacobi(x,y,z),y).
  \end{equation*}
  Then, for all $a,b,c,d\in\tortkara$, due to bilinearity, we have
  \begin{align*}
   &\tortop(\tortop(a,b),\tortop(c,b))+\tortop(\tortop(a,b),\tortop(c,d))+\tortop(\tortop(a,d),\tortop(c,b))+\tortop(\tortop(a,d),\tortop(c,d))\\
   &=\tortop(\tortop(a,b+d),\tortop(c,b+d))=\tortop(\tortjacobi(a,b+d,c),b+d)\\
   &=\tortop(\tortjacobi(a,b,c),b)+
   \tortop(\tortjacobi(a,b,c),d)+
   \tortop(\tortjacobi(a,d,c),b)+
   \tortop(\tortjacobi(a,d,c),d).
  \end{align*}
  Since $\tortop(\tortop(a,b),\tortop(c,b))=\tortop(\tortjacobi(a,b,c),b)$ and $\tortop(\tortop(a,d),\tortop(c,d))=\tortop(\tortjacobi(a,d,c),d)$, we obtain the identity
  \begin{equation*}
   \tortop( \tortop(a,b), \tortop(c,d) )
      +
      \tortop( \tortop(a,d), \tortop(c,b) )
      =
      \tortop( \tortjacobi(a,b,c), d)
      +
      \tortop( \tortjacobi(a,d,c), b)
  \end{equation*}
  for all $a,b,c,d\in\tortkara$. Using antisymmetry, we furthermore get
  \begin{align*}
   &\tortop( \tortop(a,b), \tortop(c,d) )+\tortop( \tortop(a,b), \tortop(c,d) )\\
    &=\tortop( \tortop(a,b), \tortop(c,d) )
      +
      \tortop( \tortop(a,d), \tortop(c,b) )
      +  \tortop( \tortop(b,a), \tortop(d,c) )
      +
      \tortop( \tortop(b,c), \tortop(d,a) )\\
      &=\tortop( \tortjacobi(a,b,c), d)
      +
      \tortop( \tortjacobi(a,d,c), b)
      +
      \tortop( \tortjacobi(b,a,d), c)
      +
      \tortop( \tortjacobi(b,c,d), a)
  \end{align*}
  If the field is of characteristic different from two, this reads as
  \begin{equation*}
   2\tortop( \tortop(a,b), \tortop(c,d) )=\tortop( \tortjacobi(a,b,c), d)
      +
      \tortop( \tortjacobi(a,d,c), b)
      +
      \tortop( \tortjacobi(b,a,d), c)
      +
      \tortop( \tortjacobi(b,c,d), a)
  \end{equation*}
  for all $a,b,c,d\in\tortkara$.
  \item Assume now that for all $a,b,c,d\in\tortkara$ we have
  \begin{equation*}
   \tortop( \tortop(a,b), \tortop(c,d) )
      +
      \tortop( \tortop(a,d), \tortop(c,b) )
      =
      \tortop( \tortjacobi(a,b,c), d)
      +
      \tortop( \tortjacobi(a,d,c), b).
  \end{equation*}
  This immediately implies
  \begin{equation*}
   \tortop(\tortop(x,y),\tortop(z,y))+\tortop(\tortop(x,y),\tortop(z,y))=\tortop(\tortjacobi(x,y,z),y)+\tortop(\tortjacobi(x,y,z),y)
  \end{equation*}
  for all $x,y,z\in\tortkara$, which is an empty statement in characteristic two, but in characteristic different from two reduces to
  \begin{equation}\label{eq:tortkara}
   \tortop(\tortop(x,y),\tortop(z,y))=\tortop(\tortjacobi(x,y,z),y).
  \end{equation}
  \item For the last implication we want to show, assume that the characteristic of the underlying field is different from two and for all $a,b,c,d\in\tortkara$ we have
  \begin{equation*}
   2\tortop( \tortop(a,b), \tortop(c,d) )=\tortop( \tortjacobi(a,b,c), d)
      +
      \tortop( \tortjacobi(a,d,c), b)
      +
      \tortop( \tortjacobi(b,a,d), c)
      +
      \tortop( \tortjacobi(b,c,d), a).
  \end{equation*}
  This implies
  \begin{align*}
   2\tortop(\tortop(x,y),\tortop(y,z))=
   2\tortop( \tortjacobi(x,y,z), y)
   +\tortop( \tortjacobi(y,x,y), z)
      +
      \tortop( \tortjacobi(y,z,y), x)=
   2\tortop( \tortjacobi(x,y,z), y),
  \end{align*}
  since due to antisymmetry
  \begin{equation*}
   \tortjacobi(y,x,y)=\tortjacobi(y,z,y)=0.
  \end{equation*}
  Since the characteristic is different from two, we can divide by two, and thus again arrive at \eqref{eq:tortkara}.

 \end{enumerate}

\end{proof}

In \cite[Section 6]{bib:DIM2018} it is shown
that in $d=2$, $(\areatortkara,\area)$ is the free tortkara algebra.%
\footnote{
Recall the definition of $\areatortkara$ from the introduction: the smallest linear space containing the letters $\word1,..\word{d}$
and being closed under the $\area$ operation.}

This linear space has a surprisingly simple description.
The following is \cite[Theorem 2.1]{bib:DIM2018} (see also \cite[Section 3.2, Theorem 31]{bib:Rei2018} for another proof in $d=2$).
\begin{lemma}
  \label{lem:AequalsA}
  \begin{align}
    \label{eq:AequalsA}
    \areatortkara = \spann_\R \{ \word{i} : \word{i} \text{ a letter }\} \oplus \spann_\R \{ w (\word{ij} - \word{ji}) : w \text{ a word}, \word{i},\word{j} \text{ letters} \}.
  \end{align}
\end{lemma}

\begin{example}
 We have that \cite[Equation (4)]{bib:DR2018}
 \begin{equation*}
  \Inv_n:=\sum_{\sigma\in S_n}\sign(\sigma)\,\sigma(\word{1})\cdots\sigma(\word{i}),
 \end{equation*}
 where we interpret $\sigma$ as a permutation of the letters, is in $\areatortkara$ for $d\geq n\geq 2$ by Lemma \ref{lem:AequalsA}, an element which plays an important role as the lowest order $\mathsf{SL}$ invariant component of the signature in dimension $d=n$, see \cite[Section 3.3]{bib:DR2018}, and can be interpreted as the $d=n$ dimensional signed volume of the path underlying the signature. 
 In particular, we recover
 \begin{align*}
  \Inv_2&=\area(\word{1},\word{2}),\\
  \Inv_3&=\vol(\word{1},\word{2},\word{3}),
 \end{align*}
 and in fact this can be generalized by defining the multilinear map
 \begin{equation*}
  \vol^n:T^{\geq 1}(\R^d)^n\to T^{\geq 1}(\R^d)
 \end{equation*}
 such that $\vol^n(a_1,\ldots,a_n)$ is the image of $\Inv_n$ under the unique Zinbiel homomorphism (unique due to freeness of the halfshuffle algebra as a Zinbiel algebra) that maps $\word{i}\mapsto a_i$ for $\word{i}=\word{1},\ldots,\word{n}$.
 Written out, this means
 \begin{equation*}
  \vol^n(a_1,\ldots,a_n)=\sum_{\sigma\in S_n}\sign(\sigma)\,(((a_{\sigma(1)}\hs a_{\sigma(2)})\hs a_{\sigma(3)})\hs\cdots)\hs a_{\sigma(n)}.
 \end{equation*}

 Through the fact that $\Inv_n\in\areatortkara$ for $d=n$, it is immediate that we obtain the restriction
 \begin{equation*}
  \vol^n:\areatortkara^n\to\areatortkara
 \end{equation*}
 for any $d\geq 2$.
\end{example}

We note the following conjecture, which was shown to hold true in the case $d=2$ in \cite[Section 6]{bib:DIM2018}
as well as in \cite[Section 3.2, Theorem 31]{bib:Rei2018}. The case $d\ge 3$ is still open.
\begin{conjecture}\label{conj:arealb}
  $\areatortkara$ is linearly generated by strict left-bracketings of the $\area$ operation.
  In particular, a linear basis for $\areatortkara$ (without the single letters)
  is given by
  \begin{align*}
    \area( \area( \area(\word{i}_1,\word{i}_2), \word{i}_3), .., \word{i}_n), \qquad n \ge 2, \word{i}_1, \dots,\word{i}_n \in \{\word1, \dots, \word{d} \}, \word{i}_1 < \word{i}_2.
  \end{align*}
\end{conjecture}

\begin{example}
  For example, with $d=2$, the tensor $\word{12}(\word{12}-\word{21})$, which is in $\areatortkara$, can be written as
  \begin{align*}
    \word{12}(\word{12}-\word{21})=\tfrac{1}{6}[2\,\area(\area(\area(\word1,\word2),\word1),\word2)-\area(\area(\area(\word1,\word2),\word2),\word1)].
  \end{align*}
\end{example}

It turns out that for a bilinear, antisymmetric operation,
showing that all bracketings
can be
rewritten as linear combination of left-bracketings
reduces to showing that this is possible for a small subset of bracketings.
We have not been able to show that this subset of bracketings can be rewritten,
but want to record this general fact nonetheless.
We formulate the statement imprecisely here,
and leave the exact statement Proposition \ref{prop:specialtrees} and its proof to the appendix.
\begin{proposition}
  \label{prop:imprecisely}
  To be able to rewrite
  any bracketing to a linear combination of left-brackets,
  it is enough to verify this for
  bracketings of the form
  \begin{align*}
    \area(\area(\dots,\area(\area(a_1,a_2),a_3),\dots, a_{n-2}), \area(a_{n-1},a_n)),
    \quad a_i\in\areatortkara.
  \end{align*}
\end{proposition}




While trying to find a proof for Conjecture \ref{conj:arealb} for $d\ge 3$, we investigated in detail the operator $\overleftarrow{\area}$ given by the following definition.  

\begin{definition}
If $w=\word{l}_1\dots \word{l}_n$ is a word, we define $\arealb{w}$ to be the left-bracketing expression \index{areal@$\arealb$}
\[\area(\dots\area(\area(\area(\word{l}_1,\word{l}_2),\word{l}_3),\word{l}_4),\dots,\word{l}_n).\]
This is expanded linearly to an operation on the tensor algebra with $\arealb{e}=0$ and $\arealb{\word{l}}=\word{l}$ for any letter $\word{l}$. 
\end{definition}

We came across some interesting properties. 

First, we show that there is an expansion formula for $\arealb{w}$ in terms of
permutations of the letters in the word $w$. To this end, define the
right action of a permutation $\sigma\in S_n$ on words of length $n$ as
\begin{align*}
   \sigma := \word{l}_{\sigma(1)} .. \word{l}_{\sigma(n)},
\end{align*}
where $w=\word{l}_1\cdots\word{l}_n$.

\begin{proposition}
  \label{lem:groupAlgebra1}
 We have
 \begin{equation*}
  \arealb{\word{l}_1\cdots \word{l}_n}
  =
  \word{l}_1\cdots \word{l}_n\ \theta_n,
 \end{equation*}
 where
 \begin{align*}
   \theta_n
   :=
   \sum_{\sigma\in S_n}\, f_n(\sigma) \sigma
 \end{align*}
 and $f_n:\, S_n\to\{-1,1\}$ is given as
 \begin{equation*}
  f_n(\sigma)=\prod_{i=1}^n g_i(\sigma)
 \end{equation*}
 with
 \begin{equation*}
  g_i(\sigma)=
  \begin{cases}
   +1, &\text{if $\sigma^{-1}(j)<\sigma^{-1}(i)$ for all $j\in\mathbb{N}$ with $j<i$,}\\
   -1, &\text{else.}
  \end{cases}
 \end{equation*}
\end{proposition}

\begin{proof}
 For $n=1$, there is only the identity permutation and $f_1(\id)=g_1(\id)=1$, thus the statement is obviously true.
 For $n=2$, we have $S_2=\{\id,(12)\}$, $f_2(\id)=-f_2((12))=1$ and
 \begin{equation*}
  \area(\word{l}_1 \word{l}_2)=\word{l}_1 \word{l}_2-\word{l}_2 \word{l}_1=f_2(\id)\word{l}_1 \word{l}_2+f_2((12))\word{l}_2 \word{l}_1.
 \end{equation*}
 Assume the statement holds for some $n\in\mathbb{N}\setminus\{1\}$. Then,
 \begin{align*}
  &\arealb{\word{l}_1\cdots \word{l}_{n+1}}=\area(\arealb{\word{l}_1\cdots \word{l}_n},\word{l}_{n+1})\\
  &=\sum_{\sigma\in S_n}\,f_n(\sigma)\,\word{l}_{\sigma(1)}\cdots \word{l}_{\sigma(n)}\word{l}_{n+1}-\sum_{\sigma\in S_n}\, f_n(\sigma)\,\word{l}_{n+1} \hs(\word{l}_{\sigma(1)}\cdots \word{l}_{\sigma(n)})\\
  &=\sum_{\substack{\tilde\sigma\in S_{n+1}:\\\,g_{n+1}(\tilde\sigma)=1}}\,f_{n+1}(\tilde\sigma)\,\word{l}_{\tilde\sigma(1)}\cdots \word{l}_{\tilde\sigma(n+1)}
  -\sum_{\sigma\in S_n}\, f_n(\sigma)\,\big(\word{l}_{n+1}\shuffle(\word{l}_{\sigma(1)}\cdots \word{l}_{\sigma(n-1)})\big) \word{l}_{\sigma(n)}\\
  &=\sum_{\substack{\tilde\sigma\in S_{n+1}:\\\,g_{n+1}(\tilde\sigma)=1}}\,f_{n+1}(\tilde\sigma)\,\word{l}_{\tilde\sigma(1)}\cdots \word{l}_{\tilde\sigma(n+1)}
  +\sum_{\substack{\tilde\sigma\in S_{n+1}:\\g_{n+1}(\tilde\sigma)=-1}}\, f_{n+1}(\tilde\sigma)\,\word{l}_{\tilde\sigma(1)}\cdots \word{l}_{\tilde\sigma(n+1)}\\
  &=\sum_{\tilde\sigma\in S_{n+1}}\, f_{n+1}(\tilde\sigma)\,\word{l}_{\tilde\sigma(1)}\cdots \word{l}_{\tilde\sigma(n)}.\qedhere
 \end{align*}
\end{proof}

We furthermore have the following surprising identity.
\begin{proposition}
\label{prop:arealblie}
 For all integers $n>2$, we have
 \begin{equation}\label{eq:arealblie}
  \arealb{\word{1}\,l(\word{2}\cdots\word{n})}=\word{1}\,\arealb{l(\word{2}\cdots\word{n})},
 \end{equation}
 where $l$ is the left Lie bracketing. This implies (by freeness of the half-shuffle algebra)
 \begin{equation*}
  \arealb{v\,l(w)}=\arealb{v}\,\arealb{l(w)}
 \end{equation*}
 for words $v,w$ such that $|w|\geq 2$.
\end{proposition}

\begin{remark}
 Note that due to the well known fact that strict left Lie bracketings linearly generate the free Lie algebra, more generally formulated, it holds that
 \begin{equation*}
  \arealb{v x}=\arealb{v}\,\arealb{x}
 \end{equation*}
 for any $v\in\TS$ and any Lie polynomial $x$
 with $\langle x, \word{i} \rangle = 0$ for all letters $\word{i}$.
\end{remark}

\begin{remark}
 In particular, we have
 \begin{equation*}
  \arealb{\word{l}_1\cdots \word{l}_n \word{l}_{n+1}}-\arealb{\word{l}_1\cdots \word{l}_{n+1} \word{l}_n}=2\ \arealb{\word{l}_1\cdots \word{l}_{n-1}}(\word{l}_n\word{l}_{n+1}-\word{l}_{n+1}\word{l}_n)
 \end{equation*}
for any letters $\word{l}_1,\dots,\word{l}_{n+1}$.
\end{remark}
\begin{proof}
 For the base case, we compute
  \begin{equation*}
      \arealb{\word{123}-\word{132}}
      = 2 (\word{123}-\word{132})
      =\word{1} \arealb{\word{23}-\word{32}}.
    \end{equation*}

 Assume $\eqref{eq:arealblie}$ holds for some integer $n>2$ and let $w$ be a word of length $n-1$. Then, for any letter $\word{i}$,
 \begin{align*}
  \arealb{\word{1}\,l(w\word{i})}=&\arealb{\word{1}l(w)\word{i}-\word{1}\word{i}l(w)}=\arealb{\word{1}l(w)\word{i}}-\arealb{\word{1}\word{i}}\,\arealb{l(w)}\\
  =&\arealb{\word{1}l(w)}\word{i}-\word{i}\hs\arealb{\word{1}l(w)}-\arealb{\word{1}\word{i}}\,\arealb{l(w)}\\
  =&\word{1}\,\arealb{l(w)}\word{i}-\word{i}\hs(\word{1}\,\arealb{l(w)})-\word{1}\word{i}\,\arealb{l(w)}+\word{i}\word{1}\,\arealb{l(w)}\\
  =&\word{1}\,\arealb{l(w)}\word{i}-\word{1}(\word{i}\hs\arealb{l(w)})-\word{1}\word{i}\,\arealb{l(w)}\\
  =&\word{1}\,\arealb{l(w)\word{i}}-\word{1}\word{i}\,\arealb{l(w)}\\
  =&\word{1}\,\arealb{l(w)\word{i}}-\word{1}\,\arealb{\word{i}l(w)}\\
  =&\word{1}\,\arealb{l(w\word{i})},
 \end{align*}
 where we used that $\word{i}\hs(\word{1}\,\arealb{l(w)})=\word{i}\word{1}\,\arealb{l(w)}+\word{1}(\word{i}\hs\arealb{l(w)})$ due to the combinatorial expansion formula for the half-shuffle.
\end{proof}

Interestingly, $\rho$ admits a permutation expansion which is quite similar to that of $\arealb$, in fact again via $f_n$, just that now only a subset $T_n$ of all permutations $S_n$ is involved.
\begin{proposition}
\label{prop:rhogroupalgebra}
 We have
 \begin{equation*}
  \rho(\word{l}_1\cdots \word{l}_n)
  =
  \word{l}_1\cdots \word{l}_n\ \vartheta_n,
 \end{equation*}
 where
 \begin{align*}
   \vartheta_n
   :=
   \sum_{\sigma\in T_n}\, f_n(\sigma) \sigma,
 \end{align*}
 $f_n:\, S_n\to\{-1,1\}$ is as in Lemma \ref{lem:groupAlgebra1} and $T_n$ is the set of all $\sigma\in S_n$ such that $\{\sigma(i),\ldots,\sigma(n)\}$ is an interval of integers (a set of the form $[a,b]\cap \mathbb{N}$)
 for all $i\in\{1,\ldots,n-1\}$.
\end{proposition}
\begin{proof}
 For $n=1$, there is only the identity permutation and $f_1(\id)=g_1(\id)=1$, thus the statement is obviously true.
 For $n=2$, we have $T_2=\{\id,(12)\}$, $f_2(\id)=-f_2((12))=1$ and
 \begin{equation*}
  \rho(\word{l}_1 \word{l}_2)=\word{l}_1 \word{l}_2-\word{l}_2 \word{l}_1=f_2(\id)\word{l}_1 \word{l}_2+f_2((12))\word{l}_2 \word{l}_1.
 \end{equation*}
 Assume the statement holds for some $n\in\mathbb{N}\setminus\{1\}$. 
 Then, using the recursive definition of $\rho$ from Equation \eqref{eq:rhorecursive},
 \begin{align*}
  &\rho(\word{l}_1\cdots \word{l}_{n+1})=\word{l}_1\rho(\word{l}_2\cdots \word{l}_{n+1})-\word{l}_{n+1}\rho(\word{l}_1\cdots \word{l}_n)\\
  &=\sum_{\sigma\in T_n}\,f_n(\sigma)\,\word{l}_1 l'_{\sigma(1)}\cdots l'_{\sigma(n)}-\sum_{\sigma\in T_n}\, f_n(\sigma)\,\word{l}_{n+1}\word{l}_{\sigma(1)}\cdots \word{l}_{\sigma(n)}\\
  &=\sum_{\substack{\tilde\sigma\in T_{n+1}:\\\tilde\sigma(1)=1}}\,f_{n+1}(\tilde\sigma)\,\word{l}_{\tilde\sigma(1)}\cdots \word{l}_{\tilde\sigma(n+1)}
  +\sum_{\substack{\tilde\sigma\in T_n:\\\tilde\sigma(1)=n+1}}\, f_{n+1}(\tilde\sigma)\,\word{l}_{\tilde\sigma(1)}\cdots \word{l}_{\tilde\sigma(n+1)}\\
  &=\sum_{\tilde\sigma\in T_{n+1}}\, f_{n+1}(\tilde\sigma)\,\word{l}_{\tilde\sigma(1)}\cdots \word{l}_{\tilde\sigma(n)},
 \end{align*}
 where $l'_i=\word{l}_{i+1}$.
\end{proof}

Via the recursive formula for $\rho$, we also get an alternative proof of the following.

\begin{corollary}
 \label{cor:rhoimageareas}
 $\Im\rho\subset\areatortkara$
\end{corollary}
\begin{proof}
 It suffices to show $\rho(w)\in\areatortkara$ for any word $w$. We have $\rho(\emptyWord)=0$, $\rho(\word{i})=\word{i}\in\areatortkara$ and $\rho(\word{i}\word{j})=\word{i}\word{j}-\word{j}\word{i}\in\areatortkara$ for any letters $\word{i},\word{j}$. Let $n\geq 2$ and assume that $\rho(w)\in\areatortkara$ for any word $w$ with $|w|=n$. Then, for any word $v$ with $|v|=n-1$ and any letters $\word{i},\word{j}$, we have
 \begin{equation*}
  \rho(\word{i}v\word{j})=\word{i}\rho(v\word{j})-\word{j}\rho(\word{i}v)\in\areatortkara
 \end{equation*}
 since $\rho(v\word{j}),\rho(\word{i}v)\in\areatortkara$ with $|\rho(v\word{j})|=|\rho(\word{i}v)|=n\geq 2$ due to the induction hypothesis, and by Lemma \ref{lem:AequalsA}, the non-letter part of $\areatortkara$ is stable under concatenation of any element of the tensor algebra from the left. Thus, the induction hypothesis also holds for all words of length $n+1$.
\end{proof}
\section{Conclusion}
\todonotes{RP/JD/JR: Extend conclusion}
We have linked the area operation in the tensor algebra to work in control theory and more abstract work on Tortkara algebras.
We have shown that starting from letters and applying the area operation, one obtains enough elements to shuffle-generate the tensor algebra.

There are many open directions for research.
We have not identified a minimal set of areas-of-areas which is just enough to shuffle-generate the tensor algebra -- i.e.~to shuffle-generate it exactly.
The linear span of the areas-of-areas has been identified, but a basis for it in terms of areas-of-areas has not.

\subsection{Open combinatorial problems}
\begin{enumerate}
 \item What is $\spann\{\area(\word{i}_1\shuffle\cdots\shuffle\word{i}_n,\word{j}_1\shuffle\cdots\shuffle\word{j}_m),\,n,m\in\mathbb{N},\,\word{i}_1,\ldots\word{i}_n,\word{j}_1,\ldots\word{j}_m\,\,\text{letters}\}$? Does it shuffle generate together with the letters, and if not what is the smallest subalgebra of the associative shuffle algebra containing it and the letters? This is the algebraic formulation of the question \enquote{what do we know about a path if we are only allowed to collect its increment and the values of the first area of any two dimensional polynomial image of the path}?
 \item Give linear bases for $\areatortkara\shuffle\areatortkara$, $\areatortkara\shuffle\areatortkara\shuffle\areatortkara$, $\ldots$. Does $\sum_{m=1}^{n}\areatortkara^{\shuffle m}$ already arrive at $\TS$ for a finite $n$?
 \item Give a minimal generating set for $T(\R^d)$ as a Tortkara algebra. Is it free?
 \item In light of Proposition \ref{prop:arealblie}, look at $\gen{x\in\mathfrak{g}:\,\langle x,\word{i}\rangle=0\,\forall\,\word{i}}{\cdot}$ and its image under $\arealb$ 
 \item What are the eigenspaces of $\arealb$?
\end{enumerate}

\appendix
\section{Appendix}


\begin{lemma}
  \label{lem:triangularGenerating}
  Let $V = \bigoplus_{n\ge 1} V_n$ be a graded vector space, each $V_n$ finite dimensional, and denote the grading $|.|_V$.

  Consider $\R[V]$, the symmetric algebra over $V$ (see Section \ref{sec:shuffleGenerators}), with two different gradings, defined on monomials as follows
  \begin{itemize}
    \item $|x^m|_{\operatorname{deg}} := m$ (denote the corresponding projection onto degree $1$ by $\proj_1^{\operatorname{deg}}$).
    \item $|x^m|_{\operatorname{weight}} := m\cdot |x|_V$.
  \end{itemize}
  Let $Y \subset \R[V]$, countable, be such that every $y \in Y$ is homogeneous with respect to $|.|_{\operatorname{weight}}$.
  Then:
  \begin{center}
    $Y$ generates $\R[V]$ (as commutative algebra) \\

    \vspace{0.5em}
    \emph{if and only if} \\
    \vspace{0.5em}
    
    $\operatorname{span}_\R \proj_1^{\operatorname{deg}} Y = V$
  \end{center}

  If moreover $\proj_1^{\operatorname{deg}} y$, $y\in Y$, are linearly independent,
  then $Y$ freely generates $\R[V]$ (as commutative algebra).
\end{lemma}
\begin{proof}
  We show the first statement.

  $\Rightarrow$:
  Assume $v \in V \subset \R[V]$ is not in the span of
  $\proj_1^{\operatorname{deg}} Y$.
  Then it is clearly not in the algebra generated by $Y$. Hence $Y$ does not generate $\R[V]$.
  This proves the contrapositive.

  $\Leftarrow$:
  Denote, local to this proof, by $\langle M \rangle$ the subalgebra generated by $M\subset \R[V]$.

  Claim: $\langle V_1 \rangle \subset \langle Y \rangle$.
  Indeed, $v \in V_1$ can, by assumption
  be written as linear combination of some
  \begin{align*}
    \proj_1^{\operatorname{deg}} y_i,
  \end{align*}
  where $y_i \in Y$. Since the $y_i$ are homogeneous they must be of weight $1$.
  Hence $\proj_1^{\operatorname{deg}} y_i = y_i$,
  hence $V_1 \subset \langle Y \rangle$,
  hence $\langle V_1 \rangle \subset \langle Y \rangle$, which proves the claim.

  Now let $\langle V_1 \oplus \dots \oplus V_n \rangle \subset \langle Y \rangle$.
  Claim: $V_{n+1} \subset \langle Y \rangle$.
  Indeed, $v \in V_{n+1}$ can be written as linear combination of some
  \begin{align*}
    \proj_1^{\operatorname{deg}} y_i,
  \end{align*}
  where $y_i \in Y$, of weight $n+1$. Then
  \begin{align*}
    y_i = \proj_1^{\operatorname{deg}} y_i + r_i,
  \end{align*}
  with $r_i$ monomials (of order $2$ an higher)
  in terms from $V_1 \oplus \dots \oplus V_n$, i.e.~$r_i \in \langle V_1 \oplus \dots \oplus V_n \rangle \subset \langle Y \rangle$.
  Hence $v \in \langle Y \rangle$.

  Hence $\langle V_1 \oplus \dots \oplus V_n \oplus V_{n+1} \rangle \subset \langle Y \rangle$.
  Iterating, we see that $\R[ V ] = \langle V \rangle \subset \langle Y \rangle$, which proves the first claim.
\end{proof}

We finally give a precise statement and proof of \Cref{prop:imprecisely}



Let $V$ be an $\R$-vector space
and let
\begin{align*}
  \mathfrak B : V \times V \to V,
\end{align*}
be a bilinear map.%
\footnote{This section would be most comfortably be formulated
in the language of operads. But this would require more mathematical setup, which we want to avoid.}
We encode bracketings as planar trees.
%
Define the complete left-bracketed tree with $n$ leaves as
\newcommand\leftBracketTree{\mathsf{LeftBracketTree}}
\begin{align*}
  \leftBracketTree_1 &:= \Forest{[]} \\
  \leftBracketTree_n &:=
  \leftBracketTree_{n-1}
  \rightarrow_\bullet 
  \Forest{[]},  \qquad n \ge 2,
\end{align*}
her $\rightarrow_\bullet$ denotes grafting to a new root.

Define
\newcommand\specialTree{\mathsf{SpecialTree}}
\begin{align*}
  \specialTree_n :=
  \leftBracketTree_{n-2} \rightarrow_\bullet \leftBracketTree_2.
\end{align*}

For example
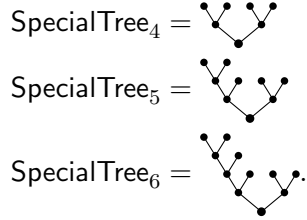
\begin{align*}
  \specialTree_4 &= \Forest{[ [[],[]], [[],[]] ]} \\
  \specialTree_5 &= \Forest{[ [[ [], [] ],[]], [[],[]] ]} \\
  \specialTree_6 &= \Forest{[ [[ [ [], [] ], [] ],[]], [[],[]] ]}.
\end{align*}

For any tree $\tau$ with $n$ leaves, and $a_1, .., a_n \in V$ write
\begin{align*}
  \tau( a_1, .., a_n ),
\end{align*}
as the corresponding bracketing.
%
We extend this definition
to the case where
(some of) the $a_i$ are planar trees (with labeled leaves) themselves, by just replacing the respective leaf of $\tau$ with $a_i$.
(This is consistent, when considering $a \in V$ as the tree with exactly on vertex, labeled $a$.)


On every new level $n+1$, it is enough to check that $\specialTree_{n+1}$
can be expressed in terms of left brackets:
\begin{proposition}\label{prop:specialtrees}
  Assume that $\mathfrak B$ is symmetric or anti-symmetric.

  Assume, for some $n$, that all trees $\tau$ with $\leaves{\tau} \le n$ can be expressed
  in terms of left brackets, i.e.
  for some $c(\tau,\sigma) \in \R$,
  \begin{align*}
    \tau(a_1,..,a_n) = \sum_{\sigma \in S_n} c(\tau,\sigma) \leftBracketTree_n(a_{\sigma(1)},\dots,a_{\sigma(n)}). 
    \quad \forall a_1,..,a_n \in V.
  \end{align*}
  Assume that (every labeling of) $\specialTree_{n+1}$ can be expressed in terms of left brackets. Then:
  \begin{center}
  (every labeling of) every tree $\sigma$ with $\leaves{\sigma} = n + 1$ can be expressed\newline
  in terms of left brackets.
  \end{center}

\end{proposition}
\begin{proof}
  Consider
  \begin{align*}
    \tau = \Forest{ [ [$T_1$, for tree={fill=none, child anchor = south}], [$T_2$, for tree={fill=none, child anchor = south}] ] ] },
  \end{align*}
  with
  \begin{align*}
    T_1 &= \tau_1(a_1,..,a_m) \\
    T_2 &= \tau_2(a_{m+1},..,a_{m+\ell})
  \end{align*}
  with $\leaves{\tau} = n+1 = m + \ell$.
  By using symmetry/antisymmetry, we can assume $\leaves{\tau_1} = m \ge \leaves{\tau_2} = \ell$.

  By assusmption, we can write both $\tau_1$ and $\tau_2$ in terms of left-bracketings.
  It is hence enough to consider
  \begin{align*}
    \tau_1 = \leftBracketTree_m \\
    \tau_2 = \leftBracketTree_\ell,
  \end{align*}
  with $m + \ell = n+1$ and $m \ge \ell$.

  Claim: we can reduce to $\ell=1$ and $\ell=2$.

  Indeed, write
  \begin{align*}
    \tau_1 &= ((t_1, t_2), .., t_{m-1}) \\
    \tau_2 &= (t_m, t_{m+1}),
  \end{align*}
  with
  \begin{align*}
    t_1 &= (a_1,a_2) \\
    t_2 &= a_3 \\
        &.. \\
    t_{m-1} &= a_m \\
    t_m &= ((a_{m+1}, a_{m+2}), .., a_{m+\ell-1}) \\
    t_{m+1} &= a_{m + \ell}.
  \end{align*}

  If $\ell \ge 2$ we have $m + 1 \le n$.
  Hence by assumption
  \begin{align*}
    \tau = \sum \text{ left bracketings }(t_1, .., t_{m+1}).
  \end{align*}
  Now, consider the rightmost spot in each leftbracketing.
  \begin{itemize}
  \item If it is taken by a letter: $\leadsto \ell = 1$.
  \item If it is taken by $t_1$: $\leadsto \ell = 2$.
  \item If it is taken by $t_m$: $\leaves{t_m} = \leaves{\tau_2} - 1= \ell - 1$.
      So we go from $\ell$ to $\ell-1$.
  \end{itemize}
  We can finish by induction.

\end{proof}

\printindex

\end{document}